\documentclass[reqno]{amsart}
\usepackage[backend=bibtex, bibstyle=alphabetic, citestyle=alphabetic, sorting=nyt, maxbibnames=199, giveninits=true, isbn=false, url=false, doi=false, eprint=false]{biblatex}  
\renewbibmacro{in:}{}
\bibliography{references.bib}
\usepackage{amssymb}
\usepackage{calrsfs}
\usepackage{graphics,graphicx}
\usepackage{enumerate}
\usepackage{enumitem}
\usepackage{url}
\usepackage{xcolor}
\usepackage[ruled, lined, linesnumbered, longend]{algorithm2e}
\usepackage{hyperref}
\usepackage[justification=centering]{caption}
\usepackage{comment}

\newtheorem{theorem}{Theorem}[section]
\newtheorem{lemma}[theorem]{Lemma}

\newtheorem{corollary}[theorem]{Corollary}
\numberwithin{equation}{section}

\theoremstyle{remark}
\newtheorem*{remark}{Remark}

\title{Explicit bounds on $\zeta(s)$ in the critical strip and a zero-free region}
\author{Andrew Yang}
\date\today
\subjclass[2010]{Primary: 11M06, 11M26; Secondary: 11Y35}
\keywords{Riemann zeta-function, exponential sums, van der Corput method, zero-free region.}

\begin{document}

\begin{abstract}
We derive explicit upper bounds for the Riemann zeta-function $\zeta(\sigma + it)$ on the lines $\sigma = 1 - k/(2^k - 2)$ for integer $k \ge 4$. This is used to show that the zeta-function has no zeroes in the region $$\sigma > 1 - \frac{\log\log|t|}{21.233\log|t|},\qquad |t| \ge 3.$$ This is the largest known zero-free region for $\exp(171) \le t \le \exp(5.3\cdot 10^{5})$. Our results rely on an explicit version of the van der Corput $A^nB$ process for bounding exponential sums.
\end{abstract}

\maketitle

\section{Introduction}
Bounding the size of the Riemann zeta-function $\zeta(s)$ within the critical strip $0 < \Re s < 1$ is a central goal in analytic number theory, with results having broad implications for the zeroes of $\zeta(s)$ and thus the distribution of prime numbers \cite{ford_zero_2002, kadiri_explicit_2018, trudgian_improved_2014}. By far the most common result of this type are estimates of $\zeta(\sigma + it)$ as $t \to\infty$ and $\sigma$ is fixed to either $1/2$ or $1$, i.e.\ bounds along the critical line and 1-line respectively. Currently, the best known unconditional results are $\zeta(1/2 + it) \ll_{\varepsilon} t^{13 / 84 + \varepsilon}$ for any $\varepsilon > 0$, due to Bourgain \cite{bourgain_decoupling_2016}, and $\zeta(1 + it) \ll \log^{2/3}t$, due to Vinogradov \cite{vinogradov_eine_1958}. Explicit results are also known, with the current best bound (for large $t$) being $|\zeta(1/2 + it)| \le 307.098 t^{27/164}$ ($t \ge 3$), proved by Patel \cite{patel_explicit_2021} and $|\zeta(1 + it)| \le 62.6\log^{2/3}t$ ($t \ge 3$) due to Trudgian \cite{trudgian_new_2014}. 

The main focus of this paper is to prove explicit bounds for $\zeta(\sigma + it)$ for special values of $\sigma \in (1/2, 1)$. For many applications of interest, such as zero-free regions and zero-density estimates, we require bounds holding uniformly in the strip $1/2 \le \sigma \le 1$. The Vinogradov--Korobov zero-free region, for instance, relies on a Ford--Richert type result $|\zeta(\sigma + it)| \le A t^{B(1 - \sigma)^{3/2}}\log^{2/3}t$ for all $1/2 \le \sigma \le 1$ \cite{cheng_explicit_2000, richert_zur_1967, ford_zero_2002}. Via a convexity argument (see e.g.\ \cite[\S 7.8]{titchmarsh_theory_1986}), we may use estimates of $\zeta(1/2 + it)$ and $\zeta(1 + it)$ to obtain bounds on $\zeta(\sigma + it)$ for any $1/2 \le \sigma \le 1$. However, through van der Corput's method of bounding exponential sums, we can directly derive asymptotically sharper bounds for specific values of $\sigma$. Titchmarsh \cite[Thm.\ 5.13]{titchmarsh_theory_1986} shows for instance that if $\sigma = 1 - k/(2^k - 2)$ for some integer $k \ge 4$, then
\begin{equation}\label{titchmarsh_ineq}
\zeta(\sigma + it) \ll t^{1/(2^k - 2)}\log t.
\end{equation}
This is sharper than what is achievable via convexity arguments for all $k \ge 4$.  

Having an explicit version of \eqref{titchmarsh_ineq} will allow us to improve many existing explicit results about $\zeta(s)$. The main obstacle to making \eqref{titchmarsh_ineq} explicit is the difficulty with bounding the implied constants of the $k$th derivative test, obtained through van der Corput's $A^{k - 2}B$ process. In this work we refine existing approaches to explicit exponential sum theory, to obtain an explicit $k$th derivative test with constants holding uniformly for all $k \ge 3$. This allows us to show the following theorem.

\begin{theorem}\label{theorem1}
Let $k \ge 4$ be an integer and $\sigma_k := 1 - k/(2^k - 2)$. Then 
\begin{equation}\label{theorem1_bound}
\left|\zeta\left(\sigma_k + it\right)\right| \le 1.546 t^{1/(2^k - 2)}\log t,\qquad t \ge 3.
\end{equation}
\end{theorem}
For example, substituting $k = 4$ gives $|\zeta(5/7 + it)| \le 1.546 t^{1/14}\log t$. By comparison, the sharpest bound that can currently be obtained using bounds on $\zeta(1/2 + it)$, $\zeta(1 + it)$ and the convexity principle is $\zeta(5/7 + it) \ll_{\varepsilon} t^{13/147 + \varepsilon}$, where $13/147 > 1/14$. Theorem \ref{theorem1} is sharpest for small to moderately sized $k$ and $t$. In particular, as $k \to \infty$, \eqref{theorem1_bound} reduces to $|\zeta(1 + it)| \le 1.546\log t$, which is weaker than other known bounds on the 1-line  \cite{backlund_uber_1916, patel_explicit_2022}. If we are only interested in large $k$ and $t$, then Theorem~\ref{theorem1} can be sharpened (see remarks in \S \ref{sec:concluding_remarks}). 


We briefly highlight some immediate applications of our results. The explicit $k$th derivative test is an ingredient in deriving an explicit version of Littlewood's bound $\zeta(1 + it) \ll \log t/\log\log t$ \cite[Thm.\ 5.16]{titchmarsh_theory_1986}. Similarly, the $k$th derivative test can be used to make explicit the bounds $1/\zeta(1 + it), \zeta'/\zeta(1 + it) \ll \log t / \log\log t$, which are useful for bounding Mertens' function $M(x)$ \cite{trudgian_explicit_2015, lee_explicit_2022}. Additionally, Theorem \ref{theorem1} can be used to improve explicit bounds on $S(t)$, the argument of the zeta-function along the critical line (see e.g.\ \cite{trudgian_improved_2014, hasanalizade_counting_2021}). In this work, we use Theorem \ref{theorem1} to prove an explicit version of Littlewood's \cite{littlewood_researches_1922} zero-free region of the form $1 - \sigma \ll \log\log t / \log t$.

\subsection{Littlewood's zero-free region} Zero-free regions for $\zeta(s)$ are widely studied partly due to their implications for prime distributions; some recent results include \cite{stechkin_zeros_1970, rosser_sharper_1975, kondratev_some_1977, cheng_explicit_2000, ford_zero_2002, kadiri_region_2005, jang_note_2014, mossinghoff_nonnegative_2014, mossinghoff_explicit_2022}. The current best explicit zero-free region for small $t$ is due to Mossinghoff, Trudgian and Yang \cite{mossinghoff_explicit_2022}, who proved that there are no zeroes of $\zeta(\sigma + it)$ in the region
\begin{equation}\label{classical_zfr}
\sigma > 1 - \frac{1}{5.558691\log|t|},\qquad |t| \ge 2.
\end{equation}
For intermediate $t$, the following zero-free region is currently the sharpest known
\begin{equation}\label{ford_classic_zfr}
\sigma > 1 - \frac{0.04962 - 0.0196/(J(|t|) + 1.15)}{J(|t|) + 0.685 + 0.155\log\log |t|},\qquad |t| \ge 3,
\end{equation}
where $J(t) := \frac{1}{6}\log t +\log\log t +\log 0.618$. This is formed by substituting \cite[Thm.\ 1.1]{hiary_improved_2022} into \cite[Thm.\ 3]{ford_zero_2002} and noting that $J(t) < \frac{1}{4}\log t + 1.8521$ for $t \ge 3$. 

For large $t$, the following Vinogradov--Korobov zero-free region due to \cite{mossinghoff_explicit_2022}, building on the method of Ford \cite{ford_zero_2002, ford_zero_2022}, is currently the sharpest known
\begin{equation}\label{vk_zfr}
\sigma > 1 - \frac{1}{55.241(\log|t|)^{2/3}(\log\log|t|)^{1/3}},\qquad |t| \ge 3.
\end{equation}
In this work we use Theorem \ref{theorem1} and the explicit $k$th derivative test to prove the following zero-free region.
\begin{corollary}\label{corollary1}
There are no zeroes of $\zeta(\sigma + it)$ in the region
\begin{equation}\label{littlewood_region}
\sigma > 1 - \frac{\log\log |t|}{21.233\log |t|},\qquad |t| \ge 3. 
\end{equation}
\end{corollary}
This represents the largest known zero-free region in the range $\exp(170.3) \le t \le \exp(532\,141)$. To summarise the current state of knowledge for other ranges of $t$: for $t \le 3\cdot 10^{12}$, all zeroes are known to lie on the critical line, due to the computational verification performed in \cite{platt_riemann_2021}. For $3\cdot 10^{12} < t \le\exp(46.2)$, \eqref{classical_zfr} is the sharpest known zero-free region; for $\exp(46.3)\le t \le \exp(170.2)$, \eqref{ford_classic_zfr} is sharpest; for $t \ge \exp(532\,142)$, \eqref{vk_zfr} is the sharpest. 

\subsection{Approach}
The main tool used to establish Theorem \ref{theorem1} are upper bounds on sums of the form 
\begin{equation}\label{S_defn}
S_f(a, N) := \left|\sum_{a < n \le a + N}e(f(n))\right|
\end{equation}
where $e(x) := \exp(2\pi i x)$ and $f$ is a sufficiently smooth function. In Titchmarsh \cite[Ch.\ V]{titchmarsh_theory_1986}, it was shown that if $f(x)$ has $k \ge 3$ continuous derivatives satisfying $0 < \lambda_k \le f^{(k)}(x) \le h\lambda_k$, then 
\begin{equation}\label{titchmarsh_thm_513}
S_f(a, N) \le A_1h^{1/2^{k - 2}}N\lambda_k^{1/(2^k - 2)} + A_2N^{1- 1/2^{k - 2}}\lambda_k^{-1/(2^k - 2)}
\end{equation}
for some unspecified absolute constants $A_1$ and $A_2$. This is also known as a $k$th derivative test, and the method of derivation was to use van der Corput's $A^{k - 2}B$ process, where a single application of Poisson summation is followed by $k - 2$ applications of the Weyl--van der Corput inequality. To our knowledge, to date the constants $A_1$ and $A_2$ in \eqref{titchmarsh_thm_513} have not been explicitly computed. However, Granville and Ramar\'e \cite[Prop.\ 8.2]{granville_explicit_1996} have proved an explicit bound of the form
\begin{equation}\label{granville_result}
S_f(a, N) \ll_{k,h} N\left(\lambda_k^{2^{k - 1}/(2^k - 2)} + N^{-1}\log^{k - 1}(\lambda_k^{-1})\right)^{1/2^{k - 1}}
\end{equation}
for $k \ge 2$. To facilitate a comparison, in our eventual application the value of $\lambda_k$ is such that  \eqref{titchmarsh_thm_513} and \eqref{granville_result} respectively reduce to bounds of the form
\[
S_f \ll N^{1 - 1/2^{k - 2}}\qquad\text{and}\qquad S_f \ll N^{1 - 1/2^{k - 2}}\log^{(k - 1)/2^{k - 1}} N.
\]
Therefore, \eqref{titchmarsh_thm_513} produces a log-power saving (see also remarks after Lemma \ref{kth_deriv_test}). 
More is known for small values of $k$; see for instance \cite[Thm.\ 6.9]{bordelles_arithmetic_2012} for $k = 2$ and Patel \cite{patel_explicit_2022} for $k = 3, 4, 5$. 
In \S\ref{sec:expsums} we derive a result of the form \eqref{titchmarsh_thm_513} and provide a comparison of our respective results. 

The main challenge to proving \eqref{titchmarsh_thm_513} with reasonable constants, is that $A_1$ and $A_2$ tend to grow rapidly when applying the bound on an ill-suited summation interval, either because the resulting sum is too short, or because the phase function $f(x)$ cannot be controlled properly on that domain. Meanwhile, an $A^{k - 2}B$ process involves $k - 2$ successive applications of Weyl-differencing, which for large $k$ limits our ability to isolate and properly address such pathological intervals. In our approach, we make progress through the repeated use of the trivial bound with each application of Weyl-differencing to avoid applying the $k$th derivative test in intervals for which it is poorly suited. 
 
\subsection{Structure of this paper} In \S \ref{sec:expsums} we review the van der Corput method and construct explicit $k$th-derivative tests corresponding to the $A^{k - 2}B(0, 1)$ exponent pair. The results of this section are agnostic to the choice of phase function $f(x)$. In \S \ref{sec:thm2_proof} we specialise to a specific phase function to bound $\zeta(s)$ on certain vertical lines inside the critical strip. Finally, in \S \ref{sec:cor1_proof} we use the results of the previous two sections to prove Corollary \ref{corollary1}.  

\section[An explicit k-th derivative test]{An explicit $k$th derivative test}\label{sec:expsums}
The primary tool we use to bound $\zeta(s)$ in the critical strip is an upper bound on the exponential sum $S_f(a, N)$ (defined in \eqref{S_defn}), where $f$ is a smooth function possessing at least $k \ge 1$ continuous derivatives.  
Four established methods of bounding $S_f(a, N)$ are the Weyl--Hardy--Littlewood method, the van der Corput method, Vinogradov's method and the Bombieri--Iwaniec method (for an exposition, see Titchmarsh \cite[Ch.\ V]{titchmarsh_theory_1986} and \cite{bombieri_on_1986}). Here, we review relevant aspects of the van der Corput method in an explicit context. First, we have the trivial bound 
\begin{equation}\label{trivial_bound}
S_f(a, N) \le N + 1
\end{equation}
arising from applying the triangle inequality and counting the maximum number of integers in $(a, a + N]$. If $N$ is an integer, we can improve this to $S_f(a, N) \le N$. The explicit Kuzmin--Landau lemma improves on the trivial bound if $f(x)$ is sufficiently well-behaved. Let $\|x\|$ denote the distance to the nearest integer to $x$. Suppose that $f(x)$ is a real-valued function with a monotonic and continuous derivative on $[a, a + N]$, satisfying $\|f'(x)\| \ge \lambda_1 > 0$. Then
\begin{equation}\label{kuzmin_landau}
S_f(a, N) \le \frac{2}{\pi\lambda_1}.
\end{equation}
Proofs of this result can be found in \cite{landau_ueber_1928}, \cite{hiary_explicit_2016}, \cite{patel_explicit_2022} and \cite{hiary_improved_2022}. See also \cite{corput_1921}, \cite{kuzmin_sur_1927}, \cite{herzog_sets_1949}, \cite[p. 91]{titchmarsh_theory_1986}, \cite[p. 7]{graham_kolesnik_1991} and the survey in \cite{dereyna_on_2020}. Note that the bound in \eqref{kuzmin_landau} does not depend on the length of the summation interval $(a, a + N]$. In \cite{karamata_sur_1950} and \cite{hiary_improved_2022}, the following generalisation was proved. If $f'$ is monotonic and continuous on $(a, a + N]$, and $l + \lambda_1 \le f'(x) \le l + 1 - \mu_1$ for some integer $l$ and $\lambda_1, \mu_1 > 0$, then 
\begin{equation}\label{generalised_kuzmin_landau}
S_f(a, N) \le \frac{\lambda_1^{-1} + \mu_1^{-1}}{\pi}.
\end{equation}
In practice, the conditions imposed on $f'(x)$ are rarely satisfied on the entire summation interval $(a, a + N]$. Instead, we divide the interval of summation into multiple subintervals and apply \eqref{generalised_kuzmin_landau} within some of the intervals, and the trivial bound \eqref{trivial_bound} in the remaining intervals. By appropriately choosing the locations of the subdivisions, we arrive at an explicit version of an inequality due to van der Corput (see also \cite[Thm.\ 6.9]{bordelles_arithmetic_2012} for a similar explicit result). 

\begin{lemma}[Second-derivative test]\label{second_deriv_test}
Suppose $f(x)$ is real-valued and twice continuous differentiable on $[a, a + N]$ for some integers $a, N$, with $f''(x)$ monotonic and satisfying
\[
\lambda_2 \le |f''(x)| \le h\lambda_2,\qquad x\in [a, a + N],
\]
for some $\lambda_2 > 0$ and $h > 1$. Then, 
\[
S_f(a, N) \le \frac{4}{\sqrt{\pi}}N h\lambda_2^{1/2} + Nh\lambda_2 + \frac{4}{\sqrt{\pi}}\lambda_2^{-1/2}.
\]
\end{lemma}
\begin{proof}
We follow closely the argument in \cite[Lem. 2.5]{hiary_improved_2022}. There are two notable differences. First, we implement a suggested refinement mentioned in the concluding remarks of \cite{hiary_improved_2022} to eliminate the constant term of $2 - 4\pi^{-1}$. This slightly sharpens the bound and simplifies the arguments that follow. Second, we generalise the argument to a broader class of functions. In doing so we incur a penalty of $h^{1/3}$ in the first two terms. We opt for this generalisation because in our eventual application (the $k$th derivative test) it becomes increasingly difficult to leverage the benefits of specialising $f(x)$ as $k$ grows. 

Consider first the case if $\lambda_2 > \pi / 16$. Then, using the trivial bound, since $a$ and $N$ are integers, and using $h > 1$,
\begin{equation}
S_f(a, N) \le N < \frac{4}{\sqrt{\pi}} N h\lambda_2^{1/2}
\end{equation}
so the desired result is true for all $\lambda_2 > \pi / 16$.\footnote{We can expand the range of $\lambda_2$ under consideration via the following argument, as remarked by Timothy S. Trudgian. For all $\lambda_2 > 1 + 4(2 - \sqrt{4 + \pi})/\pi = 0.1439\ldots$, we have $(4/\sqrt{\pi})N\lambda_2^{1/2} + N\lambda_2 > N$ so the desired result once again follows from the trivial bound, as $h > 1$. This allows us to assume a sharper upper bound on $\lambda_2$ in the subsequent argument.} In the remainder of the proof we will assume that $\lambda_2 \le \pi / 16$. 

The conditions imposed on $f(x)$ imply that either $f''(x) > 0$ on $[a, a + N]$ or $f''(x) < 0$. Without loss of generality, assume that $f''(x) > 0$, since we may replace $f$ with $-f$ without changing the value of $S_f(a, N)$. Due to the continuity of $f''$, there exists some $\xi \in [a, a + N]$ for which
\begin{equation}\label{f_diff_bound}
f'(a + N) - f'(a) = N f''(\xi) \le N h\lambda_2.
\end{equation}
Meanwhile, define
\begin{equation}
C_0 := \lfloor f'(a)\rfloor,\qquad C_k := \lfloor f'(a + N)\rfloor,
\end{equation}
and let $\{f'(a)\} = \varepsilon_1$ and $\{f'(b)\} = \varepsilon_2$, where
$\{x\}$ denotes the fractional part of $x$. Let $0 < \Delta < 1/2$ be a parameter to be chosen later, and let
\begin{align}
C_j &:= C_{j - 1} + 1,&& 1\le j\le k - 1,  \notag\\
x_j &:= \max\{(f')^{-1}(C_j - \Delta), a\},&& 1\le j \le k,\\
y_j &:= \min\{(f')^{-1}(C_j + \Delta), a + N\},&& 0\le j \le k.  \notag
\end{align}
By \eqref{f_diff_bound}, we have
\begin{equation}\label{k_bound}
k = C_k - C_0 = f'(a + N) - \varepsilon_2 - (f'(a) - \varepsilon_1 ) \le Nh\lambda_2 + \varepsilon_1 - \varepsilon_2,
\end{equation}
Furthermore, since both $f'$ and its inverse function $(f')^{-1}$ are increasing, for $1 \le j \le k$ we have
\begin{align}
y_j - x_j &= \min\{(f')^{-1}(C_j + \Delta), a + N\} - \max\{(f')^{-1}(C_j - \Delta), a\}\notag\\
&\le 2\Delta|((f')^{-1})'(\xi_j)|,
\end{align}
for some $\xi_j$ satisfying
\begin{equation}
\max\{C_j - \Delta, f'(a)\} \le \xi_j \le \min\{C_j + \Delta, f'(a + N)\}.
\end{equation}
Therefore, $\nu_j := (f')^{-1}(\xi_j) \in [a, a + N]$ and hence 
\begin{equation}\label{trivial_deriv_bound}
y_j - x_j \le 2\Delta|((f')^{-1})'(\xi_j)| = \frac{2\Delta}{|f''(\nu_j)|}\le 2\Delta\lambda_2^{-1}.
\end{equation}
Next, because $\Delta < 1/2$ by assumption, and since $(f')^{-1}$ is increasing, we have $a\le x_1 < y_1 < x_2 < y_2 < \cdots < x_k < y_k \le a + N$. First, by the trivial bound and \eqref{trivial_deriv_bound}, we have for $1 \le j \le k$,
\begin{equation}\label{second_deriv_test_bound_1}
S_f(x_j, y_j - x_j) \le y_j - x_j + 1 \le 2\Delta\lambda_2^{-1} + 1.
\end{equation}
Next, in intervals of the form $[y_j, x_{j + 1})$ for $1 \le j \le k - 1$, we have, by construction, $\|f'(x)\| \ge \Delta$. By Lemma \ref{kuzmin_landau}, for $1 \le j \le k - 1$, 
\begin{equation}\label{second_deriv_test_bound_2}
S_f(y_j, x_{j + 1} - y_j) \le \frac{2}{\pi \Delta}
\end{equation}
It remains to consider the boundary sums $S_f(a, x_1 - a)$ and $S_f(y_k, a + N - y_k)$. Here we use the same argument as \cite{hiary_improved_2022}. First, consider $S_f(a, x_1 - a)$. There are three cases. 

\subsubsection*{Case 1: $0 \le \varepsilon_1 < \Delta$. } 
Then, $y_0 = (f')^{-1}(\lfloor f'(a)\rfloor + \Delta) > (f')^{-1}(f'(a)) = a$, as $(f')^{-1}$ is increasing. We divide $(a, x_1]$ into $(a, y_0]$ and $(y_0, x_1]$. In $(a, y_0]$ we use the trivial bound in a similar fashion as \eqref{trivial_deriv_bound}. In $(y_0, x_1]$ we have $\|f'(n)\| \ge \Delta$, so we use the Kuzmin--Landau lemma \eqref{kuzmin_landau} to cover this subinterval. Together, we have 
\begin{equation}
S_f(a, x_1 - a) \le S_f(a, y_1 - a) + S_f(y_1, x_1 - y_1) \le (\Delta - \varepsilon_1)\lambda_2^{-1} + 1 + \frac{2}{\pi \Delta}.
\end{equation}
\subsubsection*{Case 2: $\Delta \le \varepsilon_1 \le 1 - \Delta$. }
We have $y_0 \le a \le x_1$ and $C_0 + \varepsilon_1 \le f'(n) \le C_0 + 1 - \Delta$ for all $n \in (a, x_1]$. By \eqref{generalised_kuzmin_landau}, we have 
\begin{equation}
S_f(a, x_1 - a) \le \frac{1}{\pi}\left(\frac{1}{\varepsilon_1} + \frac{1}{\Delta}\right). 
\end{equation}

\subsubsection*{Case 3: $1 - \Delta < \varepsilon_1 \le 1$. } Then $f'(a) = \lfloor f'(a)\rfloor + \varepsilon_1 > C_0 + 1 - \Delta = C_1 - \Delta$. This implies $(f')^{-1}(C_1 - \Delta) < a$. Hence $x_1 = a$ and thus $S_f(a, x_1 - a) = 0$ in this case.

Combining the three cases, we conclude that
\begin{equation}
S_f(a, x_1 - a) \le \frac{1}{\pi\Delta} + H_{\Delta}(\varepsilon_1),
\end{equation}
where 
\begin{equation}\label{h_def}
H_{\Delta}(\varepsilon) := \begin{cases}
\displaystyle
(\Delta - \varepsilon)\lambda_2^{-1} + 1 + \frac{1}{\pi\Delta} &\text{if } \varepsilon \in [0, \Delta)\\
\displaystyle
\frac{1}{\pi\varepsilon} &\text{if }\varepsilon \in [\Delta, 1 - \Delta]\\
\displaystyle
-\frac{1}{\pi\Delta} &\text{if } \varepsilon \in (1 - \Delta, 1]
\end{cases}
\end{equation}
Via a similar argument, we have 
\begin{equation}\label{Sk_bound}
S_f(y_k, a + N - y_k) \le \frac{1}{\pi\Delta} + H_{\Delta}(1 - \varepsilon_2).
\end{equation}
Combining \eqref{second_deriv_test_bound_1}, \eqref{second_deriv_test_bound_2} and \eqref{k_bound}, we obtain that $S_f(a, N)$ is majorised by
\begin{equation}\label{second_deriv_test_combined_bound}
\begin{split}
&S_f(a, x_1 - a) + \sum_{j = 1}^kS_f(x_j, y_j - x_j) + \sum_{j = 1}^{k - 1}S_f(y_j, x_{j + 1} - y_j) + S_f(y_k, a + N - y_k)\\
&\qquad\qquad\le \frac{1}{\pi\Delta} + H_{\Delta}(\varepsilon_1) + k\left(2\Delta\lambda_2^{-1} + 1\right) + (k - 1)\frac{2}{\pi\Delta} + \frac{1}{\pi\Delta} + H_{\Delta}(1 - \varepsilon_2)\\
&\qquad\qquad\le 2\left(Nh\lambda_2 + \varepsilon_1 - \varepsilon_2\right)\left(\frac{1}{\pi\Delta} + \Delta\lambda_2^{-1} + \frac{1}{2}\right) + H_{\Delta}(\varepsilon_1) + H_{\Delta}(1 - \varepsilon_2).
\end{split}
\end{equation}
We choose $\Delta = \Delta_0 := \sqrt{\lambda_2/\pi}$ to minimise the second factor, noting that the upper bound on $\lambda_2$ guarantees our choice satisfies the previous assumption that $\Delta < 1/2$. By \eqref{second_deriv_test_combined_bound}, we have 
\begin{equation}\label{after_choosing_delta}
S_f(a, N) \le \frac{4}{\sqrt{\pi}}N h\lambda_2^{1/2} + N h\lambda_2 + J(\varepsilon_1) + J(1 - \varepsilon_2)
\end{equation}
where
\begin{equation}\label{J_defn}
J(\varepsilon) = H_{\Delta_0}(\varepsilon)  + \left(\frac{4}{\pi\Delta_0} + 1\right)\left(\varepsilon - \frac{1}{2}\right).
\end{equation}
To complete the proof, it suffices to show that
\begin{equation}\label{Jbound}
J(\varepsilon) \le \frac{2}{\pi\Delta_0}, \qquad 0 \le \varepsilon < 1.
\end{equation}
We once again consider three cases.

\subsubsection*{Case 1: $\varepsilon \in \left[0, \Delta_0\right)$.}
From \eqref{h_def} and \eqref{J_defn}, and using $0 \le \varepsilon < \Delta_0$,
\begin{equation}
\begin{split}
J(\varepsilon) &< \frac{\Delta_0}{\lambda_2} + 1 + \frac{1}{\pi\Delta_0} + \left(\frac{4}{\pi\Delta_0} + 1\right)\left(\Delta_0 - \frac{1}{2}\right) = \Delta_0 + \frac{4}{\pi} + \frac{1}{2} < \frac{2}{\pi\Delta_0},
\end{split}
\end{equation}
where the last inequality holds since $0 < \Delta_0 \le 1/4$.

\subsubsection*{Case 2: $\varepsilon \in \left[\Delta_0, 1 - \Delta_0\right]$.} 
Then,
\begin{equation}
J(\varepsilon) = \frac{1}{\pi\varepsilon} + \left(\frac{4}{\pi\Delta_0} + 1\right)\varepsilon - \frac{1}{\pi\Delta_0} - \frac{1}{2}
\end{equation}
is a convex function, so $J(\varepsilon) \le \max\left\{J(\Delta_0), J(1 - \Delta_0)\right\}$. However, for any $0 < \Delta_0 \le 1/4$ we have
\begin{equation}\label{Hdeltabound}
J(\Delta_0) = \frac{4}{\pi} - \frac{1}{2} + \Delta_0 < \frac{2}{\pi\Delta_0}
\end{equation}
and
\begin{equation}\label{H1deltabound}
\begin{split}
J(1 - \Delta_0) &= \frac{1}{\pi(1 - \Delta_0)} + (1 - \Delta_0) + \frac{2}{\pi\Delta_0} - \frac{4}{\pi} - \frac{1}{2} < \frac{2}{\pi\Delta_0}.
\end{split}
\end{equation}

\subsubsection*{Case 3: $\varepsilon \in (1 - \Delta_0, 1)$.} From $\varepsilon < 1$ and $\Delta_0 < 1/4$,
\begin{equation}
J(\varepsilon) < -\frac{1}{\pi\Delta_0} + \frac{1}{2}\left(\frac{4}{\pi\Delta_0} + 1\right) < \frac{2}{\pi\Delta_0},
\end{equation}
as required. 
\end{proof}

\begin{remark}
A qualitatively similar result is historically obtained via Poisson summation (known as process $B$, see e.g.\ \cite[Ch.\ V]{titchmarsh_theory_1986}). In our treatment we bypass Poisson summation altogether to obtain more favourable constants, while still achieving the goal of shortening the lengths of the exponential sums under consideration. In van der Corput notation, Lemma \ref{second_deriv_test} corresponds to the $B(0, 1)$ exponent pair. 
\end{remark}

In our application it is convenient to have the second-derivative test to be of the same form as all higher derivative tests, which motivates the following lemma. This may be compared to \cite[Thm.\ 6.9]{bordelles_arithmetic_2012}, which has $A_2 = 4/\sqrt{\pi}$ and $B_2 = 8/\sqrt{\pi}$. 

\begin{lemma}\label{second_deriv_test2}
Let $f(x)$, $a$, $N$, $h$ and $\lambda_2$ satisfy the same conditions as Lemma \ref{second_deriv_test}. Then 
\[
S_f(a, N) \le A_2 N h\lambda_2^{1/2} + B_2\lambda_2^{-1/2},
\]
where 
\begin{equation}
A_2 := \frac{2 + \sqrt{4 + \pi}}{\sqrt{\pi}},\qquad B_2 := \frac{4}{\sqrt{\pi}}.
\end{equation}
\end{lemma}
\begin{proof}
Let $\lambda_0 := 1 + 4(2 - \sqrt{4 + \pi})/\pi = 0.1439\ldots$ be the unique solution to 
\begin{equation}\label{second_deriv_test_lambda_0}
\left(\frac{4}{\sqrt{\pi}} + \lambda_0^{1/2}\right)\lambda_0^{1/2} = 1.
\end{equation}
If $\lambda_2 \le \lambda_0$, then by Lemma \ref{second_deriv_test}, 
\[
S_f(a, N) \le  \frac{4}{\sqrt{\pi}}N h\lambda_2^{1/2} + Nh\lambda_2 + \frac{4}{\sqrt{\pi}}\lambda_2^{-1/2} \le \left(\frac{4}{\sqrt{\pi}} + \lambda_0^{1/2}\right)Nh\lambda_2^{1/2} + \frac{4}{\sqrt{\pi}}\lambda_2^{-1/2}.
\]
On the other hand if $\lambda_2 > \lambda_0$, then by \eqref{second_deriv_test_lambda_0} and the trivial bound,
\[
S_f(a, N) \le N < \left(\frac{4}{\sqrt{\pi}} + \lambda_0^{1/2}\right)Nh\lambda_0^{1/2} < \left(\frac{4}{\sqrt{\pi}} + \lambda_0^{1/2}\right)Nh\lambda_2^{1/2} + \frac{4}{\sqrt{\pi}}\lambda_2^{-1/2},
\]
hence the result follows in either case.
\end{proof}

To obtain higher-derivative tests, we use an explicit $A$ process, which makes use of the Weyl-differencing operation. Here, $S_f$ is expressed in terms of $S_g$ with $g(x) = f(x + r) - f(x)$ for some integer $r > 0$. Intuitively, if $f(x)$ is well-approximated by a degree $d$ polynomial on $(a, b]$, with $d > 0$, then we can expect that $g(x)$ is well-approximated by a degree $d - 1$ polynomial on $(a, b]$. We can achieve sharper bounds on $S_g$ since the lower order of $g(x)$ means it likely satisfies the conditions of \eqref{generalised_kuzmin_landau} over longer intervals, hence increasing the savings produced by the bound. 

\begin{lemma}[Explicit $A$ process]\label{weyl_differencing}
Let $f(x)$ be real-valued and defined on $(a, a + N]$, for some integers $a, N$. For all integers $q > 0$, we have 
\[
(S_f(a, N))^2 \le \left(N - 1 + q\right)\left(\frac{N}{q} + \frac{2}{q}\sum_{r = 1}^{q - 1}\left(1 - \frac{r}{q}\right)S_{g_r}(a, N - r)\right)
\]
where $g_r(x) := f(x + r) - f(x)$. 
\end{lemma}
\begin{proof}
See \cite[Lem. 5]{cheng_explicit_2004} and \cite[Lem. 2]{platt_improved_2015}. The first factor originally appears as $N + 1 + q$ in \cite{cheng_explicit_2004} (upon making the substitution $a \mapsto N + 1$, $N \mapsto L - 1$). As remarked in \cite{platt_improved_2015}, it is possible to reduce this factor to $N + q$ via a more careful bound. Here, we further decrease the factor by 1 by assuming that $a$ and $N$ are integers, as sums over $n \in (a, a + N]$ are equivalent to sums over $n\in [a + 1, a + N]$.
\end{proof}

By using Lemma \ref{second_deriv_test} to estimate $S_{g_r}(a, N - r)$, we obtain the following estimate, which is historically obtained by applying the Poisson summation formula then applying the $A$ process. The next lemma corresponds to an explicit $AB(0, 1)$ process.

\begin{lemma}[Explicit third derivative test]\label{third_deriv_test}
Let $f(x)$ have three continuous derivatives and suppose $f'''$ is monotonic and satisfies $\lambda_3 \le |f'''(x)| \le h\lambda_3$ for all $x \in (a, a + N]$, where $\lambda_3 > 0$, $h > 1$ and $a$, $N$ are integers. Then, for any $\eta_3 > 0$, we have 
\[
S_f(a, N) \le A_3(\eta_3, h)h^{1/2}N\lambda_3^{1/6} + B_3(\eta_3)N^{1/2}\lambda_3^{-1/6}
\]
where 
\[
A_3 := \sqrt{\frac{1}{\eta_3 h} + \frac{32}{15\sqrt{\pi}}\sqrt{\eta_3 + \lambda_0^{1/3}} + \frac{1}{3}\left(\eta_3 + \lambda_0^{1/3}\right)\lambda_0^{1/3}}\delta_3,\quad B_3 := \frac{\sqrt{32}}{\sqrt{3}\pi^{1/4}\eta_3^{1/4}}\delta_3,
\]
\[
\lambda_0 := \left(\frac{1}{\eta_3} + \frac{32\eta_3^{1/2}h}{15\sqrt{\pi}}\right)^{-3},\qquad \delta_3 := \sqrt{\frac{1}{2} + \frac{1}{2}\sqrt{1 + \frac{3}{8}\pi^{1/2}\eta_3^{3/2}}}.
\]
\end{lemma}
\begin{proof}
We will first consider the case if $\lambda_3 \ge \lambda_0$. Using the trivial bound, since $a$ and $N$ are integers, we have 
\begin{align}
S_f(a, N) \le N = \sqrt{\frac{1}{\eta_3} + \frac{32\eta_3^{1/2}h}{15\sqrt{\pi}}}N\lambda_0^{1/6} &\le \sqrt{\frac{1}{\eta_3h} + \frac{32\eta_3^{1/2}}{15\sqrt{\pi}}}h^{1/2}N\lambda_3^{1/6}    \notag\\
&< A_3h^{1/2}N\lambda_3^{1/6}, 
\end{align}
and hence the desired result follows if $\lambda_3 \ge \lambda_0$. In the last inequality we have used $\delta_3 > 1$ and $\lambda_0 > 0$. 

Next, consider the case of $\lambda_3 \le \lambda_1$, where
\begin{equation}
\lambda_1 := \left(\frac{\eta_3}{\kappa N}\right)^3,\qquad \kappa := \frac{1}{2}\sqrt{1 + \frac{3}{8}\pi^{1/2}\eta_3^{3/2}} - \frac{1}{2}.
\end{equation}
Note that $\kappa$ satisfies
\begin{equation}
\frac{32\kappa(1 + \kappa)}{3\pi^{1/2}\eta_3^{3/2}} = 1\qquad\text{and}\qquad \delta_3 = \sqrt{1 + \kappa}.
\end{equation}
We once again apply the trivial bound to obtain
\begin{align}
S_f(a, N) \le N &= (1 + \kappa)^{1/2}\cdot \frac{\sqrt{32}}{\sqrt{3}\pi^{1/4}\eta_3^{3/4}}N\kappa^{1/2}    \notag\\
&= \delta_3\frac{\sqrt{32}}{\sqrt{3}\pi^{1/4}\eta_3^{1/4}}N^{1/2}\lambda_1^{-1/6} \le B_3N^{1/2}\lambda_3^{-1/6},
\end{align}
hence the desired result follows in this case too. 

Suppose now that $\lambda_1 < \lambda_3 < \lambda_0$ (we may assume that $\lambda_1 < \lambda_0$, since otherwise the proof is complete). Let $g_r(x) = f(x + r) - f(x)$ as in Lemma \ref{weyl_differencing}. By the mean value theorem, since $f'''$ is continuous we have $g_r''(x) = f''(x + r) - f''(x) = rf'''(\xi)$ for some $\xi \in [x, x + r] \subseteq [a, a + N]$. Therefore, by the lemma's assumption we have
\begin{equation}
r\lambda_3 \le |g_r''(x)| \le h\cdot r\lambda_3,\qquad x\in [a, a + N - r]
\end{equation}
and hence we may apply Lemma \ref{second_deriv_test} with $f = g_r$ and $\lambda_2 = r\lambda_3$, since both $a$ and $a + N - r$ are integers. This gives
\begin{equation}\label{third_deriv_test_first_bound}
S_{g_r}(a, b - r) \le \frac{4}{\sqrt{\pi}}(N - r)h(r\lambda_3)^{1/2} + (N - r)h(r\lambda_3) + \frac{4}{\sqrt{\pi}}(r\lambda_3)^{-1/2}.
\end{equation}
We apply the inequality 
\begin{equation}\label{dhir_inequality}
\sum_{r = 1}^q\left(1 - \frac{r}{q}\right)r^{s} \le \frac{q^{1 + s}}{(1 + s)(2 + s)},\qquad -1 < s \le 1,
\end{equation}
(see e.g.\ \cite[(32)]{patel_explicit_2022} for $-1 < s < 1$ and \cite[(45)]{hiary_explicit_2016} for $s = 1$) to obtain, after using $N - r < N$,
\begin{equation}
\frac{2}{q}\sum_{r = 1}^{q - 1}\left(1 - \frac{r}{q}\right)S_{g_r}(a, N - r) \le \frac{32}{15\sqrt{\pi}}h N q^{1/2}\lambda_3^{1/2} + \frac{1}{3} h N q\lambda_3 + \frac{32}{3\sqrt{\pi}}q^{-1/2}\lambda_3^{-1/2}.
\end{equation}
We choose $q = \lceil \eta_3\lambda_3^{-1/3}\rceil$ for some $\eta_3 > 0$, so that $\eta_3\lambda_3^{-1/3} \le q \le \eta_3\lambda_3^{-1/3} + 1$ and the RHS is
\begin{align}
&\le \frac{32hN}{15\sqrt{\pi}}\sqrt{\eta_3\lambda_3^{-1/3} + 1}\lambda_3^{1/2} + \frac{hN}{3}(\eta_3\lambda_3^{-1/3} + 1)\lambda_3 + \frac{32}{3\sqrt{\pi}}\left(\eta_3\lambda_3^{-1/3}\right)^{-1/2}\lambda_3^{-1/2}\notag\\
&= hN\lambda_3^{1/3}\left(\frac{32}{15\sqrt{\pi}}\sqrt{\eta_3 + \lambda_3^{1/3}} + \frac{1}{3}\left(\eta_3 + \lambda_3^{1/3}\right)\lambda_3^{1/3}\right) + \frac{32}{3\sqrt{\pi\eta_3}}\lambda_3^{-1/3}.
\end{align}
Substituting this estimate into Lemma \ref{weyl_differencing},
\begin{align}
(S_f(a, N))^2 &\le \left(N + q - 1\right)\left(\frac{N}{q} + \frac{2}{q}\sum_{r = 1}^{q - 1}\left(1 - \frac{r}{q}\right)S_{g_r}(a, N - r)\right)\\
&\le \left(N + \eta_3\lambda_3^{-1/3}\right)\Bigg(\frac{N}{\eta_3\lambda_3^{-1/3}} + \frac{32}{3\sqrt{\pi\eta_3}}\lambda_3^{-1/3}    \notag\\
&\qquad\qquad + hN\lambda_3^{1/3}\left(\frac{32}{15\sqrt{\pi}}\sqrt{\eta_3 + \lambda_3^{1/3}}  + \frac{1}{3}\left(\eta_3 + \lambda_3^{1/3}\right)\lambda_3^{1/3}\right)\Bigg)\\
&< \left(1 + \frac{\eta_3\lambda_3^{-1/3}}{N}\right)\left(\left(\frac{A_3}{\delta_3}\right)^2 hN^2\lambda_3^{1/3} + \left(\frac{B_3}{\delta_3}\right)^2 N\lambda_3^{-1/3}\right),
\end{align}
where the last inequality follows from the assumption that $\lambda_3 < \lambda_0$. Applying the assumption $\lambda_3 > \lambda_1$ to the first factor, then taking square roots,
\begin{align}
S_f(a, N) &\le \sqrt{\left(1 + \frac{\eta_3\lambda_1^{-1/3}}{N}\right)\left(\left(\frac{A_3}{\delta_3}\right)^2 hN^2\lambda_3^{1/3} + \left(\frac{B_3}{\delta_3}\right)^2 N\lambda_3^{-1/3}\right)}\\
&\le \sqrt{1 + \kappa}\frac{A_3}{\delta_3} h^{1/2}N\lambda_3^{1/6} + \sqrt{1 + \kappa}\frac{B_3}{\delta_3} N^{1/2} \lambda_3^{-1/6},
\end{align}
since $\sqrt{x + y} \le \sqrt{x} + \sqrt{y}$ for all $x, y \ge 0$. However $\delta_3 = \sqrt{1 + \kappa}$ so the result follows. 
\end{proof}
If we now use Lemma \ref{third_deriv_test} instead of Lemma \ref{second_deriv_test} to bound $S_{g_r}(a, N - r)$ in Lemma \ref{weyl_differencing}, then we obtain the fourth-derivative test, via the $A^2B(0, 1)$ process. Performing this substitution recursively as necessary, it is possible to derive an explicit $A^{k - 2}B(0, 1)$ process, for any $k \ge 2$. The following lemma is our main result.

\begin{lemma}[Explicit $k$th derivative test]\label{kth_deriv_test} Let $a, N$ be integers with $N > 0$. Let $f(x)$ be equipped with $k \ge 4$ continuous derivatives, with $f^{(k)}$ monotonic, and suppose that $0 < \lambda_k \le |f^{(k)}(x)| \le h\lambda_k$ for all $x \in (a, a + N]$ and some $h > 1$. Then
\[
S_f(a, N) \le A_k h^{2/K}N\lambda_k^{1/(2K - 2)} + B_kN^{1 - 2/K}\lambda_k^{-1/(2K - 2)}
\]
where $K = 2^{k - 1}$, $A_3(\eta_3, h)$ and $B_3(\eta_3)$ are defined in Lemma \ref{third_deriv_test}, and $A_k, B_k$ for $k \ge 4$ are defined recursively via
\begin{equation}\label{s1_Ak_recursive_defn}
A_{j + 1}(\eta_3, h) := \delta_j\left(h^{-1/J} + \frac{2^{19/12}(J - 1)}{\sqrt{(2J - 1)(4J - 3)}}A_{j}(\eta_3, h)^{1/2}\right),
\end{equation}
\begin{equation}
B_{j + 1}(\eta_3) := \delta_j\frac{2^{3/2}(J - 1)}{\sqrt{(2J - 3)(4J - 5)}} B_{j}(\eta_3)^{1/2},
\end{equation}
\begin{equation}
\delta_j := \sqrt{1 + \frac{2}{2337^{1 - 2/J}}\left(\frac{9\pi}{1024}\eta_3\right)^{1/J}},
\end{equation}
where $J := 2^{j - 1}$.
\end{lemma}
\begin{proof}
The central argument remains the same as that of Lemma \ref{third_deriv_test}, however to achieve a bound that holds uniformly for all $k \ge 4$ without excessive tedium, we forsake sharpness in several key inequalities. One motivation for the separate treatment of the $k = 3$ case is to establish good starting constants $A_3$ and $B_3$ which feed into all higher derivative tests. An immediate avenue for further refinement is therefore to extend the argument of Lemma \ref{third_deriv_test} to higher $k$ before switching to the general argument.
 
As before, we begin by considering a few edge cases. First, suppose $N \le 2336$. Since $\delta_k > 1$, for all $k \ge 3$ we have 
\[
A_{k + 1}B_{k + 1} \ge \delta_k^2\frac{2^{37/12}(K - 1)^2}{\sqrt{(2K - 1) (2K-3)(4K-3)(4K-5)}}A_k^{1/2}B_k^{1/2} > 2^{1/12}A_k^{1/2}B_k^{1/2}
\]
which implies that 
\begin{equation}
A_kB_k > 2^{1/6-2/(3K)}(A_3B_3)^{4/K}.
\end{equation}
Therefore, from the arithmetic-geometric means inequality,
\begin{equation}
\begin{split}
&A_k h^{2/K}N\lambda_k^{1/(2K - 2)} + B_kN^{1 - 2/K}\lambda_k^{-1/(2K - 2)} \ge 2(A_kB_k)^{1/2}h^{1/K}N^{1 - 1/K}\\
&\qquad\qquad\qquad > 2^{13/12 - 1/(3K)}(A_3B_3)^{2/K}h^{1/K}N^{1 - 1/K}.
\end{split}
\end{equation}
The last expression is no smaller than $N$ if $2^{13/12 - 1/(3K)}(A_3B_3)^{2/K}h^{1/K} \ge N^{1/K}$, which is true since 
\begin{equation}
A_3B_3 \ge \left(\frac{32\eta_3^{1/2}}{15\sqrt{\pi}}\right)^{1/2}\frac{\sqrt{32}}{\sqrt{3}\pi^{1/4}\eta_3^{1/4}} = \frac{32}{3\sqrt{5\pi}},
\end{equation}
and, since $h > 1$, $k \ge 4$,
\begin{equation}
2^{13K/12 - 1/3}(A_3B_3)^{2}h > 2^{25/3}\cdot \left(\frac{32}{3\sqrt{5\pi}}\right)^2 > 2336 \ge N.
\end{equation}
Therefore, the desired result follows from the trivial bound. 

Next, suppose $N \ge 2337$ and $\lambda_k \le \lambda_0(k)$, where 
\begin{equation}
\lambda_0(k) := \left(\frac{9\pi}{1024}\eta_3\right)^{-2 + 2/K}N^{-4 + 4/K}.
\end{equation}
Since $\delta_k > 1$, we have $B_{k} > B_{k - 1}^{1/2}$ for all $k \ge 4$, hence $B_{k} > B_3^{4/K}$. Thus
\begin{equation}
\begin{split}
&B_kN^{1 - 2/K}\lambda_k^{-1/(2K - 2)} \ge B_3^{4/K}N^{1 - 2/K}\lambda_0^{-1/(2K - 2)} \\
&\qquad\qquad = \left(\frac{\sqrt{32}}{\sqrt{3}\pi^{1/4}\eta_3^{1/4}}\right)^{4/K}\left[ \left(\frac{9\pi}{1024}\eta_3\right)^{-2 + 2/K}N^{-4 + 4/K}\right]^{-1/(2K - 2)}N^{1 - 2/K}\\
&\qquad\qquad = N.
\end{split}
\end{equation}
Hence the desired result once again follows from the trivial bound. 

We now proceed to the main argument, assuming that $N \ge 2337$ and $\lambda_{k} \ge \lambda_0(k)$ for all $k \ge 4$. By Lemma \ref{third_deriv_test}, the desired result holds for $k = 3$. Assume for an induction that the lemma holds for some $j \ge 3$. For convenience denote $J := 2^{j - 1}$. 

Suppose that $f(x)$ has $j + 1$ continuous derivatives on $[a, b]$, such that $\lambda_{j + 1} \le |f^{(j + 1)}(x)| \le h\lambda_{j + 1}$ for some $\lambda_{j + 1} > 0$. Let $g_r(x) := f(x + r) - f(x)$ so that, via the mean value theorem, 
\begin{equation}
g_r^{(j)}(x) = f^{(j)}(x + r) - f^{(j)}(x) = rf^{(j + 1)}(\xi),
\end{equation}
for all $r \ge 1$, $x \in (a, a + N - r]$ and some $\xi \in [x, x + r]$. From the conditions on $f^{(j + 1)}(x)$, we have 
\begin{equation}
r\lambda_{j + 1} \le |g_r^{(j)}(x)| \le rh\lambda_{j + 1}.
\end{equation}
By the inductive assumption, we have 
\begin{equation}\label{k_deriv_test_bound1}
S_{g_r}(a, N - r) \le A_j h^{2/J}N(r\lambda_{j + 1})^{1/(2J - 2)} + B_jN^{1 - 2/J}(r\lambda_{j + 1})^{-1/(2J - 2)},
\end{equation}
for some constants $A_j$, $B_j$ not depending on $r$ or $\lambda_{j + 1}$. Note we have used the inequality $N - r < N$ for simplicity. We once again apply \eqref{dhir_inequality} to obtain
\[
\sum_{r = 1}^{q - 1}\left(1 - \frac{r}{q}\right)r^{1/(2J - 2)} \le \alpha_j q^{1 + 1/(2J - 2)},\quad \sum_{r = 1}^{q - 1}\left(1 - \frac{r}{q}\right)r^{-1/(2J - 2)} \le \beta_j q^{1 - 1/(2J - 2)}
\]
where 
\begin{equation}
\alpha_j := \frac{4(J - 1)^2}{(2J - 1)(4J - 3)},\qquad \beta_j := \frac{4(J - 1)^2}{(2J - 3)(4J - 5)}. 
\end{equation}
Substituting into \eqref{k_deriv_test_bound1}, we obtain
\begin{equation}
\begin{split}
\frac{2}{q}\sum_{r = 1}^{q - 1}\left(1 - \frac{r}{q}\right)S_{g_r}(a, N - r) &\le 2\alpha_j A_{j} h^{2/J}N\lambda_{j + 1}^{1/(2J - 2)}q^{1/(2J - 2)} \\
&+ 2\beta_j B_{j}N^{1 - 2/J}\lambda_{j + 1}^{-1/(2J - 2)}q^{-1/(2J - 2)}.
\end{split}
\end{equation}
Using Lemma \ref{weyl_differencing} and the inequality $\sqrt{x + y} \le \sqrt{x} + \sqrt{y}$,
\begin{align}
S_f(a, N) &\le \frac{D_jN}{q^{1/2}} + D_jN^{1/2}\Bigg[(2\alpha_j A_j)^{1/2}\sqrt{h^{2/J}N(q\lambda_{j + 1})^{1/(2J - 2)}}\notag\\
&\qquad\qquad + (2\beta_j B_j)^{1/2}\sqrt{N^{1 - 2/J}(q\lambda_{j + 1})^{-1/(2J - 2)}}\Bigg]\\
&= \frac{D_jN}{q^{1/2}} + D_j(2\alpha_j A_j)^{1/2}h^{1/J}N(q\lambda_{j + 1})^{1/(4J - 4)} \notag\\
&\qquad\qquad + D_j (2\beta_j B_j)^{1/2} N^{1 - 1/J}(q\lambda_{j + 1})^{-1/(4J - 4)}\label{k_deriv_test_bound2}
\end{align}
where 
\begin{equation}
D_j := \sqrt{1 + \frac{q}{N}}.
\end{equation}
We choose $q = \left\lfloor \lambda_{j + 1}^{-1/(2J - 1)}\right\rfloor + 1$, so that we have (wastefully) 
\begin{equation}
\lambda_{j + 1}^{-1/(2J - 1)} < q \le 2\lambda_{j + 1}^{-1/(2J - 1)}.
\end{equation}
This choice of $q$ gives, via the assumption $\lambda_{j + 1} \ge \lambda_0(j + 1)$ and $N \ge 2337$,
\begin{equation}\label{k_deriv_test_D_j_bound}
D_j \le \sqrt{1 + \frac{2}{N}\left(\left(\frac{9\pi}{1024}\eta_3\right)^{-2 + 1/J}N^{-4 + 2/J}\right)^{-1/(2J - 1)}} \le \delta_{j}.
\end{equation}
Additionally, since $j \ge 3$, we also have
\begin{equation}
\frac{1}{4J - 4} \le \frac{1}{12},\qquad \frac{1}{4J - 4}\left(1 - \frac{1}{2J - 1}\right) = \frac{1}{4J - 2}.
\end{equation}
This gives the following estimates. 
\begin{equation}
(q\lambda_{j + 1})^{1/(4J - 4)} \le 2^{1/(4J - 4)}\lambda_{j + 1}^{\left[1 - 1/(2J - 1)\right]/(4J - 4)} < 2^{1/12}\lambda_{j + 1}^{1/(4J - 2)},
\end{equation}
\begin{equation}
(q\lambda_{j + 1})^{-1/(4J - 4)} < \lambda_{j + 1}^{-1/(4J - 2)},
\end{equation}
\begin{equation}
q^{-1/2} < \lambda_{j + 1}^{1/(4J - 2)},
\end{equation}
\begin{equation}
(q\lambda_{j + 1})^{-1/(2J - 2)} \le \lambda_{j + 1}^{(1/(2J - 1) - 1)/(2J - 2)} = \lambda_{j  + 1}^{-1/(2J - 1)}.
\end{equation}
Combining these with \eqref{k_deriv_test_D_j_bound} and \eqref{k_deriv_test_bound2}, we have
\begin{align}
S_f(a, N) &\le \delta_{j}N\lambda_{j + 1}^{1/(4J - 2)} + 2^{7/12}\delta_j(\alpha_jA_j)^{1/2}h^{1/J}N\lambda_{j + 1}^{1/(4J - 2)}   \notag\\
&\qquad\qquad + \delta_j (2\beta_j B_j)^{1/2}N^{1 - 1/J}\lambda_{j + 1}^{-1/(4J - 2)}\\
&= \delta_j\left[h^{-1/J} + 2^{7/12}(\alpha_jA_j)^{1/2}\right] h^{1/J}N\lambda_{j + 1}^{1/(4J - 2)}   \notag \\
&\qquad\qquad + \delta_j (2\beta_j B_j)^{1/2}N^{1 - 1/J}\lambda_{j + 1}^{-1/(4J - 2)}\\
&= A_{j + 1}h^{1/J}N\lambda_{j + 1}^{1/(4J - 2)} + B_{j + 1}N^{1 - 1/J}\lambda_{j + 1}^{-1/(4J - 2)},
\end{align}
hence the induction is complete. 
\end{proof}

\begin{remark}
In our eventual application we have $\lambda_k \asymp tN^{-k}$ where $t^{K/((k + 1)K - 2K + 1)} \ll N \ll t^{K/(kK - 2K + 2)}$, so that in particular,
\[
N^{-2 + 2/K} \ll \lambda_k \ll N^{-1 + 1/K} .
\]
Therefore, Lemma \ref{kth_deriv_test} implies (for fixed $h$) that $S_f(a, N) \ll N^{1 - 1/(2K)} = N^{1 - 1/2^{k - 1}}$. Meanwhile, $\log \lambda_k \asymp \log N$ so \eqref{granville_result} reduces to $S_f(a, N) \ll N^{1 - 1/2^{k - 1}}\log^{(k- 1)/2^{k - 1}} N$. In particular, in our application we take $k = 2$ when $t^{2/3} \ll N \ll t$. Therefore, using \eqref{granville_result} in place of Lemma \ref{kth_deriv_test} in the argument of Theorem \ref{theorem1} produces a bound of strength $\zeta(\sigma_k + it) \ll t^{1/(2K - 2)}\log^{3/2}t$. 
\end{remark}

For many applications we are interested in uniform bounds holding for all $k \ge k_0$. To this end we provide the following completely explicit result.  

\begin{lemma}\label{explicit_deriv_test}
Let $k \ge 10$, $a, N > 0$ be integers. Suppose $f(x)$ is any function having $k$ continuous derivatives with $f^{(k)}$ monotonic, and $\lambda_k \le |f^{(k)}(x)| \le h\lambda_k$ for all $x \in (a, a + N]$ and some $\lambda_k > 0$, $1 < h \le 3$.  Then
\[
S_f(a, N) \le 2.762\, h^{2/K} N \lambda_k^{1/(2K - 2)} + 1.02 N^{1 - 2/K}\lambda_k^{-1/(2K - 2)}.
\]
\end{lemma}
\begin{proof}
From Lemma \ref{third_deriv_test} and \eqref{s1_Ak_recursive_defn}, we observe that $A_k(\eta_3, h)$ is a decreasing function of $h$, for all $k \ge 3$. It thus suffices to bound $A_k$ with $h = 3$. We choose $\eta_3 = 4.7399$ in Lemma \ref{third_deriv_test} and recursively compute $A_k$, $B_k$ for $4 \le k \le 10$, using Lemma \ref{kth_deriv_test}, to ultimately obtain 
\begin{equation}\label{s2_A103_B103_bounds}
A_{10}(4.7399, 3) \le 2.744,\qquad B_{10}(4.7399) \le 1.020. 
\end{equation}
For $k \ge 10$, we note that $\delta_k$ is decreasing in $k$, so $\delta_k \le \delta_{10}$. Also, for any $K > 1$ we have
\begin{equation}
\frac{2^{19/12}(K - 1)}{\sqrt{(2K - 1)(4K - 3)}} = \frac{2^{19/12}}{\sqrt{\left(2 + \frac{1}{K - 1}\right)\left(4 + \frac{1}{K - 1}\right)}} < 2^{1/12}.
\end{equation}
By Lemma \ref{kth_deriv_test}, 
\begin{equation}
A_{k + 1} \le \delta_{10}\left(1 + 2^{1/12}A_{k}^{1/2}\right), \qquad k \ge 10.
\end{equation}
The discrete map $x_{n + 1} = \delta_{10}\left(1 + 2^{1/12}x_n^{1/2}\right)$ has a single stable fixed point
\begin{equation}
x^* = \left(2^{-11/12}\delta_{10} + \sqrt{2^{-11/6}\delta_{10}^2 + \delta_{10}}\right)^2 \le 2.762.
\end{equation}
Since $A_{k + 1} \le x_{k + 1}$ and $A_{10} \le x^*$, we have $A_{k} \le 2.762$ for all $k \ge 10$. As for $B_k$, for all $k \ge 10$, 
\begin{equation}
\delta_k\frac{2^{3/2}(K - 1)}{\sqrt{(2K - 3)(4K - 5)}} \le 1.002.
\end{equation}
The map $y_{n + 1} = 1.002y_n^{1/2}$ has a single stable fixed point $y^* = 1.002^{2} < 1.02$, so it follows from \eqref{s2_A103_B103_bounds} that $B_k \le 1.02$ for all $k \ge 10$. 

\end{proof}

\section{Bounds on $\zeta(s)$ in the critical strip}\label{sec:thm2_proof}

In this section we use the explicit $k$th derivative test to bound $\zeta(\sigma + it)$ for certain values of $\sigma$, specifically those corresponding to
\begin{equation}\label{sigma_k_defn}
\sigma_k := 1 - \frac{k}{2^k - 2}
\end{equation}
for integers $k \ge 4$. Such values of $\sigma$ lie in the interval $(1/2 , 1)$, so that we are bounding $\zeta(s)$ along vertical lines residing between the half-line and the 1-line. Such bounds can be used to develop explicit zero-free regions through the method of Ford \cite{ford_zero_2002}, which rely primarily on sharp bounds on $\zeta(s)$ slightly to the left of $\sigma = 1$. Using the $k$th derivative test, we can establish asymptotically sharper bounds on $\zeta(s)$ than what is possible by considering bounds on the half-line and convexity principle alone \cite{titchmarsh_theory_1986}. 

Throughout this section, we specialise the phase function $f(x)$ encountered in lemmas \ref{second_deriv_test}, \ref{third_deriv_test} and \ref{kth_deriv_test} to 
\begin{equation}\label{s3_f_defn}
f(x) := -\frac{t}{2\pi}\log x,
\end{equation}
so that
\begin{equation}
S_f(a, N) = \left|\sum_{a < n \le a + N}n^{-it}\right|.
\end{equation}

Our proof of Theorem \ref{theorem1} is divided into two sections. First, in \S \ref{zeta_bound_small_t} we prove the theorem for $t \le T_k$ where 
\begin{equation}\label{s2_T_k_defn}
T_k := \exp\left(\frac{2.6134(2^{k - 1} - 1) + 2.8876k}{k - 3}\right),\qquad k \ge 4.
\end{equation}
The main tools in this range are the Phragm\'en--Lindel\"of Principle combined with the following two bounds on $\sigma = 1/2$ and $\sigma = 1$ respectively:
\begin{align}
|\zeta(1/2 + it)| &\le 0.618 t^{1/6}\log t,\qquad t\ge 3, \label{hiary_bound}\\
|\zeta(1 + it)| &\le \log t,\qquad t \ge 3. \label{backlund_bound}
\end{align}
The first bound is due to \cite{hiary_improved_2022}, and the second is due to \cite{backlund_uber_1916}. We note that sharper bounds on the 1-line are known for large $t$ \cite{trudgian_new_2014, patel_explicit_2022}, however our argument requires a bound holding for all $t \ge 3$. Second, in \S \ref{large_t_region}, we prove Theorem \ref{theorem1} for $t > T_k$ using Euler--Maclaurin summation and Lemma \ref{third_deriv_test} and \ref{kth_deriv_test} to make explicit the argument of \cite[Thm. 5.13]{titchmarsh_theory_1986}.

As in the previous section, throughout let $K := 2^{k - 1}$. We will write $\lfloor x\rfloor$ and $\lceil x \rceil$ to mean the largest integer no greater than $x$, and the smallest integer no smaller than $x$, respectively. Unless otherwise specified, $s = \sigma + it$ with $\sigma \in (0, 1]$ and $t > 0$. 
\subsection{Proof for $3 \le t \le T_k$}
\label{zeta_bound_small_t}

In this range, we use the following version of the Phragm\'en--Lindel\"of Principle, due to Trudgian \cite{trudgian_improved_2014}.

\begin{lemma}[Phragm\'en--Lindel\"of Principle]\label{PLPlem}
Let $a, b, Q$ be real numbers satisfying $a < b$ and $a + Q > 1$. Suppose $f(s)$ is a holomorphic function in $a \le \Re s \le b$ such that $|f(s)| < C\exp(e^{k|t|})$ for some $C > 0$ and $k < \pi/(b - a)$. Suppose further that 
\[
|f(s)| \le \begin{cases}
A|Q + s|^{\alpha_1}\log^{\alpha_2}|Q + s| &\text{for }\Re s = a\\
B|Q + s|^{\beta_1}\log^{\beta_2}|Q + s| &\text{for }\Re s = b
\end{cases}
\]
for some $A, \alpha_1, \alpha_2, B, \beta_1, \beta_2 > 0$. Then, for all $a < \Re s < b$, 
\[
|f(s)| \le \left(A|Q + s|^{\alpha_1}\log^{\alpha_2}|Q + s|\right)^{\frac{b - \Re s}{b - a}}\left(B|Q + s|^{\beta_1}\log^{\beta_2}|Q + s|\right)^{\frac{\Re s - a}{b - a}}.
\]
\end{lemma}
\begin{proof}
See \cite[Lem. 3]{trudgian_improved_2014}.
\end{proof}
The motivation for using such a convexity argument for small $t$ is the bounds \eqref{hiary_bound} and \eqref{backlund_bound} are comparatively sharp for small $t$. We choose the holomorphic function $f(s) := (s - 1)\zeta(s)$. 

First, we verify numerically that if $s = 1/2 + it$, then
\begin{equation}
\sup_{|t| \le 3}|(s - 1)\zeta(s)| < 0.618|1.31 + s|^{7/6}\log|1.31 + s|
\end{equation}
and if $s = 1 + it$, then 
\begin{equation}
\sup_{|t| \le 3}|(s - 1)\zeta(s)| < |1.31 + s|\log|1.31 + s|.
\end{equation}
Therefore, by combining with \eqref{hiary_bound} and \eqref{backlund_bound} and using Lemma \ref{PLPlem},
\begin{equation}
\begin{split}
|\zeta(s)| &\le \frac{1}{|s - 1|}\left(0.618|Q_0 + s|^{7/6}\log|Q_0 + s|\right)^{2 - 2\sigma}\left(|Q_0 + s|\log|Q_0 + s|\right)^{2\sigma - 1}\\
&= \frac{0.618^{2 - 2\sigma}}{|s - 1|}|Q_0 + s|^{(4 - \sigma)/3}\log|Q_0 + s|,\qquad 1/2 \le \Re s \le 1,
\end{split}
\end{equation}
where $Q_0 = 1.31$. If $s = \sigma_k + it$ with $1/2 \le \sigma_k \le 1$, we have $|Q_0 + s| = \sqrt{(Q_0 + \sigma_k)^2 + t^2} \le \sqrt{2.31^2 + t^2}$. If furthermore $t \ge t_0$, then 
\[
\log|Q_0 + s| \le \frac{\log\left(t\sqrt{\frac{2.31^2}{t_0^2} + 1}\right)}{\log t}\log t \le \left(1 + \frac{\log \left(\frac{2.31^2}{t_0^2} + 1\right)}{2\log t_0}\right)\log t
\]
and for $k \ge 4$, $\sigma_k \ge 5/7$, hence for $t \ge t_0$
\begin{equation}
|Q_0 + s|^{(4 - \sigma_k) / 3} \le \left(\frac{2.31^2}{t_0^2} + 1\right)^{23/42}t^{(4 - \sigma_k) / 3}.
\end{equation}
Hence 
\begin{equation}\label{zeta_prelim_bound}
|\zeta(\sigma_k + it)| \le 0.618^{2 - 2\sigma_k} A(t_0) t^{(1 - \sigma_k) / 3}\log t,\qquad t \ge t_0,
\end{equation}
where 
\begin{equation}
A(t_0) := \left(\frac{2.31^2}{t_0^2} + 1\right)^{23/42}\left(1 + \frac{\log \left(\frac{2.31^2}{t_0^2} + 1\right)}{2\log t_0}\right).
\end{equation}
The RHS of \eqref{zeta_prelim_bound} is majorized by $1.546 t^{1/(2K - 2)}\log t$ if 
\begin{equation}\label{t_requirement}
0.618^{k/(K - 1)} A(t_0) \le 1.546 t^{(3 - k)/(6K - 6)},
\end{equation}
i.e.\ if 
\begin{equation}\label{t1_requirement}
t \le t_1(k) = \left(\frac{A(t_0)}{1.546}\right)^{(6K - 6)/(3 - k)}0.618^{6k/(3 - k)}.
\end{equation}
Taking $t_0 = 3$, we have $A(t_0) \le 1.4747$, in which case $t_1(k) \ge \exp(8.7)$ for all $k \ge 4$. Therefore, \eqref{t_requirement} is satisfied for all $3 \le t \le \exp(8.7)$. Similarly, taking $t_0 = \exp(8.7)$, we have $A(t_0) \le 1.0001$ and 
\begin{equation}
t_1(k) \ge \exp\left(\frac{(6K - 6)\log \frac{1.546}{1.0001} - 6k\log 0.618}{k - 3}\right) \ge T_k,\qquad k \ge 4.
\end{equation}
It follows that 
\begin{equation}
|\zeta(\sigma_k + it)| \le 1.546t^{1/(2K - 2)}\log t,\qquad 3 \le t \le T_k, \;k\ge 4,
\end{equation}
as required.

\subsection{Proof for $t \ge T_k$}\label{large_t_region}
This subsection contains our main argument. We begin by bounding the difference between $\zeta(s)$ and its partial sum using Euler--Maclaurin summation. Recent explicit bounds on $\zeta(1/2 + it)$ have instead used the Riemann--Siegel formula \cite{platt_improved_2015, hiary_explicit_2016, hiary_improved_2022}, and the Gabcke \cite{gabcke_neue_1979} remainder bound. This produces better constants on the critical line since it allows us to consider an exponential sum of length $O(t^{1/2})$ instead of $O(t)$. Off the critical line, applying the Riemann--Siegel formula requires an explicit bound on the remainder term holding for $1/2 \le \sigma \le 1$. This can be achieved by appealing to the results of \cite{de_reyna_high_2011}, and has been done, for instance, on the 1-line in \cite{patel_explicit_2022}. For simplicity, however, we instead use Euler--Maclaurin summation, where explicit bounds on the remainder term have already been computed. 

To bound the remainder term arising from the Euler--Maclaurin summation, we use the following result due to Simoni\v{c} \cite[Cor. 2]{simonic_explicit_2020}, which builds on results in \cite[Thm. 1.2]{kadiri_zero_2013} and \cite[Prop. 1]{cheng_explicit_2000}. 

\begin{lemma}[\cite{simonic_explicit_2020} Corollary 2]\label{cheng_prop_1}
Let $s = \sigma + it$ where $1/2 \le \sigma \le 1$ and $t \ge t_0 > 0$. If $h > (2\pi)^{-1}$, then 
\begin{equation}
\left|\zeta(s) - \sum_{1 \le n \le ht}n^{-s}\right| \le \frac{1}{(ht_0)^{\sigma}}\left(h + \frac{1}{2} + 3\sqrt{1 + t_0^{-2}}\left(1 - \frac{1}{2h}\cot\frac{1}{2h}\right)\right).
\end{equation}
\end{lemma}

Suppose now that $j, k, r$ are positive integers with $k \ge 4$ and 
\begin{equation}
J := 2^{j - 1}, \qquad K := 2^{k - 1}, \qquad R := 2^{r - 1}.
\end{equation}
Throughout, we will write $s = \sigma_k + it$ where $\sigma_k = 1 - k/(2K - 2)$ and $t \ge T_k$. Furthermore let
\begin{equation}\label{varphi_j_defn}
\theta_r := \frac{R}{rR - 2R + 2},
\end{equation}
for all $2 \le r \le k$. Roughly speaking, our approach is to use Lemma \ref{cheng_prop_1} then use partial summation and Lemma \ref{kth_deriv_test} to bound
\begin{equation}\label{s3_target_sum}
\sum_{0 < n \le X_2}n^{-\sigma_k - it}.
\end{equation}
We divide the sum \eqref{s3_target_sum} into three subsums --- in $(0, X_0]$, we use the trivial bound, and in $(X_0, X_1]$ and $(X_1, X_2]$ we apply Lemma \ref{explicit_deriv_test} with different choices of $k$. Here, we ultimately make the choice 
\begin{equation}\label{defn_X1_X2}
X_1 := X_1(k, h_0) = \left\lfloor h_0 t^{\theta_k}\right\rfloor,\qquad X_2 := X_2(h_0, h_2) = \lfloor h_0h_2 t \rfloor
\end{equation}
for some scaling parameters $h_0, h_2 > 0$, to be chosen later.

\begin{lemma}
Let $s = \sigma_k + it$, where $t \ge t_0 > 0$ for some integer $k \ge 4$. Let $r \ge 2$ and $R = 2^{r - 1}$ and let $\eta_3 > 0$ be any real number. Furthermore, suppose $a$, $b$ are integers satisfying $0 < a < b\le ha$ for some $h > 1$. Then
\begin{equation}
\begin{split}
\left|\sum_{a < n \le b}n^{-s}\right| &\le C_r(\eta_3, h) a^{k/(2K - 2) - r/(2R - 2)}t^{1/(2R - 2)} \\
&\qquad\qquad + D_r(\eta_3, h) a^{k/(2K - 2) - 2/R + r/(2R - 2)}t^{-1/(2R - 2)}
\end{split}
\end{equation}
where 
\begin{equation}
\begin{split}
C_r(\eta_3, h) &:= A_r(\eta_3, h^r) h^{2r/R - r / (2R - 2)}(h - 1)\left(\frac{(r - 1)!}{2\pi}\right)^{\frac{1}{2R - 2}},\\
D_r(\eta_3, h) &:= B_r(\eta_3) h^{r/(2R - 2)}(h - 1)^{1 - 2/R}\left(\frac{2\pi}{(r - 1)!}\right)^{\frac{1}{2R - 2}}.
\end{split}
\end{equation}
\end{lemma}
\begin{proof}
Observe that with our choice of $f(x)$ in \eqref{s3_f_defn},
\begin{equation}
f^{(r)}(x) = \frac{(-1)^r(r - 1)! t}{2\pi x^{r}}.
\end{equation}
Hence for all $x \in (a, b]$, where $a$ and $b$ are integers satisfying $a < b \le ha$ for some $h > 1$, we have 
\begin{equation}
\frac{(r - 1)!}{2\pi}\frac{t}{(ha)^{r}} \le |f^{(r)}(x)| \le \frac{(r - 1)! }{2\pi}\frac{t}{a^r}.
\end{equation}
Therefore, we may apply Lemma \ref{kth_deriv_test} with 
\begin{equation}
\lambda_r = \frac{(r - 1)!}{2\pi}\frac{t}{(h a)^r},\qquad h = h^r,\qquad N = b - a \le (h - 1)a.
\end{equation}
This gives
\begin{align}
\sum_{a < n \le b}n^{-it} &\le A_r(\eta_3, h^r) h^{2r/R}(h - 1)a\left(\frac{(r - 1)!t}{2\pi(ha)^r}\right)^{1/(2R - 2)} \\
&\qquad\qquad + B_r(\eta_3) ((h - 1)a)^{1 - 2/R}\left(\frac{(r - 1)!t}{2\pi (ha)^r}\right)^{-1/(2R - 2)}     \notag\\
&= C_r a^{1 - r/(2R - 2)}t^{1/(2R - 2)} + D_r a^{1 - 2/R + r/(2R - 2)}t^{-1/(2R - 2)}.   \label{s3_CD_bound}
\end{align}
By partial summation, for any $f$ and $\sigma > 0$ we have  
\begin{align}
\left|\sum_{a < n \le b}n^{-\sigma}e(f(n))\right| &= \left|b^{-\sigma}\sum_{a < n \le b}e(f(n)) + \int_{a}^{b}\sigma t^{-\sigma - 1}\sum_{a < n \le t}e(f(n))\text{d}t\right|\\
&\le a^{-\sigma}\max_{a < L \le b}\left|\sum_{a < n \le L}e(f(n))\right| .    \label{s3_partialsum}
\end{align}
The result follows from combining \eqref{s3_CD_bound} and \eqref{s3_partialsum}.
\end{proof}

\subsubsection{The small region $[1, X_1]$}

\begin{lemma}\label{small_range_lem}
Let $s = \sigma_k + it$ with $t \ge t_0$. Then, for any $h_0 > 0$, $h_1 > 1$, $\eta_3 > 0$ and $(kK - 2K + 2)^{-1} \le \phi \le k^{-1}$, 
\[
\left|\sum_{1 \le n \le X_1}n^{-s}\right| \le \alpha_k(h_0, h_1, \eta_3, \phi, t_0) \; t^{1/(2K - 2)}\log t.
\]
where $X_1 = X_1(k, h_0)$ is defined in \eqref{defn_X1_X2}, and
\begin{equation}
\begin{split}
&\alpha_k(h_0, h_1, \eta_3, \phi, t) := \frac{2K - 2}{kt^{(1 - \phi k)/(2K - 2)}\log t} + \frac{1 - \frac{2K - 2}{k}}{t^{1/(2K - 2)}\log t}\\
&\qquad +\left(\frac{\theta_k - \phi}{\log h_1} + \frac{\max\left\{0, \frac{\log (h_0h_1)}{\log h_1}\right\}}{\log t}\right)\left(C_k(\eta_3, h_1) + D_k(\eta_3, h_1) h_1^{2/K - k/(K - 1)}\right) \\
&\qquad + \frac{1}{\log t}\frac{h_1^{1 - k/(2K - 2)}}{h_1^{1 - k/(2K - 2)} - 1}h_0^{k/(2K - 2) - 1}t^{1/(kK - 2K + 2) - \phi}. 
\end{split}
\end{equation}
with $\theta_k$ is as defined in \eqref{varphi_j_defn}. 
\end{lemma}
\begin{proof}
Let $\phi$ be a parameter to be chosen later, satisfying
\begin{equation}
\frac{1}{kK - 2K + 2} \le \phi \le \frac{1}{k}.
\end{equation}
Define
\begin{equation}
M := \left\lceil \frac{(\theta_k - \phi)\log t + \log h_0}{\log h_1}\right\rceil,
\end{equation}
\begin{equation}\label{s3_X1_X0_defn}
X_1 := \left\lfloor h_0 t^{\theta_k}\right\rfloor,\qquad X_0 := \lfloor h_1^{-M}X_1 \rfloor.
\end{equation}
First, observe that
\begin{equation}
X_0 \le h_0 h_1^{-\left(\left(\theta_k - \phi\right)\log t + \log h_0\right) / \log h_1}t^{\theta_k} = t^{\phi}.
\end{equation}
Hence, from the trivial bound we have 
\begin{equation}
\begin{split}
\left|\sum_{1 \le n \le X_0}n^{-s}\right| &\le \sum_{1 \le n \le X_0}n^{-\sigma_k} \le 1 + \int_1^{X_0}u^{-\sigma_k}\text{d}u \\
&= 1 + \frac{X_0^{1 - \sigma} - 1}{1 - \sigma} \le 1 + \frac{2K - 2}{k}(t^{\phi k/(2K - 2)} - 1)
\end{split}
\end{equation}
since $1 - \sigma = k/(2K - 2)$. If $k$ is so large that $X_0 < 1$, then the sum on the LHS is empty while the RHS is positive, so the inequality holds regardless. Furthermore, as $\phi k \le 1$, we have, for $t \ge t_0$,
\begin{equation}\label{s3_trivial_bound_part}
\left|\sum_{1 \le n \le X_0}n^{-s}\right| \le \left(\frac{2K - 2}{kt_0^{(1 - \phi k)/(2K - 2)}\log t_0} + \frac{1 - \frac{2K - 2}{k}}{t_0^{1/(2K - 2)}\log t_0}\right)t^{1/(2K - 2)}\log t.
\end{equation}
Next, consider the sum over the interval $(X_1, X_0]$. We divide the interval into $M$ pieces of the form $(Y_{m + 1}, Y_{m}]$, where
\begin{equation}
Y_m := \left\lfloor h_1^{-m}X_1 \right\rfloor,\qquad 0 \le m \le M. 
\end{equation}
Note that $Y_0 = X_1$ and $Y_M \le X_0$, so the entire interval $(X_0, X_1]$ is covered. We divide the sum over $(Y_{m}, Y_{m - 1}]$ into 
\begin{equation}\label{s3_Ym_division_sum}
\sum_{Y_m < n \le Y_{m - 1}}n^{-s} = \sum_{Y_{m} < n \le \lfloor h_1Y_{m} \rfloor}n^{-s} + \sum_{\lfloor h_1Y_{m} \rfloor < n \le Y_{m - 1}}n^{-s} = S_1 + S_2,
\end{equation}
say. Recalling that $\sigma = 1 - k/(2K - 2)$, we take $a = Y_{m}$, $r = k$ and $h = h_1$ in \eqref{s3_CD_bound} and combining with \eqref{s3_partialsum}, we obtain
\begin{align}
S_1 &\le C_k(\eta_3, h_1) Y_m^{k/(2K - 2) - k/(2K - 2)}t^{1/(2K - 2)}      \notag\\
&\qquad\qquad + D_k(\eta_3, h_1) Y_m^{k/(2K - 2) - 2/K + k/(2K - 2)}t^{-1/(2K - 2)}   \label{s3_Ym_bound_line1}\\
&\le C_k t^{1/(2K - 2)} + D_k\left(h_1^{-m}h_0 t^{\theta_k}\right)^{k/(K - 1) - 2/K}t^{-1/(2K - 2)}  \label{s3_Ym_bound_line2}\\
&= \left(C_k + D_k(h_1^{-m}h_0)^{k/(K - 1) - 2 / K}\right)t^{1/(2K - 2)}  \label{s3_Ym_bound_line3}
\end{align}
where, passing from \eqref{s3_Ym_bound_line1} to \eqref{s3_Ym_bound_line2} we used 
\begin{equation}
Y_m \le h_1^{-m}X_1 \le h_1^{-m}h_0t^{\theta_k},
\end{equation}
and passing from \eqref{s3_Ym_bound_line2} to \eqref{s3_Ym_bound_line3} we used 
\begin{equation}\label{s3_K_equivalence}
\theta_k\left(\frac{k}{K - 1} - \frac{2}{K}\right) - \frac{1}{2K - 2} = \frac{1}{2K - 2},
\end{equation}
for all $k \ge 4$. 

To bound $S_2$ of \eqref{s3_Ym_division_sum}, we note that 
\begin{equation}
Y_{m - 1} - \lfloor h_1 Y_{m}\rfloor < h_1^{-(m - 1)}X_1 - h_1(h_1^{-m}X_1 - 1) + 1 \le h_1 + 1.
\end{equation}
However $Y_{m - 1} - \lfloor h_1 Y_{m}\rfloor$ is an integer so $Y_{m - 1} - \lfloor h_1 Y_{m}\rfloor \le \lceil h_1 \rceil$, i.e.\ there are at most $\lceil h_1 \rceil$ terms in $S_2$. Therefore, via the trivial bound, and noting that $k/(2K - 2) - 1 < 0$,
\begin{equation}
\begin{split}
S_2 \le \lceil h_1\rceil (\lfloor h_1 Y_m \rfloor + 1)^{k/(2K - 2) - 1} \le \lceil h_1\rceil (h_1^{-(m - 1)}h_0t^{\theta_k})^{k/(2K - 2) - 1}
\end{split}
\end{equation}
Therefore, writing $\delta_1 = h_1^{-k/(K - 1) + 2/K}$ and $\delta_2 = h_1^{1 - k/(2K - 2)}$, 
\begin{equation}
\left|\sum_{X_0 < n \le X_1}n^{-s}\right| \le \sum_{m = 1}^{M}\left|\sum_{Y_{m} < n \le Y_{m - 1}}n^{-s}\right| = \left(C_k M + D_k\sum_{m = 1}^{M}\delta_1^m\right) t^{1/(2K - 2)} + S_3
\end{equation}
where 
\begin{equation}\label{T1_defn}
S_3 = h_0^{k/(2K - 2) - 1}t^{\theta_k(k/(2K - 2) - 1)}\sum_{m = 1}^M\delta_2^{m - 1}.
\end{equation}
Since $\delta_2 > 1$ we have 
\begin{equation}\label{s3_delta2_sum_bound}
\sum_{m = 1}^{M}\delta_2^{m - 1} = \frac{\delta_2^M - 1}{\delta_2 - 1} \le \frac{h_1^{1 - k/(2K - 2)}t^{\theta_k - \phi} - 1}{h_1^{1 - k/(2K - 2)} - 1} = \frac{\delta_2}{\delta_2 - 1}(t^{\theta_k - \phi} - 1).
\end{equation}
Meanwhile, from the lower bound on $\phi$, we have 
\begin{equation}\label{s3_exponent_bound}
\begin{split}
&\theta_k\left(\frac{k}{2K - 2} - 1\right) + \theta_k - \phi - \frac{1}{2K - 2} = \frac{1}{kK - 2K + 2} - \phi \le 0
\end{split}
\end{equation}
so that, in light of \eqref{T1_defn}, \eqref{s3_delta2_sum_bound} and \eqref{s3_exponent_bound}, we have, for $t \ge t_0$,
\begin{equation}
S_3 \le E_1(t_0) t^{1/2K - 2}\log t,
\end{equation}
\begin{equation}\label{s3_E1_defn}
E_1(t) := \frac{1}{\log t}\frac{\delta_2}{\delta_2 - 1}h_0^{k/(2K - 2) - 1}t^{1/(kK - 2K + 2) - \phi}. 
\end{equation}
Next, since $\delta_1 < 1$, we have trivially
\begin{align}
\sum_{m = 1}^{M}\delta_1^m < \delta_1M &\le Mh_1^{2/K - k/(K - 1)}. \label{s3_delta1_lastline}
\end{align}
Meanwhile, if $t \ge t_0$ and $h_1 > 1$ then 
\[
M \le \frac{(\theta_k - \phi)\log t}{\log h_1} + \frac{\log h_0}{\log h_1} + 1 \le \left(\frac{\theta_k - \phi}{\log h_1} + \frac{1}{\log t_0}\max\left\{0, \frac{\log h_0}{\log h_1} + 1\right\}\right)\log t
\]
Therefore, combining the previous estimates, we have for $t \ge t_0$,
\begin{align}
\left|\sum_{X_0 < n \le X_1}n^{-s}\right| &\le \left(\frac{MC_k}{\log t} + MD_k h_1^{2/K - k/(K - 1)} + E_1(t_0)\right) t^{1/(2K - 2)}\log t \\
&\le E_2(h_0, h_1, \eta_3, \phi, t_0)\; t^{1/(2K - 2)}\log t
\end{align}
where 
\begin{align*}
&E_2(h_0, h_1, \eta_3, \phi, t) = \frac{1}{\log t}\frac{h_1^{1 - k/(2K - 2)}}{h_1^{1 - k/(2K - 2)} - 1}h_0^{k/(2K - 2) - 1}t^{1/(kK - 2K + 2) - \phi}\\
&\qquad + \left(\frac{\theta_k - \phi}{\log h_1} + \frac{\max\left\{0, \frac{\log (h_0h_1)}{\log h_1}\right\}}{\log t}\right)\left(C_k(\eta_3, h_1) + D_k(\eta_3, h_1) h_1^{2/K - k/(K - 1)}\right)
\end{align*}
is decreasing in $t$. Combining with \eqref{s3_trivial_bound_part}, the result follows. 
\end{proof}

\subsubsection{The large region $(X_1, X_2]$} In this region we use Lemma \ref{explicit_deriv_test} with smaller choices of $k$. 

\begin{lemma}\label{bound_large}
Let $k \ge 4$ and $s = \sigma_k + it$ with $t \ge t_0 > 0$. Then, for any $h_0, \eta_3 > 0$, $h_2 \in (0, e)$ and $h_3 > 1$, we have
\begin{equation}
\left|\sum_{X_1 < n \le X_2}n^{-s}\right| \le \beta_k(t_0) t^{1/(2K - 2)}\log t
\end{equation}
where $X_1(k, h_0)$ and $X_2(h_0, h_2)$ are defined in \eqref{defn_X1_X2}, and 
\begin{equation}\label{Sk_defn}
\beta_k(t) = \beta_k(h_0, h_2, h_3, \eta_3, t) := \sum_{r = 2}^{k - 1}F_k(r, t),
\end{equation}
\begin{equation}\label{Erk_defn}
\begin{split}
&F_{k}(r, t) := \left(\frac{\theta_r - \theta_{r + 1}}{\log h_3} - \frac{\log h_2}{(k - 2)\log t} + \frac{1}{\log t}\right) t^{(\theta_rk - 1)/(2K - 2)}\times\\
&\quad \bigg[\left(C_r(\eta_3, h_3)H^{k/(2K - 2) - r/(2R - 2)} + D_r(\eta_3, h_3)H^{k/(2K - 2) - 2/R + r/(2R - 2)}\right)t^{-\theta_r/R}\\
&\qquad\qquad+ \lceil h_3\rceil (h_3H)^{k/(2K - 2) - 1}t^{-\theta_r}\bigg]
\end{split}
\end{equation}
and 
\begin{equation}\label{H_defn}
H := H(r, k) = \frac{h_0h_2^{(k - r)/(k - 2)}}{h_3}.
\end{equation}
\end{lemma}
\begin{proof}
We begin by defining  
\begin{equation}
Z_r :=  \lfloor h_0h_2^{(k - r)/(k - 2)}  t^{\theta_r} \rfloor,\qquad 2 \le r \le k.
\end{equation}
where $\theta_k$ is defined in \eqref{varphi_j_defn}. Noting that $Z_2 = X_2$ and $Z_{k} = X_1$, we have 
\begin{equation}\label{s3_X1_t_bound}
\left|\sum_{X_1 < n \le X_2}n^{-s}\right| \le \sum_{r = 2}^{k - 1}\left|\sum_{Z_{r + 1} < n \le Z_r}n^{-s}\right|.
\end{equation}
We further divide each interval $(Z_{r + 1}, Z_r]$ into intervals of the form $(W_{m}, W_{m - 1}]$, $m = 1, 2, \ldots, M_r$ where
\begin{equation}
W_{m} := \max\{Z_{r + 1}, \lfloor h_3^{-m}Z_r \rfloor\}
\end{equation}
and $M_r$ is the smallest integer for which $W_{M_r} = Z_{r + 1}$. Note in particular that since $h_3 > 1$,
\begin{equation}
\left\lfloor h_3^{-\left\lceil \frac{1}{\log h_3}\left((\theta_r - \theta_{r + 1})\log t - \frac{\log h_2}{k - 2}\right)\right\rceil} Z_{r} \right\rfloor \le \lfloor h_0 h_2^{(k - r - 1)/(k - 2)} t^{-(\theta_r - \theta_{r + 1})}t^{\theta_r} \rfloor = Z_{r + 1},
\end{equation}
so that
\begin{equation}
M_r \le \left\lceil\frac{1}{\log h_3}\left((\theta_r - \theta_{r + 1})\log t - \frac{\log h_2}{k - 2}\right)\right\rceil.
\end{equation}
Consider now the sum over $(W_{m}, W_{m - 1}]$. We divide the sum as before, with
\begin{equation}\label{s3_second_sum_div}
\sum_{W_{m} < n \le W_{m - 1}}n^{-s} = \sum_{W_m < n \le \lfloor h_3W_{m} \rfloor}n^{-s} + \sum_{\lfloor h_3W_m \rfloor < n \le W_{m - 1}}n^{-s}.
\end{equation}
Since $\lfloor h_3 W_m \rfloor \le h_3 W_m$, we may apply \eqref{s3_CD_bound} to the first sum with $a = W_{m}$ and $h = h_3$ to obtain, for $r \ge 2$ and $1 \le m \le M_r$,
\begin{equation}\label{s3_main_W_bound}
\begin{split}
\left|\sum_{W_{m} < n \le h_3\lfloor W_{m}\rfloor} n^{-s}\right| &\le C_r(\eta_3, h_3) W_m^{\kappa_3}t^{1/(2R - 2)} + D_r(\eta_3, h_3) W_m^{\kappa_4}t^{-1/(2R - 2)}
\end{split}
\end{equation}
where
\begin{equation}
\kappa_3 := \kappa(r, k) = \frac{k}{2K - 2} - \frac{r}{2R - 2},\qquad \kappa_4 := \kappa_4(r, k) = \frac{k}{2K - 2} - \frac{2}{R} + \frac{r}{2R - 2}.
\end{equation}
For $k > r \ge 2$, we have the identities
\begin{equation}\label{s3_exponent_identity_1}
\theta_r\kappa_3 + \frac{1}{2R - 2} = \theta_r\left(\frac{k}{2K - 2} - \frac{r - \theta_r^{-1}}{2R - 2}\right) = \theta_r\left(\frac{k}{2K - 2} - \frac{1}{R}\right),
\end{equation}
\begin{equation}\label{s3_exponent_identity_2}
\theta_r\kappa_4 - \frac{1}{2R - 2} = \theta_r\left(\frac{k}{2K - 2} - \frac{1}{R}\right).
\end{equation}
Therefore, noting that $W_m$ is decreasing in $m$ and $\kappa_3 > 0$ for all $r < k$, we have, via  \eqref{s3_exponent_identity_1},
\begin{equation}\label{W_m_kappa_3_bound}
\begin{split}
W_m^{\kappa_3}t^{1/(2R - 2)} \le W_1^{\kappa_3}t^{1/(2R - 2)} &\le \left(h_3^{-1}h_0h_2^{(k - r)/(k - 2)}t^{\theta_r}\right)^{\kappa_3}t^{1/(2R - 2)} \\
&= H^{\kappa_3}t^{\theta_r(k/(2K - 2) - 1/R)}
\end{split}
\end{equation}
where $H$ is defined in \eqref{H_defn}. Similarly, via \eqref{s3_exponent_identity_2} and $\kappa_4 > 0$ we have 
\[
W_m^{\kappa_4}t^{-1/(2R - 2)} \le H^{\kappa_4}t^{\theta_r(k/(2K - 2) - 1/R)}.
\]
For the second sum in \eqref{s3_second_sum_div}, we note as before that 
\[
W_{m - 1} - \lfloor h_1 W_{m}\rfloor < h_3^{-(m - 1)}Z_r - h_3(h_3^{-m}Z_r - 1) + 1 \le h_3 + 1.
\]
and hence, since $W_{m - 1} - \lfloor h_1 W_{m}\rfloor$ is an integer, we in fact have $\lfloor h_1 W_{m}\rfloor \le \lceil h_3 \rceil$, which is the maximum number of terms in the second sum of $\eqref{s3_second_sum_div}$. Therefore, using the fact that $n^{-\sigma}$ is decreasing and $h_3W_m < \lfloor h_3W_m\rfloor + 1$,
\begin{equation}\label{s3_W_m_remainder_term_bound}
\begin{split}
\left|\sum_{\lfloor h_3W_m \rfloor < n \le W_{m - 1}}n^{-s}\right| &< \lceil h_3\rceil (h_3 W_m)^{k/(2K - 2) - 1} \\
&\le \lceil h_3\rceil\left(h_3^{-(m - 1)}h_0h_2^{(k - r)/(k - 2)} t^{\theta_r}\right)^{k/(2K - 2) - 1}.
\end{split}
\end{equation}
Combining \eqref{s3_main_W_bound}, \eqref{W_m_kappa_3_bound} and \eqref{s3_W_m_remainder_term_bound} gives
\begin{align}
\left|\sum_{Z_{r + 1} < n \le Z_r}n^{-s}\right| &\le \sum_{m = 1}^{M_r}\left|\sum_{W_m < n \le W_{m - 1}}n^{-s}\right| \\
&\le M_r\left[C_r(\eta_3, h_3)H^{\kappa_3} + D_r(\eta_3, h_3)H^{\kappa_4}\right]t^{\theta_r(k/(2K - 2) - 1/R)} \notag \\
&\qquad\qquad + M_r\lceil h_3\rceil (h_3H)^{k/(2K - 2) - 1} t^{\theta_r(k/(2K - 2) - 1)}.\label{s3_Zr_bound0}
\end{align}
Using 
\begin{equation}
M_r \le \frac{(\theta_r - \theta_{r + 1})\log t}{\log h_3} - \frac{\log h_2}{k - 2} + 1
\end{equation}
and substituting into \eqref{s3_X1_t_bound}, we obtain
\begin{equation}
\left|\sum_{X_1 < n \le X_2}n^{-s}\right| \le \beta_k(t) t^{1/(2K - 2)}\log t
\end{equation}
where $\beta_k(t)$ is defined in \eqref{Sk_defn}. 

Next, we will show that if $k$ is fixed, then $F_{k}(r, t)$, and hence $\beta_k(t)$, is a decreasing function of $t$, so that $\beta_k(t) \le \beta_k(t_0)$ for $t \ge t_0$. Upon noting that for $2^x \ge x + 1$ for $x \ge 1$, we have 
\begin{equation}
(k - r)R \le (2^{k - r} - 1)R = K - R
\end{equation}
and hence
\begin{equation}
\begin{split}
\theta_r\left(\frac{k}{2K - 2} - \frac{1}{R}\right) - \frac{1}{2K - 2} &= \frac{1}{2K - 2}\left(\frac{kR}{rR - 2R + 2} - 1\right) - \frac{1}{rR - 2R + 2}\\
&= \frac{(k - r)R + 2(R - K)}{(2K - 2)(rR - 2R + 2)} \le 0.
\end{split}
\end{equation}
Additionally,  
\begin{equation}
\theta_r \left(\frac{k}{2K - 2} - 1\right) - \frac{1}{2K - 2} < \theta_r \left(\frac{k}{2K - 2} - \frac{1}{R}\right) - \frac{1}{2K - 2}\le 0.
\end{equation}
Finally, from the assumption that $h_2 < e$ and $k \ge 4$,
\begin{equation}
-\frac{\log h_2}{k - 2} + 1 > 0
\end{equation}
and so $F_k(r, t)$ is decreasing, so $\beta_k(t) \le \beta_k(t_0)$ for $t \ge t_0$. 
\end{proof}

\subsubsection{Putting it all together}\label{putting_it_all_together}
For each row $(k, \eta_3, h_0, h_1, h_2, h_3, \alpha_k, \beta_k, \gamma_k)$ of Table \ref{overall_bound_small_k_table}, we substitute the appropriate variables into Lemma \ref{small_range_lem} and \ref{bound_large} and take $t_0 = T_k$. This gives 
\begin{equation}
\left|\sum_{1 \le n \le X_2}n^{-s}\right| \le (\alpha_k + \beta_k)t^{1/(2K - 2)}\log t,\qquad t \ge T_k.
\end{equation}
Subsequently, taking $h = h_0h_2 > 1/(2\pi)$ and $\sigma = \sigma_k \in [5/7, 1)$ in Lemma \ref{cheng_prop_1},
\begin{equation}\label{final_bound_1}
\begin{split}
|\zeta(\sigma_k + it)| &\le (\alpha_k + \beta_k)t^{1/(2K - 2)}\log t + G(h_0 h_2, \sigma_k) \le \gamma_k t^{1/(2K - 2)}\log t,
\end{split}
\end{equation}
for $t \ge t_0 \ge 10^{10}$, where 
\begin{equation}\label{Gkt_defn}
G(h, \sigma) := \frac{1}{(ht_0)^{\sigma}}\left(h + \frac{1}{2} + 3\sqrt{1 + t_0^{-2}}\left(1 - \frac{1}{2h}\cot\frac{1}{2h}\right)\right).
\end{equation}
The last inequality of \eqref{final_bound_1} follows from 
\begin{equation}
\alpha_k + \beta_k + \frac{G(h_0h_2, \sigma)}{t^{1/(2K - 2)}\log t} \le \alpha_k + \beta_k + \frac{G(h_0h_2, \sigma)}{t_0^{1/(2K - 2)}\log t_0} \le \gamma_k
\end{equation}
for all $t \ge t_0$. 
This proves Theorem \ref{theorem1} for $4 \le k \le 9$ and $t \ge T_k$. 

\def\arraystretch{1.1}
\begin{table}[h]
\centering
\caption{Values appearing in \S \ref{putting_it_all_together}}
\begin{tabular}{|c|c|c|c|c|c|c|c|c|}
\hline
$k$ & $\eta_3$ & $h_0$ & $h_1$ & $h_2$ & $h_3$ & $\alpha_k$ & $\beta_k$ & $\gamma_k$ \\
\hline
4 & 1.22626 & 0.03640 & 1.30262 & 4.37500 & 1.30021 & 1.1796 & 0.3655 & 1.546\\
\hline
5 & 1.43074 & 0.10750 & 1.17205 & 17.2191 & 1.28297 & 0.7253 & 0.6401 & 1.366\\
\hline
6 & 1.79198 & 0.40548 & 1.08095 & 25.8377 & 1.19628 & 0.4944 & 0.6267 & 1.122\\
\hline
7 & 1.95195 & 0.97083 & 1.02940 & 6.87426 & 1.09787 & 0.3634 & 0.5350 & 0.899\\
\hline
8 & 1.94390 & 0.98846 & 1.01101 & 5.00587 & 1.05355 & 0.2824 & 0.4405 & 0.723\\
\hline
9 & 1.85285 & 0.99604 & 1.00392 & 3.80684 & 1.02923 & 0.2285 & 0.3652 & 0.594\\
\hline
\end{tabular}
\label{overall_bound_small_k_table}
\end{table}

Suppose now that $k \ge 10$. We take $h_0 = h_1 = h_3 = e$, $h_2 = 1$ and $\eta_3 = 4.7399$ (as in Lemma \ref{explicit_deriv_test}). This choice of parameters gives us $H = 1$ and, by Lemma \ref{explicit_deriv_test},
\begin{equation}
A_k(\eta_3, h_1) = A_k(\eta_3, h_3) \le 2.762,\qquad B_k(\eta_3, h_1) = B_k(\eta_3, h_3) \le 1.02. 
\end{equation}
We will show that $F_{k}(r, t)$ is decreasing in $k$ for $k \ge 10$, for fixed $r < k$ and $t > 1$. Let
\begin{equation}
f(k) = \frac{\theta_rk - 1}{2K - 2}
\end{equation}
so that, since $\theta_r$ is decreasing in $r$, 
\begin{equation}
\begin{split}
f(k) - f(k + 1) &= \theta_r\left(\frac{k}{2K - 2} - \frac{k + 1}{4K - 2}\right) + \frac{1}{4K - 2} - \frac{1}{2K - 2}\\
&\ge \frac{\theta_{k - 1}(kK - K + 2) - K}{2(K - 1)(2K - 1)}
\end{split}
\end{equation}
However, 
\begin{equation}
\theta_{k - 1} = \frac{K}{(k - 1)K - 2K + 4} > \frac{K}{kK - K + 2}
\end{equation}
and thus $f(k) - f(k + 1) > 0$, i.e.\ $f$ is decreasing. Therefore, the factor $t^{(\theta_rk - 1)/(2K - 2)}$ appearing in \eqref{Erk_defn} is decreasing in $k$ for fixed $t > 1$. Since $H = 1$, the only other way $F_k(r, t)$ depends on $k$ is through the factor
\begin{equation}
(h_3H)^{k/(2K - 2) - 1}
\end{equation}
which is decreasing in $k$ since $h_3H = h_3 > 1$. Hence, $F_k(r, t)$ is decreasing in $k$. Therefore, for all $k \ge k_0 \ge 4$, 
\begin{equation}\label{s3_subadditive_property}
\beta_k(t) = \sum_{r = 2}^{k - 1}F_{k}(r, t) \le \sum_{r = 2}^{k_0 - 1}F_{k_0}(r, t) + \sum_{r = k_0}^{k - 1}F_{k}(r, t) = \beta_{k_0}(t) + \sum_{r = k_0}^{k - 1}F_{k}(r, t),
\end{equation}
with the understanding that if $k = k_0$ then the last sum is 0. Now, for all $r < k$, we have 
\begin{equation}
k - r + 2 < 2\cdot 2^{k - r} = \frac{2K}{R}
\end{equation}
hence 
\begin{equation}
k - \frac{2K - 2}{R} < r - 2 + \frac{2}{R} = \frac{1}{\theta_r},
\end{equation}
from which it follows that
\begin{equation}
\frac{\theta_rk - 1}{2K - 2} - \frac{\theta_r}{R} < 0.
\end{equation}
This implies that for $t \ge 1$,
\begin{equation}\label{t_inequalities}
t^{(\theta_rk - 1)/(2K - 2)} < 1, \qquad t^{\theta_r(k/(2K - 2) - 1/R) - 1/(2K - 2)} < 1
\end{equation}
Substituting \eqref{t_inequalities} into \eqref{Erk_defn}, and using $h_2 = H = 1$, we obtain
\begin{equation}\label{Fk_final_bound}
F_{k}(r, t) \le \left(\frac{\theta_r - \theta_{r + 1}}{\log h_3} + \frac{1}{\log t}\right)\left(\lceil h_3\rceil h_0^{k/(2K - 2) - 1} + C_r(\eta_3, h_3) + D_r(\eta_3, h_3)\right).
\end{equation}
Next, for $k \ge 10$, and given the choices of $h_0$, $h_3$, 
\begin{equation}\label{H_group_final_bound}
\lceil h_3\rceil h_0^{k/(2K - 2) - 1} \le 1.115.
\end{equation}
Meanwhile, for $t \ge T_k$,
\begin{equation}\label{front_factor_main_bound}
\sum_{r = 10}^{k - 1}\left(\frac{\theta_r - \theta_{r + 1}}{\log h_3} + \frac{1}{\log t}\right) < \theta_{10} + \frac{k - 10}{\log t} \le \theta_{10} + \frac{(k - 10)(k - 3)}{3.2078(K - 1) + 2.8876k} \le 0.12494.
\end{equation}
If $r \ge 10$, then via Lemma \ref{explicit_deriv_test}, since $h_3 < 3$, we have $A_r \le 2.762$, $B_r \le 1.02$ and
\begin{equation}
h_3^{2r/R - r / (2R - 2)}(h_3 - 1)\left(\frac{\Gamma(r)}{2\pi}\right)^{\frac{1}{2R - 2}} \le 1.0179,
\end{equation}
\begin{equation}
h_3^{r/(2R - 2)}(h_3 - 1)^{1 - 2/R}\left(\frac{2\pi}{\Gamma(r)}\right)^{\frac{1}{2R - 2}} < 1,
\end{equation}
so that
\begin{equation}\label{CrDr_estimates}
C_r(h_3) \le 2.804,\qquad D_r(h_3) \le 1.02,\qquad r \ge 10.
\end{equation}
Combining \eqref{Fk_final_bound}, \eqref{H_group_final_bound}, \eqref{front_factor_main_bound}, and \eqref{CrDr_estimates}, we have 
\begin{equation}
\sum_{r = 15}^{k - 1}F_{r, k}(t) < 0.12494\left(1.115 + 2.804 + 1.02\right) \le 0.6171.
\end{equation}
Hence, using $k_0 = 10$ in \eqref{s3_subadditive_property},
\begin{equation}\label{beta_k_terminal_case_final_bound}
\begin{split}
\beta_k(h_0, h_2, h_3, \eta_3, T_k) &\le \beta_{10}(h_0, h_2, h_3, \eta_3, T_{10}) + \sum_{r = 10}^{k - 1}F_{r, k}(T_k)\\
&\le 0.6064 + 0.6171 = 1.2235.
\end{split}
\end{equation}
Finally, substituting $\eta_3 = 1, h_0 = e, h_1 = e, \phi = k^{-1}, t_0 = T_k \ge T_{10}$ into Lemma \ref{small_range_lem},
\begin{equation}\label{alpha_k_terminal_case}
\begin{split}
\alpha_k &= \frac{2K - 2}{k \log t} + \frac{1 - \frac{2K - 2}{k}}{t^{1/(2K - 2)}\log t}\\
&\qquad +\left(\theta_k - \frac{1}{k} + \frac{2}{\log t}\right)\left(C_k(\eta_3, h_1) + D_k(\eta_3, h_1) h_1^{2/K - k/(K - 1)}\right) \\
&\qquad + \frac{1}{\log t}\frac{h_1^{1 - k/(2K - 2)}}{h_1^{1 - k/(2K - 2)} - 1}h_0^{k/(2K - 2) - 1}t^{1/(kK - 2K + 2) - 1/k}. 
\end{split}
\end{equation}
First, we use $1 - e^{-x} \le x$ to obtain
\begin{equation}
\frac{2K - 2}{k\log t}\left(1 - t^{-1/(2K - 2)}\right) \le \frac{2K - 2}{k\log t}\frac{\log t}{2K - 2} \le \frac{1}{10},
\end{equation}
hence the first two terms of \eqref{alpha_k_terminal_case} are majorized by 
\begin{equation}\label{alpha_k_terminal_first_term_bound}
\frac{1}{10} + \frac{1}{T_k^{1/(2K - 2)}\log T_k} \le 0.105.
\end{equation}
Next, 
\begin{equation}\label{alpha_k_terminal_second_term_bound}
\theta_k - \frac{1}{k} + \frac{2}{\log t} \le \frac{2}{k(k - 2)} + \frac{2}{\log T_k} \le 0.036,
\end{equation}
and, as before, $C_k(\eta_3, h_1) \le 2.804$ and $D_k(\eta_3, h_1) \le 1.02$ since we have chosen $h_1 = h_3$ and $k \ge 10$. Combined with the trivial bound $h_1^{2/K - k/(K - 1)} < 1$, the third term of \eqref{alpha_k_terminal_case} is no greater than $0.138$. 

Finally, to bound the last term we combine the (wasteful) estimates 
\begin{equation}
h_1^{1 - k/(2K - 2)} \ge 2.691,\qquad h_0^{k/(2K - 2) - 1} < 1,\qquad \frac{1}{kK - 2K + 2} - \frac{1}{k} < 0,
\end{equation}
each valid for $k \ge 10$. Combined with $t > 1$ and the decreasingness of $x/(x - 1)$ for $x > 1$, the last term of \eqref{alpha_k_terminal_case} is bounded by  
\begin{equation}\label{alpha_k_terminal_third_term_bound}
\frac{1}{\log T_k}\frac{2.691}{2.691 - 1} \le 0.009. 
\end{equation}
Hence, combining \eqref{alpha_k_terminal_first_term_bound}, \eqref{alpha_k_terminal_second_term_bound} and \eqref{alpha_k_terminal_third_term_bound}, we finally have 
\begin{equation}
\alpha_k(h_0, h_1, \eta_3, \phi, T_k) \le 0.105 + 0.138 + 0.009 = 0.252.
\end{equation}
Combining with \eqref{beta_k_terminal_case_final_bound} gives
\begin{equation}
\left|\sum_{1 \le n \le \lfloor e t\rfloor}n^{-s}\right| \le 1.476 t^{1/(2K - 2)}\log t,\qquad t \ge T_k,\; k \ge 10. 
\end{equation}
Meanwhile, applying $\sigma_k \ge \sigma_{10}$ and $t \ge T_k \ge T_{10}$ to \eqref{Gkt_defn} gives the crude bound $G(h_0h_2, \sigma_k) \le 0.001$. Thus
\begin{equation}
|\zeta(\sigma_k + it)| \le 1.476 t^{1/(2K - 2)}\log t + 0.001 < 1.546 t^{1/(2K - 2)}\log t
\end{equation}
for all $k\ge 10$ and $t \ge T_k$, as required.

\section{Proof of Corollary \ref{corollary1}}\label{sec:cor1_proof}
Existing methods of deriving explicit zero-free regions of $\zeta(s)$ fall into two broad categories --- the ``global" approach of Hadamard, and the ``local" method of Landau. The first method relies on explicit estimates of $\Re\zeta'/\zeta(s)$ to the right of the line $\sigma = 1$, which in turn rely on estimates of $\Re\Gamma'/\Gamma(s)$ \cite{kadiri_region_2005}. Currently, the sharpest known classical zero-free region is derived using a variant of this method \cite{mossinghoff_explicit_2022}. 

The second method uses upper bounds on $|\zeta(s)|$ slightly to the left of $\sigma = 1$, i.e.\ within the critical strip. The strength of the resulting zero-free region depends on the sharpness of this bound. The explicit Vinogradov--Korobov zero-free region uses a Ford--Richert type bound 
\begin{equation}
\zeta(\sigma + it) \ll_{\varepsilon} t^{B(1 - \sigma)^{3/2} + \varepsilon},
\end{equation}
uniformly for $1/2 \le \sigma \le 1$, for some constant $B > 0$ and any $\varepsilon > 0$. On the other hand, the classical zero-free region can be obtained via a bound of the form 
\begin{equation}
\zeta(\sigma + it) \ll_{\varepsilon} t^{\theta + \varepsilon}
\end{equation}
for some fixed $\sigma, \theta > 0$ and any $\varepsilon > 0$ \cite{ford_zero_2002}. In this work, we use Theorem \ref{theorem1} to obtain a bound of the form
\begin{equation}
\zeta(\sigma(t) + it) \ll_{\varepsilon} t^{\theta(t) + \varepsilon}
\end{equation}
for some functions $\sigma(t) \to 1$ and $\theta(t) \to 0$ as $t \to \infty$, to prove Corollary \ref{corollary1}. 

Both methods also make critical use of a non-negative trigonometric polynomial $P(x)$, of the form
\begin{equation}
P(x) := \sum_{j = 0}^D b_j\cos(jx)
\end{equation}
such that $P(x) \ge 0$, $b_j \ge 0$ and $b_0 < b_1$. In this work, we use the polynomial of degree $D = 46$ presented in \cite[Table 1]{mossinghoff_explicit_2022}, which was found via a large-scale computational search. The coefficients $b_j$ of this polynomial satisfy
\begin{equation}\label{poly_coefficients}
b_0 = 1,\qquad b_1 = 1.74708744081848, \qquad b := \sum_{j = 1}^{D}b_j = 3.57440943022073.
\end{equation}

\subsection{Notation} 
Throughout, let $k \ge 4$ be an integer, and let 
\begin{equation}
\eta_k := \frac{k}{2^k - 2}.
\end{equation}
The variables $b_j, b$ and $D$ will be reserved for the trigonometric polynomial \eqref{poly_coefficients}. Non-trivial zeroes of $\zeta(s)$ are denoted by $\rho = \beta + it$ and $\rho' = \beta' + it'$. We also find it convenient to write 
\begin{equation}
L_1 := L_1(t) = \log(Dt + 1),\qquad L_2 := L_2(t) = \log\log(Dt + 1).
\end{equation}
The variables $A = 76.2$ and $B = 4.45$ will be used to denote the constants appearing in the Ford--Richert bound 
\begin{equation}\label{s4_richert_bound}
|\zeta(\sigma + it)| \le At^{B(1 - \sigma)^{3/2}}\log^{2/3}t,\qquad\frac{1}{2} \le \sigma \le 1, t \ge 3. 
\end{equation}
This inequality is instrumental in the proof of Ford's \cite{ford_zero_2002} Vinogradov--Korobov zero-free region \eqref{vk_zfr}. Throughout, we write
\begin{equation}
N(t, \eta) := \#\left\{\rho: |1 + it - \rho| \le \eta\right\}
\end{equation}
where the zeroes are counted with multiplicity.

Our argument is roughly the same as \cite{ford_zero_2002} --- an upper bound on $\zeta(s)$ within the critical strip is used to construct an inequality involving the real and imaginary parts of a hypothetical zero, then a contradiction is obtained if the real part of the zero is too large. We divide the argument into two sections, and this is reflected in the intermediary lemmas below. For small $t$, we use Theorem \ref{theorem1} and \cite[Thm. 1.1]{hiary_improved_2022}, combined with the Phragm\'en--Lindel\"of Principle to bound $\zeta(s)$ uniformly in $1/2 \le \sigma \le 5/7$. For large $t$, we use Theorem \ref{theorem1} directly by setting $k = k(t)$, tending to infinity with $t$. The two-part argument allows us to customise our tools for a specific region, significantly improving the constant in Corollary \ref{corollary1}. Throughout, we will reuse Ford's results where possible, making changes only when necessary or if a substantial improvement can be obtained.

We begin by bounding an integral involving $\zeta(s)$, taken on a vertical line within the critical strip. Ultimately, this bound is combined with an upper bound on $\zeta(s)$ to produce the ``main" term in the zero-free region constant. The following lemma is largely the same as \cite[Lem.\ 3.4]{ford_zero_2002}, the main difference being that we only require a bound on $\zeta(\sigma + it)$ to hold for large $t$ instead of for all $t \ge 3$. This technical change allows us to avoid applying the Phragm\'en--Lindel\"of Principle (Lemma \ref{PLPlem}) at very small values of $t$ where it is poorly suited. 

\begin{lemma}\label{fordlem34}
Let $0 < a \le 1/2$, $t_0 \ge 100$, $t \ge 3t_0$ and $1/2 \le \sigma \le 1 - t^{-1}$, and suppose
\[
|\zeta(\sigma + iy)| \le Xt^{Y}\log t,\qquad t_0 \le |y| \le t,
\]
for some $X, Y > 0$. Then
\[
\frac{1}{2}\int_{-\infty}^{\infty}\frac{\log|\zeta(\sigma + it + iau)|}{\cosh^2 u}\text{d}u \le \log X + Y\log t + \log\log t.
\]
\end{lemma}
\begin{proof}
We follow essentially the same argument as \cite[Lem. 3.4]{ford_zero_2002} by splitting the integral into four parts. Let
\begin{equation}
\begin{split}
\int_{-\infty}^{\infty}\frac{\log|\zeta(\sigma + it + iau)|}{\cosh^2 u}\text{d}u &= \int_{-\infty}^{-2t/a} + \int_{-2t/a}^{-(t + t_0)/a} + \int_{-(t + t_0)/a}^{-(t - t_0)/a} + \int_{-(t - t_0)/a}^{\infty}\\
&= I_1 + I_2 + I_3 + I_4. 
\end{split}
\end{equation}
To bound $I_1$, if $u \le -2t/a$ then $t + au \in (-\infty, -t]$ and by the lemma's assumptions, we have 
\begin{equation}
\begin{split}
\log|\zeta(\sigma + i(t + au))| &\le \log X + Y\log(-t - au) + \log\log(-t - au)\\
&\le \log(Xt^Y\log t) + \left(Y + \frac{1}{\log t}\right)\left(\frac{-au - 2t}{t}\right)
\end{split}
\end{equation}
where in the second inequality we have used $\log(x + 1) \le x$ with $x = -\frac{au}{t} - 2 \ge 0$. Therefore,
\begin{equation}\label{fordlem34_I1_bound}
I_1 \le \left[\log(Xt^Y\log t) + \left(Y + \frac{1}{\log t}\right)\left(\frac{-au - 2t}{t}\right)\right]\int_{-\infty}^{-2t/a}\frac{\text{d}u}{\cosh^2 u}.
\end{equation}
Next, if $-2t/a \le u \le -(t + t_0)/a$, then $t + au \in [-t, -t_0]$, so by the lemma's assumption, we have $|\zeta(\sigma + i(t + au))| \le Xt^Y\log t$ and
\begin{equation}\label{fordlem34_I2_bound}
I_2 \le \log(Xt^Y\log t)\int_{-2t/a}^{-(t + t_0)/a}\frac{\text{d}u}{\cosh^2u}
\end{equation}
If $-(t + t_0)/a \le u \le -(t - t_0)/a$, then $t + au \in [-t_0, t_0]$ and we use
\[
\zeta(s) = \frac{1}{s - 1} + \frac{1}{2} + s\int_1^{\infty}\frac{\lfloor x\rfloor - x + 1/2}{x^{s + 1}}\text{d}x.
\]
Letting $s = \sigma + i(t + au)$, we have $|s - 1| \ge |\Re(s - 1)| \ge t^{-1}$. Also, $|s| = \sqrt{\sigma^2 + (t + au)^2} < \sqrt{1^2 + t_0^2}$. Using $\sigma \ge 1/2$ and $t_0 \ge 1$, we have
\[
|\zeta(s)| \le \frac{1}{|s - 1|} + \frac{1}{2} + \frac{|s|}{2}\int_1^{\infty}\frac{\text{d}x}{x^{\sigma + 1}} \le t + \frac{1}{2} + \frac{1}{2\sigma}\sqrt{1 + t_0^2} < t + t_0 + 1,
\]
so that 
\begin{equation}\label{fordlem34_I3_bound}
I_3 \le \log(t + t_0 + 1)\int_{-(t + t_0)/a}^{-(t - t_0)/a}\frac{\text{d}u}{\cosh^2 u}.
\end{equation}
Finally, if $u \ge (t_0 - t)/a$, then $t + au \in [t_0, \infty)$. By the lemma's assumption, and applying  $\log(x + 1) \le x$ and $\log(x + 1) \le x - \frac{1}{2}x^2 + \frac{1}{3}x^3$ with $x = au/t > -1$, gives
\begin{align*}
\log |\zeta(\sigma + i(t + au))| &\le \log X + Y\log (t + au) + \log\log (t + au)\\
&\le \log(Xt^{Y}\log t) + \left(Y + \frac{1}{\log t}\right)\left(\frac{au}{t} - \frac{(au)^2}{2t^2} + \frac{(au)^3}{3t^3}\right)
\end{align*}
so that 
\begin{equation}\label{fordlem34_I4_bound}
I_4 \le \left[\log(Xt^{Y}\log t) + \left(Y + \frac{1}{\log t}\right)\left(\frac{au}{t} - \frac{(au)^2}{2t^2} + \frac{(au)^3}{3t^3}\right)\right]\int_{(t_0 - t)/a}^{\infty}\frac{\text{d}u}{\cosh^2 u}
\end{equation}
Combining \eqref{fordlem34_I1_bound}, \eqref{fordlem34_I2_bound}, \eqref{fordlem34_I3_bound} and \eqref{fordlem34_I4_bound}, we have 
\[
\int_{-\infty}^{\infty}\frac{\log|\zeta(\sigma + i(t + au))|}{\cosh^2 u}\text{d}u < 2\log(Xt^Y\log t) + E,
\]
\begin{equation}\label{fordlem34_E_defn}
\begin{split}
E &:= \log(t + t_0 + 1)\int_{-(t + t_0)/a}^{-(t - t_0)/a}\frac{\text{d}u}{\cosh^2 u}\\
&\qquad + \left(Y + \frac{1}{\log t}\right)\left(\int_{(t_0 - t)/a}^{\infty}\frac{\frac{au}{t} - \frac{(au)^2}{2t^2} + \frac{(au)^3}{3t^3}}{\cosh^2u}\text{d}u - \int_{-\infty}^{-2t/a}\frac{\frac{au + 2t}{t}}{\cosh^2 u}\text{d}u\right)
\end{split}
\end{equation}
However, since $-(t + t_0) / a < -(t - t_0)/a < 0$, 
\[
\int_{-(t + t_0)/a}^{-(t - t_0)/a}\frac{\text{d}u}{\cosh^2 u} \le \frac{2t_0}{a\cosh^2\left(\frac{t - t_0}{a}\right)}.
\]
Using $a \le 1/2$, $t \ge 3t_0$, $t_0 \ge 100$ and $\cosh^2 u \ge \frac{1}{4}e^{2|u|}$, the first term on the RHS of \eqref{fordlem34_E_defn} is majorised by 
\begin{equation}\label{s4_E_term_1}
\frac{t_0\log (t + t_0 + 1)}{a\cosh^2(\frac{t - t_0}{a})} \le \frac{4t_0\log(t + t_0 + 1)}{ae^{2(t - t_0)/a}} \le e^{-t/a}\frac{8t_0\log(4t_0 + 1)}{e^{2t_0}} < e^{-t/a}.
\end{equation}
Next, letting $v = -au - 2t$ and using $\cosh^2 u \ge \frac{1}{4}e^{2|u|}$,
\begin{equation}\label{s4_E_term_2}
\int_{-\infty}^{-2t/a}\frac{-\frac{au}{t} - 2}{\cosh^2 u}\text{d}u \le \int_{-\infty}^{-2t/a}\frac{-\frac{au}{t} - 2}{e^{-2u}}\text{d}u = \frac{4a}{t}e^{-4t/a}\int_0^{\infty}ve^{-2v}\text{d}v = \frac{a}{t}e^{-4t/a}.
\end{equation}
Finally, if $x = \frac{au}{t}$ and $u \ge \frac{t - t_0}{a}$ then $x \ge 1 - \frac{t_0}{t} \ge \frac{2}{3}$ and that $x + \frac{x^2}{2} + \frac{x^3}{3} \le \frac{10}{3}x^3$. Hence
\begin{equation}\label{s4_E_term_3}
\begin{split}
&\int_{(t_0 - t)/a}^{\infty}\frac{\frac{au}{t} - \frac{(au)^2}{2t^2} + \frac{(au)^3}{3t^3}}{\cosh^2u}\text{d}u \\
&\qquad\qquad\le \int_{-\infty}^{\infty}\frac{\frac{au}{t} - \frac{(au)^2}{2t^2} + \frac{(au)^3}{3t^3}}{\cosh^2u}\text{d}u + \int_{(t - t_0)/a}^{\infty}\frac{\frac{au}{t} + \frac{(au)^2}{2t^2} + \frac{(au)^3}{3t^3}}{\frac{1}{4}e^{2u}}\text{d}u\\
&\qquad\qquad\le -\frac{\pi^2}{12}\frac{a^2}{t^2} + \frac{40a^3}{3t^3}\int_{(t - t_0)/a}^{\infty}u^3e^{-2u}\text{d}u < -\frac{\pi^2}{12}\frac{a^2}{t^2} + 10^{-100}\frac{a^3}{t^3},
\end{split}
\end{equation}
where the last inequality follows from $(t - t_0) / a \ge 2t_0 / a \ge 400$ and evaluating the integral explicitly. Combining \eqref{s4_E_term_1}, \eqref{s4_E_term_2} and \eqref{s4_E_term_3}, we have 
\begin{equation}
E \le e^{-t/a} + \left(Y + \frac{1}{\log t}\right)\left(-\frac{\pi^2}{12}\frac{a^2}{t^2} + \frac{a}{t}e^{-4t/a} + 10^{-100}\frac{a^3}{t^3}\right) < e^{-t/a} - \frac{1}{\log t}\frac{a^2}{2t^2} < 0.
\end{equation}
\end{proof}

We use the above lemma in two forms - the first (Lemma \ref{fordlem34} below) is used for large $t$ and the second (Lemma \ref{fordlem34_1} below) is used for small values of $t$. 

\begin{lemma}\label{fordlem34_1}
Let $k \ge 4$ be an integer, $0 < a \le 1/2$ and $t \ge \max\{3\cdot 10^{12}, 2^k\}$. Then
\[
\frac{1}{2}\int_{-\infty}^{\infty}\frac{\log|\zeta(1 - \eta_k + it + iau)|}{\cosh^2 u}\text{d}u \le \frac{\log t}{2^{k} - 2} + \log\log t + \log 1.546.
\]
\end{lemma}
\begin{proof}
Follows by taking $t_0 = 10^{12}$, $\sigma = 1 - \eta_k$, $X = 1.546$ and $Y = 1/(2^k - 2)$ in Lemma \ref{fordlem34}, and using Theorem \ref{theorem1}. The condition that $\sigma \le 1 - t^{-1}$ is satisfied since for $k \ge 4$, $\eta_k\ge 2^{-k} \ge t^{-1}$. 
\end{proof} 

\begin{lemma}\label{fordlem34_2}
For all $2/7\le \eta \le 1/2$, $0 < a \le 1/2$ and $t \ge 3\cdot 10^{12}$, we have 
\[
\frac{1}{2}\int_{-\infty}^{\infty}\frac{\log|\zeta(1 - \eta + it + iau)|}{\cosh^2 u}\text{d}u \le \frac{8\eta - 1}{18}\log t + \log\log t + 1.659 - 4.279\eta.
\]
\end{lemma}
\begin{proof}
Let $s = \sigma + it$. We begin with the estimates 
\[
|(s - 1)\zeta(s)| \le 0.618|Q + s|^{7/6}\log |Q + s|, \qquad \Re s = 1/2, 
\]
\[
|(s - 1)\zeta(s)| \le 1.546|Q + s|^{15/14}\log |Q + s|,\qquad \Re s = 5/7
\]
with $Q = 1.31$. These bounds are verified numerically for $|t| \le 3$, then combined with \cite[Thm. 1.1]{hiary_improved_2022} and the case $k = 4$ in Theorem \ref{theorem1}, respectively. Applying Lemma \ref{PLPlem}, for all $1/2 \le \sigma \le 5/7$, 
\begin{equation}
|\zeta(s)| \le 0.618^{(10 - 14\sigma)/3}1.546^{(14\sigma - 7)/3}\frac{|Q + s|^{(25 - 8\sigma)/18}\log|Q + s|}{|s - 1|}.
\end{equation}
Next, we combine the estimates (for $1/2 \le \sigma \le 5/7$ and $t \ge t_0$)
\[
\frac{|Q + s|^{(25 - 8\sigma)/18}}{|s - 1|} < \frac{\left((Q + \sigma)^2 + t^2\right)^{(25 - 8\sigma)/36}}{t} \le \left(\left(\frac{Q + \frac{5}{7}}{t_0}\right)^2 + 1\right)^{7 / 12}t^{(7 - 8\sigma) / 18},
\]
\[
\log|Q + s| \le \log \sqrt{\frac{(Q + \sigma)^2}{t_0^2} + 1} + \log t \le \left(1 + \frac{1}{\log t_0}\log \sqrt{\frac{(Q + \frac{5}{7})^2}{t_0^2} + 1}\right)\log t
\]
to obtain
\begin{equation}
|\zeta(s)| \le  C_1 \, C_2^{\sigma}\, t^{(7 - 8\sigma) / 18}\log t,\qquad t \ge t_0, \frac{1}{2}\le \sigma \le \frac{5}{7}
\end{equation}
where 
\begin{equation}
C_1 := \frac{0.618^{10/3}}{1.546^{7/3}}\left(\left(\frac{Q + \frac{5}{7}}{t_0}\right)^2 + 1\right)^{7 / 12}\left(1 + \frac{1}{2\log t_0}\log \left(\frac{(Q + \frac{5}{7})^2}{t_0^2} + 1\right)\right),
\end{equation}
\begin{equation}
C_2 := \left(\frac{1.546}{0.618}\right)^{14/3}.
\end{equation}
Substituting $t_0 = 10^{12}$, $Q = 1.31$, $\sigma = 1 - \eta$ gives 
\[
\log|\zeta(1 - \eta + it)| \le 1.659 - 4.279\eta + \frac{8\eta - 1}{18}\log t + \log\log t, \qquad t \ge t_0.
\]
The result follows from applying Lemma \ref{fordlem34}. 
\end{proof}

Next, we slightly modify a lemma due to Ford \cite{ford_zero_2002}, which bounds $N(t, \eta)$, the number of zeroes $\rho$ satisfying $|1 + it - \rho| \le \eta$ (counted with multiplicity). The main change is to use sharper bounds on $\zeta(s)$ to the right of the 1-line, and to change the constants so that the lemma continues to hold for $1/4 \le \eta \le 2/7$. Note that a result like Lemma \ref{fordlem34} can be used to produce a sharper bound for small values of $t$, however we do not pursue this here for sake of brevity. 

\begin{lemma}\label{fordlem42}
For $0 < \eta \le 2/7$ and $t \ge 100$, we have 
\begin{equation}\label{fordlem42_equation}
N(t, \eta) \le 5.9975\eta^{3/2}\log t + 6.12 + \frac{\frac{2}{3}\log\log t - \log \eta}{1.879}. 
\end{equation}
\end{lemma}
\begin{proof}
We follow the argument of \cite[Lem.\ 4.2]{ford_zero_2002}. In place of \cite[(4.1)]{ford_zero_2002} we use a result appearing in \cite{bastien_rogalski_2002} that 
\begin{equation}\label{bastien_bound}
\zeta(\sigma) \le \frac{e^{\gamma(\sigma - 1)}}{\sigma - 1},\qquad \sigma > 1,
\end{equation}
where $\gamma = 0.5772\ldots$ is the Euler-Mascheroni constant, to obtain
\begin{equation}
\zeta(1 + 3.1421\eta) \le \frac{e^{3.1421\gamma \eta}}{3.1421\eta}.
\end{equation}
Note that we are using $\eta$ in place of $R$ in \cite[Lem. 4.2]{ford_zero_2002}. The constant 3.1421 is chosen so as to minimise the first term on the RHS of \eqref{fordlem42_equation}. Then, using \cite[Lem.\ 4.1]{ford_zero_2002} with $s = 1 + 0.6421\eta$,
\begin{align*}
-\frac{1}{0.6421\eta} &\le -\left|\frac{\zeta'}{\zeta}(1 + 0.6421\eta + it)\right| \le -\Re\frac{\zeta'}{\zeta}(1 + 0.6421\eta + it) \\
&\le -\frac{0.3758N(t, \eta)}{\eta} + \frac{1}{5\eta}\Big(\frac{2}{3}\log\log t + (1.8579\eta)^{3/2}B\log t \\
&\qquad\qquad + \log A + 3.1421\gamma \eta - \log 3.1421\eta\Big)
\end{align*}
However, for $\eta \le 2/7$ we have $3.1421\gamma \eta - \log 3.1421\eta \le -0.6267$, so the result follows from substituting $A = 76.2$ and $B = 4.45$ originating from \eqref{s4_richert_bound}.\footnote{Here we applied \eqref{s4_richert_bound} with $\sigma = 1 - 1.8579\eta$ where $0 < \eta \le 2/7$. Therefore, we need to extend the range of validity of this bound to $\sigma \ge 0.4691\ldots$. This is permissible by examining the proof of \cite[Lem.\ 7.1]{ford_vinogradov_2002} (in fact, \eqref{s4_richert_bound} is sharpest near $\sigma = 1$).}
\end{proof}

Next, we use the above lemma to bound a sum over zeroes away from $1 + it$, reproduced below. This result is the same as \cite[Lem 4.3]{ford_zero_2002}, except we use a sharper estimate of $S(t)$, due to Trudgian \cite{trudgian_improved_2014}, and extend the range of permissible values of $\eta$ up to $2/7$. We note that for large $t$, the estimate of $S(t)$ due to \cite{hasanalizade_counting_2021} is sharper, however our results are most sensitive to sharpness of the estimate for ``small" $t$. Using Theorem \ref{theorem1}, it is possible to further refine the arguments of \cite{trudgian_improved_2014} and \cite{hasanalizade_counting_2021} to improve bounds on $S(t)$, and thus improve Lemma \ref{fordlem43}. We leave such considerations for possible future work (see remarks in \S \ref{sec:concluding_remarks}). 

\begin{lemma}\label{fordlem43}
Let $t \ge 3\cdot 10^{12}$. If $0 < \eta \le 2/7$, then
\begin{equation*}
\begin{split}
\sum_{|1 + it - \rho| \ge \eta} \frac{1}{|1 + it - \rho|^2} &\le \left(\frac{23.99}{\sqrt{\eta}} - 40.385\right)\log t + \left(\frac{0.3548}{\eta^2} + 1.2031\right)\log\log t\\
&\qquad - 40.236 + \frac{5.86}{\eta^2} - \frac{\log \eta}{1.879\eta^{2}} - \frac{N(t, \eta)}{\eta^2}.
\end{split}
\end{equation*}
\end{lemma}
\begin{proof}
We follow the approach taken by Ford \cite[Lem. 4.3]{ford_zero_2002}. We divide the zeroes satisfying $|1 + it - \rho| \ge \eta$ into the following sets 
\begin{equation}\label{Zj_defn}
\begin{split}
Z_1 &:= \{\rho: |\Im \rho - t| \ge \delta\},\\
Z_2 &:= \{\rho: \rho \notin Z_1, |1 + it - \rho| \ge \eta_0, |it - \rho| \ge \eta_0\},\\ 
Z_3 &:= \{\rho: \rho \notin Z_1, \rho \notin Z_2, |1 + it - \rho| \ge \eta\}
\end{split}
\end{equation}
for some $\delta \ge \eta_0 > 0$ to be chosen later. We will also assume that $\delta \le 1$, this will be verified after choosing $\delta$. Let $N(t)$ denote the number of zeroes $\rho$ of $\zeta(s)$, with multiplicity, satisfying $0 < \Im \rho < t$. We use the following result, obtained by using \cite[Cor.\ 1]{platt_improved_2015} in place of \cite[Thm.\ 1]{trudgian_improved_2014} in the proof of \cite[Cor.\ 1]{trudgian_improved_2014}:
\footnote{The proof of \cite[Cor.\ 1]{platt_improved_2015} uses \cite[Thm.\ 1]{platt_improved_2015}, which in turn uses an erroneous lemma \cite[Lem.\ 3]{cheng_explicit_2004} (see e.g. \cite{patel_explicit_2021, patel_explicit_2022, ford_zero_2022, hiary_improved_2022} for a discussion). However, \cite[Thm.\ 1]{platt_improved_2015} has since been proved independently in \cite[Thm.\ 1.1]{hiary_improved_2022}, so we may use it here.}
\begin{equation}
\left|N(t) - \frac{t}{2\pi}\log\frac{t}{2\pi e} - \frac{7}{8}\right| \le 0.11\log t + 0.29\log\log t + 2.29 + \frac{0.2}{t_0},\qquad t \ge t_0 \ge e. 
\end{equation}
Therefore, for $t \ge 14$
\begin{equation}
N(t) = \frac{t}{2\pi}\log\frac{t}{2\pi e} + \frac{7}{8} + Q(t),
\end{equation}
for some $Q(t)$ satisfying
\begin{equation}
|Q(t)| \le 0.11\log t + 0.29\log\log t + 2.305.
\end{equation}
Let
\begin{equation}\label{Sj_defn}
S_j := \sum_{\rho \in Z_j}\frac{1}{|1 + it - \rho|^2},\qquad j = 1, 2, 3. 
\end{equation}
Then, since there are no zeroes with $|\Im \rho| \le 14$, and using $|1 + it - \rho| \ge |t - \Im \rho|$,
\[
S_1 \le \int_{t + \delta}^{\infty}\frac{\text{d}N(u)}{(u - t)^2} + \int_{14}^{t - \delta}\frac{\text{d}N(u)}{(u - t)^2} + \int_{14}^{\infty}\frac{\text{d}N(u)}{(u + t)^2} = I_1 + I_2 + I_3,
\]
where, since $\text{d}N(u) = \frac{1}{2\pi}\log \frac{u}{2\pi} + \text{d}Q(u)$
\begin{align*}
I_1 = \frac{1}{2\pi}\left[(\delta^{-1} + t^{-1})\log(t + \delta) - \frac{\log \delta}{t} - \frac{\log 2\pi}{\delta}\right] + \int_{\delta}^{\infty}\frac{\text{d}Q(t + x)}{x^2}
\end{align*}
and 
\begin{align*}
\int_{\delta}^{\infty}\frac{\text{d}Q(t + x)}{x^2} = \left[\frac{Q(t + x)}{x^2}\right]_{\delta}^{\infty} + 2\int_{\delta}^{\infty}\frac{Q(t + x)}{x^3}\text{d}x.
\end{align*}
Since $\log (1 + u) \le u$ if $u \ge 0$, we have that $\log (t + x) - \log t = \log (1 + \frac{x}{t}) \le \frac{x}{t}$ and 
\[
\log\log (t + x) - \log\log t = \log\left(\frac{\log(t + x)}{\log t}\right) = \log\left(1 + \frac{x}{t\log t}\right) < \frac{x}{t\log t}.
\]
Therefore, for $\delta \le 1$ and $t \ge 10000$,
\begin{align*}
|Q(t + \delta)| &\le 0.11 \log(t + \delta) + 0.29 \log \log(t + \delta) + 2.305\\
&\le 0.11\left(\log t + \frac{\delta}{t}\right) + 0.29 \left(\log\log t + \frac{\delta}{t\log t}\right) + 2.305\\
&\le 0.11\log t + 0.29\log\log t + 2.306.
\end{align*}
Furthermore,
\begin{align*}
\int_{\delta}^{\infty}\frac{|Q(t + x)|}{x^3}\text{d}x &\le \int_{\delta}^{\infty}\left(\frac{0.11\log t + 0.29\log\log t + 2.305}{x^3} + \left(\frac{0.11}{t} + \frac{0.29}{t\log t}\right)\frac{\delta}{x^2}\right)\text{d}x\\
&\le \frac{1}{2}\delta^{-2}(0.11\log t + 0.29\log\log t + 2.306).
\end{align*}
Finally, if $\delta \le 1$ then $\log \delta \le 0$ and for $t \ge t_0$, we have  
\[
\frac{1}{2\pi}\left[\left(\delta^{-1} + t^{-1}\right)\log(t + \delta) - \frac{\log \delta}{t} - \frac{\log 2\pi}{\delta}\right] \le \frac{\delta^{-1} + t_0^{-1}}{2\pi}\left(\log t + \frac{\delta}{t_0}\right) - \frac{\log\delta}{t_0} - \frac{\log 2\pi}{2\pi\delta}.
\]
Therefore 
\begin{equation}\label{s4_I1_bound}
I_1 \le A_1'\log t + B_1'\log \log t + C_1',
\end{equation}
\[
A_1' = 0.22\delta^{-2} + \frac{\delta^{-1} + t_0^{-1}}{2\pi},\quad B_1' = 0.58\delta^{-2},\quad C_1' = 4.612\delta^{-2} + \frac{1 + \delta/t_0}{2\pi t_0}  - \frac{\log\delta}{t_0} - \frac{\log 2\pi}{2\pi\delta}.
\]
Next,
\begin{align}
I_2 &= \frac{1}{2\pi}\int_{14}^{t - \delta}\frac{1}{(u - t)^2}\log \frac{u}{2\pi}\text{d}u + \int_{14}^{t - \delta}\frac{\text{d}Q(u)}{(u - t)^2} \notag\\
&\le \frac{1}{2\pi}\log \frac{t}{2\pi}\int_{14}^{t - \delta}\frac{\text{d}u}{(u - t)^2} + \left[\frac{Q(u)}{(u - t)^2}\right]_{14}^{t - \delta} + 2\int_{14}^{t - \delta}\frac{Q(u)}{(u - t)^3}\text{d}u \notag\\
&< \frac{\delta^{-1}}{2\pi}\log \frac{t}{2\pi} + \delta^{-2}|Q(t - \delta)| + 2|Q(t - \delta)|\int_{14}^{t - \delta}\frac{\text{d}u}{|u - t|^3}\qquad\text{(since $Q(14) > 0$)} \notag\\
&< \frac{\delta^{-1}}{2\pi}\log \frac{t}{2\pi} + \delta^{-2}(0.22 \log t + 0.58 \log \log t + 4.61) \notag\\
&= A_1''\log t + B_1''\log\log t + C_1'',\label{s4_I2_bound}
\end{align}
where 
\[
A_1'' = \frac{\delta^{-1}}{2\pi} + 0.22\delta^{-2},\qquad B_1'' = 0.58\delta^{-2},\qquad C_1'' = 4.61\delta^{-2} - \frac{\log 2\pi}{2\pi\delta}.
\]
Next, 
\begin{equation}\label{s4_I3_bound}
I_3 \le \frac{1}{2\pi}\int_{14}^{\infty}\frac{\log\left(\frac{u + t}{2\pi}\right)}{(u + t)^2}\text{d}u + 2\int_{14}^{\infty}\frac{|Q(u)|}{(u + t)^3}\text{d}u \le 0.00014. 
\end{equation}
Therefore, combining \eqref{s4_I1_bound}, \eqref{s4_I2_bound} and \eqref{s4_I3_bound}, 
\begin{equation}\label{s4_S1_bound}
S_1 \le A_1 \log t + B_1 \log\log t + C_1
\end{equation}
where 
\begin{align*}
A_1 &:= 0.44\delta^{-2} + \frac{2\delta^{-1} + t_0^{-1}}{2\pi},\qquad B_1 = 1.16\delta^{-2},
\end{align*}
\[
C_1 = 9.222\delta^{-2} + \frac{1 + \delta/t_0}{2\pi t_0}  - \frac{\log\delta}{t_0}  - \frac{\log 2\pi}{\pi\delta} + 0.00014.
\]
Next, if $\nu_0 := (\eta_0^{-2} + (1 - \eta_0)^{-2})/2$, then 
\begin{equation}\label{s4_S2_bound}
S_2 \le \max\left\{4|Z_2|, \nu_0|Z_2|\right\} = \nu_0|Z_2|.
\end{equation}
We also have 
\begin{align*}
|Z_2| + |Z_3| &= N(t + \delta) - N(t - \delta) - N(t, \eta) \\
&= \frac{t + \delta}{2\pi}\log\frac{t + \delta}{2\pi} - \frac{t - \delta}{2\pi}\log \frac{t - \delta}{2\pi}- \frac{\delta}{\pi} + Q(t + \delta) - Q(t - \delta) - N(t, \eta)
\end{align*}
If $f(t) := (t/2\pi)\log (t/2\pi)$, then, since $f'(t)$ is increasing,
\begin{align*}
&\frac{t + \delta}{2\pi}\log\frac{t + \delta}{2\pi} - \frac{t - \delta}{2\pi}\log \frac{t - \delta}{2\pi} = f(t + \delta) - f(t - \delta) \le 2\delta f'(t + \delta)\\
&\qquad\qquad\qquad\qquad = \frac{\delta}{\pi}\left(\log(t + \delta) - \log 2\pi + 1\right) < \frac{\delta}{\pi}\log t + \frac{\delta}{\pi}\left(\frac{\delta}{t_0} - \log 2\pi + 1\right)
\end{align*}
and, since $\log t$ and $\log\log t$ are both concave, we have $\log(t - \delta) + \log(t + \delta) \le 2\log t$ and $\log\log (t - \delta) + \log\log(t + \delta) \le 2\log\log t$ by Jensen's inequality. Hence
\begin{align*}
Q(t + \delta) - Q(t - \delta) &\le 0.22\log t + 0.58\log\log t + 4.61
\end{align*}
so that 
\begin{equation}\label{s4_Z2_Z3_bound}
|Z_2| + |Z_3| \le A'\log t + B'\log\log t + C' - N(t, \eta),
\end{equation}
where 
\[
A' = \frac{\delta}{\pi} + 0.22,\qquad B' = 0.58\qquad C' = \frac{\delta}{\pi}\left(\frac{\delta}{t_0} - \log 2\pi\right) + 4.61.
\]
Next, since $N_3 + N(t, \eta) = 2N(t, \eta_0)$, 
\begin{equation}\label{s4_S3_bound}
\begin{split}
S_3 &\le \frac{N(t, \eta_0)}{(1 - \eta_0)^2} + \int_{\eta}^{\eta_0}\frac{\text{d}N(t, u)}{u^2} = \frac{1}{(1 - \eta_0)^2}N(t, \eta_0) + \left[\frac{N(t, u)}{u^2}\right]_{\eta}^{\eta_0} + 2\int_{\eta}^{\eta_0}\frac{N(t, u)}{u^3}\text{d}u\\
&= 2\nu_0N(t, \eta_0) - \frac{N(t, \eta)}{\eta^2} + 2\int_{\eta}^{\eta_0}\frac{N(t, u)}{u^3}\text{d}u\\
&= \nu_0|Z_3| + \left(\nu_0 - \frac{1}{\eta^2}\right)N(t, \eta) + 2\int_{\eta}^{\eta_0}\frac{N(t, u)}{u^3}\text{d}u
\end{split}
\end{equation}
and 
\begin{align*}
2\int_{\eta}^{\eta_0}\frac{N(t, u)}{u^3}\text{d}u &\le 2\int_{\eta}^{\eta_0}\frac{5.9975u^{3/2}\log t + 6.12 + \frac{\frac{2}{3}\log\log t - \log u}{1.879}}{u^3}\text{d}u\\
&\le 23.99\left(\eta^{-1/2} - \eta_0^{-1/2}\right)\log t + \left(\eta^{-2} - \eta_0^{-2}\right)(5.86 + 0.3548\log\log t)\\
&\qquad\qquad + \frac{1}{1.879}\left(\frac{\log \eta_0}{\eta_0^2} - \frac{\log \eta}{\eta^2}\right).
\end{align*}
Therefore, combining \eqref{s4_S2_bound}, \eqref{s4_Z2_Z3_bound} and \eqref{s4_S3_bound},
\begin{equation}\label{s4_S2_S3_bound}
\begin{split}
S_2 + S_3 &\le \nu_0\left(A'\log t + B'\log\log t + C'\right) + 23.99\left(\eta^{-1/2} - \eta_0^{-1/2}\right)\log t \\
& + \left(\eta^{-2} - \eta_0^{-2}\right)(5.86 + 0.3548\log\log t) + \frac{1}{1.879}\left(\frac{\log \eta_0}{\eta_0^2} - \frac{\log \eta}{\eta^2}\right) - \frac{N(t, \eta)}{\eta^2}.
\end{split}
\end{equation}
Hence finally, combining \eqref{s4_S1_bound} and \eqref{s4_S2_S3_bound}, 
\[
\sum_{|1 + it - \rho| \ge \eta} \frac{1}{|1 + it - \rho|^2} \le A''\log t + B''\log\log t + C'' - \frac{N(t, \eta)}{\eta^2},
\]
where 
\[
A'' = 0.44\delta^{-2} + \frac{2\delta^{-1} + t_0^{-1}}{2\pi} + \nu_0\left(\frac{\delta}{\pi} + 0.22\right) + 23.99\left(\eta^{-1/2} - \eta_0^{-1/2}\right),
\]
\[
B'' = 1.16\delta^{-2} + 0.58\nu_0 + 0.3548\left(\eta^{-2} - \eta_0^{-2}\right),
\]
\begin{align*}
C'' &= 9.222\delta^{-2} + \frac{1 + \delta/t_0}{2\pi t_0} - \frac{\log \delta}{t_0} - \frac{\log 2\pi}{\pi\delta} + \nu_0\left(\frac{\delta}{\pi}\left(\frac{\delta}{t_0} - \log 2\pi\right) + 4.61\right)\\
&\qquad\qquad + 5.86\left(\eta^{-2} - \eta_0^{-2}\right) +  \frac{1}{1.879}\left(\frac{\log \eta_0}{\eta_0^2} - \frac{\log \eta}{\eta^2}\right) + 0.00014.
\end{align*}
We choose $t_0 = 3\cdot 10^{12}$, $\eta_0 = 2/7$, and $\delta = 0.90114$ to minimise the bound when $\eta = 2/7$, $t = \exp(1000)$, which are close to the values most relevant to our application. The desired result follows. 
\end{proof}

\begin{remark}
Lemma \ref{fordlem43} is ultimately used to bound a secondary term that is asymptotically smaller, as $t \to \infty$, than the ``main" term, which is bounded using Lemma \ref{fordlem34}. However, the values of $t$ we encounter are small enough that Lemma \ref{fordlem43} noticeably influences the resulting constant of Corollary \ref{corollary1}. For this reason we derive a second version of the result (Lemma \ref{fordlem43_1} below) that is sharper for ``large" values of $\eta$.
\end{remark}

\begin{lemma}\label{fordlem43_1}
If $t \ge 3\cdot 10^{12}$ and $2/7 \le \eta \le 1/2$, then 
\begin{equation*}
\begin{split}
\sum_{|1 + it - \rho| \ge \eta} \frac{1}{|1 + it - \rho|^2} &\le (0.5576 + 0.6079\nu)\log t + (0.7813 + 0.58\nu)\log\log t \\
&\qquad + 5.732 + 3.898\nu - \frac{N(t, \eta)}{\eta^2},
\end{split}
\end{equation*}
where $\nu := \frac{1}{2}(\eta^{-2} + (1 - \eta)^{-2})$.
\end{lemma}
\begin{proof}
For some $\delta \in (1, 2)$ to be chosen later (independent of the $\delta$ parameter appearing in Lemma \ref{fordlem43}), define
\begin{equation}
\begin{split}
Z_1' &:= \{\rho: |\Im \rho - t| \ge \delta\},\\
Z_2' &:= \{\rho: \rho \notin Z_1', |1 + it - \rho| \ge \eta, |it - \rho| \ge \eta\},\\ 
Z_3' &:= \{\rho: \rho \notin Z_1', |it - \rho| < \eta\}
\end{split}
\end{equation}
and let $S_j'$, $j = 1, 2, 3$ be defined analogously to \eqref{Sj_defn}. This is equivalent to setting $\eta_0 = \eta$ in \eqref{Zj_defn}. We use a similar method as before to bound $S_1'$, however, since $\delta$ is now greater than 1, we drop the $-\log\delta / t_0$ term in the bound of $S_1$ to ensure the monotonicity argument remains valid. Explicitly, we have 
\begin{equation}\label{s4_S1_prime_bound}
S_1' \le A_1' \log t + B_1' \log\log t + C_1'
\end{equation}
where 
\begin{align*}
A_1' &:= 0.44\delta^{-2} + \frac{2\delta^{-1} + t_0^{-1}}{2\pi},\qquad B_1' = 1.16\delta^{-2},
\end{align*}
\[
C_1' = 9.222\delta^{-2} + \frac{1 + \delta/t_0}{2\pi t_0} - \frac{\log 2\pi}{\pi\delta} + 0.00014.
\]
Next, similarly to before, 
\begin{equation}\label{s4_S2_prime_bound}
S_2' \le \max\left\{4|Z_2'|, \nu|Z_2'|\right\} = \nu|Z_2'|
\end{equation}
and 
\begin{equation}\label{s4_Z2_prime_bound}
\begin{split}
|Z_2'| &= N(t + \delta) - N(t - \delta) - 2N(t, \eta)\\
&\le A_2'\log t + B_2'\log\log t + C_2' - 2N(t, \eta),
\end{split}
\end{equation}
where 
\[
A_2' = \frac{\delta}{\pi} + 0.22,\qquad B_2' = 0.58\qquad C_2' = \frac{\delta}{\pi}\left(\frac{\delta}{t_0} - \log 2\pi\right) + 4.61.
\]
Furthermore, since each zero in $Z_3$ contributes at most $(1 - \eta)^{-2}$ each to the sum $S_3$, and there are $N(t, \eta)$ of them,
\begin{equation}\label{s4_S3_prime_bound}
S_3' \le \frac{N(t, \eta)}{(1 - \eta)^2}.
\end{equation}
Combining \eqref{s4_S1_prime_bound}, \eqref{s4_S2_prime_bound}, \eqref{s4_Z2_prime_bound} and \eqref{s4_S3_prime_bound},
\begin{equation}
\begin{split}
\sum_{|1 + it - \rho| \ge \eta} \frac{1}{|1 + it - \rho|^2} \le A''\log t + B''\log\log t + C'' - \frac{N(t, \eta)}{\eta^2}
\end{split}
\end{equation}
where 
\[
A'' = 0.44\delta^{-2} + \frac{2\delta^{-1} + t_0^{-1}}{2\pi} + \nu\left(\frac{\delta}{\pi} + 0.22\right), \qquad B'' = 1.16\delta^{-2} + 0.58\nu,
\]
\[
C'' = 9.222\delta^{-2} + \frac{1 + \delta/t_0}{2\pi t_0} - \frac{\log\delta}{t_0} - \frac{\log 2\pi}{\pi\delta} + 0.00014 + \nu\left(\frac{\delta}{\pi}\left(\frac{\delta}{t_0} - \log 2\pi\right) + 4.61\right).
\]
 Choosing $\delta = 1.2185$ and $t_0 = 3\cdot 10^{12}$ gives the desired result. 
\end{proof}

The following lemma, due to \cite[Lem.\ 4.6]{ford_zero_2002} using the methods of \cite[Lem.\ 5.1]{heath_brown_zero_1992}, provides an upper bound on a  Dirichlet series ``mollified" using a smoothing function $f$. 

\begin{lemma}[\cite{ford_zero_2002} Lemma 4.6]\label{fordlem46}
Let $0 \le \eta \le 1/2$, $f(u) \ge 0$ be a real function with a continuous derivative, and a Laplace transform $F(z) := \int_0^{\infty}f(y)e^{-zu}\text{d}u$ that is absolutely convergent for $\Re z > 0$. Let $F_0(z) := F(z) - z^{-1}f(0)$ and suppose 
\[
|F_0(z)| \le \frac{\delta}{|z|^2},\qquad \Re z \ge 0,\; |z| \ge \eta,
\]
for some constant $\delta > 0$. Then, if $s = 1 + it$ with $t \ge 1000$,
\begin{align*}
&\Re \sum_{n \ge 1}\frac{\Lambda(n)f(\log n)}{n^{s}} \le -\sum_{|s - \rho| \le \eta}\Re\left\{F(s - \rho) + f(0)\left(\frac{\pi}{2\eta}\cot\left(\frac{\pi(s - \rho)}{2\eta}\right) - \frac{1}{s - \rho}\right)\right\}\\
&\qquad + \frac{f(0)}{4\eta}\int_{-\infty}^{\infty}\frac{\log |\zeta(s - \eta + \frac{2\eta ui}{\pi})| - \log |\zeta(s + \eta + \frac{2\eta ui}{\pi})|}{\cosh^2 u}\text{d}u\\
&\qquad + \delta\left\{1.8 + \frac{\log t}{3} + \sum_{|s - \rho| \ge \eta}\frac{1}{|s - \rho|^2}\right\},
\end{align*}
and 
\begin{align*}
\Re\sum_{n \ge 1}\frac{\Lambda(n)f(\log n)}{n} \le F(0) + 1.8\delta. 
\end{align*}
\end{lemma}
\begin{proof}
Follows by combining \cite{ford_zero_2002} Lemma 4.5 and Lemma 4.6.
\end{proof}

\begin{remark}
As noted in \cite{mossinghoff_explicit_2022}, in the above lemma it is possible to take $t \ge t_0$ with $t_0$ much larger than 1000, however this has no noticeable effect in the final result. 
\end{remark}

The choice of the smoothing function $f$ in the above lemma has a direct impact on the size of the resulting zero-free region. Here, we choose $f$ the same way as Ford \cite{ford_zero_2002}, which is in turn based on \cite{heath_brown_zero_1992}. Such functions are derived via a calculus-of-variations argument and are provably optimal under certain assumptions. Jang and Kwon \cite{jang_note_2014} presented an alternative smoothing function that produced better results for the classical zero-free region (see also \cite{mossinghoff_explicit_2022}), however this smoothing function did not lead to significant improvements in Corollary \ref{corollary1}.

Given some $\lambda > 0$ to be fixed later, define
\begin{equation}\label{smoothing_f_defn}
f(u) := \lambda e^{\lambda u}g(\lambda u), \qquad u \ge 0,
\end{equation}
where, successively,
\begin{equation}
g(u) := (h*h)(u) = \int_0^{\infty}h(u - t)h(t)\text{d}t,\qquad u \ge 0,
\end{equation}
\begin{equation}
h(u) := \begin{cases}
(\cos(u\tan\theta) - \cos\theta)\sec^2\theta &\text{if }|u| \le \theta\cot\theta\\
0 &\text{otherwise}
\end{cases},
\end{equation}
and $\theta = \theta(b_0, b_1)$ is defined as the unique solution to 
\begin{equation}
\sin^2\theta = \frac{b_1}{b_0}(1 - \theta \cot \theta),\qquad 0 < \theta < \frac{\pi}{2},
\end{equation}
where $b_0$ and $b_1$ are coefficients of the non-negative trigonometric polynomial \eqref{poly_coefficients}. Via a direct computation, we have
\begin{equation}\label{theta_value}
\theta = \theta(b_0, b_1) = 1.132693699\ldots,
\end{equation}
\begin{equation}\label{f0_equation}
f(0) = \lambda g(0),\\
\end{equation}
\begin{equation}\label{w0_equation}
\begin{split}
g(0) &= \sec^2\theta(\theta \tan\theta + 3\theta \cot \theta - 3) = 5.610921922\ldots.
\end{split}
\end{equation}
We furthermore require information about the Laplace transforms of $f(u)$ and $g(u)$. Let
\begin{equation}\label{F_W_definition}
F(z) := \int_0^{\infty}e^{-zu}f(u)\text{d}u,\qquad G(z) := \int_0^{\infty}e^{-zu}g(u)\text{d}u
\end{equation}
denote the Laplace transforms of $f(u)$ and $g(u)$ respectively, and for convenience write
\begin{equation}
F_0(z) := F(z) - \frac{f(0)}{z},\qquad G_0(z) := G(z) - \frac{g(0)}{z}.
\end{equation}
The function $G_0(z)$ has an explicit formula (given in \cite[(7.3)]{ford_zero_2002} \footnote{Note that in \cite{ford_zero_2002}, $g(u)$, $G(z)$ and $G_0(z)$ appear as $w(u)$, $W(z)$ and $W_0(z)$ respectively. We have renamed these functions to prevent confusion with $W_0(x)$, the product log function, which appears later.}) in terms of $\theta$ which can be used to bound $|F_0(z)|$ via the relation
\begin{equation}
F_0(z) = G\left(\frac{z}{\lambda} - 1\right) - \frac{\lambda g(0)}{z} = G_0\left(\frac{z}{\lambda} - 1\right) + \frac{\lambda f(0)}{z(z - \lambda)}.
\end{equation}
Here we have used \eqref{smoothing_f_defn} and properties of the Laplace transform. Repeating the steps in \cite[Lem. 7.1]{ford_zero_2002}, we obtain, for all $\Re z \ge -1$ and $|z| \ge R \ge 3$, 
\begin{equation}
|F_0(z)| \le \frac{c\lambda f(0)}{|z|^2},
\end{equation}
where 
\begin{equation}\label{c_defn}
c := c(R) = \frac{H(R)(R + 1)^2}{R^3w(0)} + 1 + \frac{1}{R}
\end{equation}
and 
\begin{equation}\label{s4_H_defn}
H(R) := \frac{c_0}{(1 - \tan^2\theta / R^2)^2}\left(\frac{c_2(R + 1)}{R^3}(e^{2\theta/\tan\theta} + 1) + \frac{c_1}{R^2} + c_3\right),
\end{equation}
where, as in \cite{ford_zero_2002}, and using \eqref{theta_value},
\begin{equation}
\begin{split}
c_0 &:= \frac{1}{\sin\theta \cos^3\theta} = 14.464\ldots,\\
c_1 &:= (\theta - \sin\theta \cos\theta)\tan^4\theta = 15.541\ldots,\\
c_2 &:= \tan^3\theta\sin^2\theta = 7.9763\ldots,\\
c_3 &:= (\theta - \sin\theta\cos\theta)\tan^2\theta = 3.4108\ldots.
\end{split}
\end{equation}
We also have, via a direct computation using \eqref{theta_value} and \eqref{F_W_definition},
\begin{equation}\label{Wp0_defn}
G'(0) = \frac{1}{3}\csc^2\theta(3(4\theta^2 - 5) + \theta(15 - 4\theta^2)\cot\theta) - \theta\csc\theta\sec\theta \ge -0.659108.
\end{equation}
In the next lemma, we combine all of our above results with Lemma \ref{fordlem46} to form an explicit inequality involving the real and imaginary parts of a zero $\rho = \beta + it$.

\begin{lemma}\label{fordlem71}
Let $k \ge 4$ be an integer and $R \ge 3$. Suppose $\beta + it$ is a zero satisfying $t \ge 3\cdot 10^{12}$ and $1 - \beta \le \eta_k / 2$. Furthermore, suppose that there are no zeroes in the rectangle 
\[
1 - \lambda < \Re s \le 1,\qquad t - 1 \le \Im s \le Dt + 1
\]
where $\lambda$ is a constant satisfying $0 < \lambda \le \min\left\{1 - \beta , \eta_k / (R + 1)\right\}$. Then
\begin{align*}
&\frac{1}{\lambda}\left(0.17996 - 0.20523\left(\frac{1 - \beta}{\lambda} - 1\right)\right) \le 0.087\pi^2\frac{b_1}{b_0}\frac{1 - \beta}{\eta_k^2} \\
&\qquad\qquad + \frac{1}{2\eta_k}\left\{\frac{b}{b_0}\left(\frac{L_1 - 3.377}{2^k - 2} + L_2 + \log 1.546\right) + \log\zeta(1 + \eta_k) \right\}\\
&\qquad\qquad + c(R)\lambda \Bigg[\frac{b}{b_0}\,\bigg\{\left(\frac{23.99}{\sqrt{\eta_k}} - 40.051\right) (L_1 - 3.377) + 1.2031 L_2 - 38.58 \\
&\qquad\qquad\qquad\qquad + \frac{2.0373 L_2 - 0.5322\log \eta_k}{\eta_k^2}\bigg\} + 1.8\Bigg],
\end{align*}
with $c(R)$ defined in \eqref{c_defn}.
\end{lemma}
\begin{proof}
We largely follow the approach of \cite[Lem. 7.1]{ford_zero_2002} (see also \cite[Lem. 4.7]{mossinghoff_explicit_2022}). The main difference is instead of using the Ford--Richert bound \eqref{s4_richert_bound}, we use Theorem~\ref{theorem1}. We also use Lemma \ref{fordlem34} instead of \cite[Lem. 3.4]{ford_zero_2002} and Lemma \ref{fordlem43} instead of \cite[Lem. 4.3]{ford_zero_2002}. Furthermore, since we eventually apply this lemma with smaller values of $t$, more care is taken with bounding secondary error terms which can be significant for small $t$. Throughout, we retain the definitions of $\theta$, $f$, $g$, $F$, $G$, $F_0$, $G_0$ and $H(R)$ encountered previously. 

If $|z| \ge \eta_k$, then $|z| \ge (R + 1)\lambda$ via the lemma's assumptions. Hence
\begin{equation}
|F_0(z)| \le \frac{c(R)\lambda f(0)}{|z|^2},\qquad \Re z \ge 0, |z| \ge \eta_k.
\end{equation}
Therefore the conditions of Lemma \ref{fordlem46} are satisfied with $\delta = c(R) \lambda f(0)$ and $\eta = \eta_k$. Since $t \ge 1000$, we may apply Lemma \ref{fordlem46} with $s = s_j := 1 + ijt$ with $j = 0, 1, \ldots, D$. Summing the resulting expressions, and using the non-negativity of the trigonometric polynomial \eqref{poly_coefficients}, we obtain
\begin{align}
0 &\le \sum_{n = 1}^\infty\frac{\Lambda(n)}{n}f(\log n)\left[\sum_{j = 0}^Db_j\cos(jt\log n)\right]   \notag\\
&\le -\Re\sum_{j = 1}^Db_j\sum_{|s_j - \rho| \le \eta_k}\left\{F(s_j - \rho) + f(0)\left(\frac{\pi}{2\eta_k}\cot\left(\frac{\pi}{2\eta_k}(s_j - \rho)\right) - \frac{1}{s_j - \rho}\right)\right\}\notag\\
&\quad + \frac{f(0)}{4\eta_k}\sum_{j = 1}^{D}b_j\int_{-\infty}^{\infty}\frac{\log|\zeta(s_j - \eta_k + \frac{2\eta_k ui}{\pi})| - \log |\zeta(s_j + \eta_k + \frac{2\eta_k ui}{\pi})|}{\cosh^2 u}\text{d}u\notag\\
&\quad + b_0 F(0) + c(R)\lambda f(0)\left\{b\left(1.8 + \frac{\log t}{3}\right) + 1.8b_0 + \sum_{j = 1}^Db_j\sum_{|s_j - \rho| \ge \eta_k}\frac{1}{|s_j - \rho|^2}\right\}.\label{classic_bound0}
\end{align}
We now seek an upper bound for the right-hand side of \eqref{classic_bound0}. First, by \cite[Lem. 5.1]{ford_zero_2002}, we have
\begin{equation}\label{classic_bound_secondary_term}
-\frac{1}{2}\int_{-\infty}^{\infty}\frac{1}{\cosh^2 u}\sum_{j = 1}^D b_j \log\left|\zeta\left(1 + \eta_k + ijt + i\frac{2\eta_k u}{\pi}\right)\right|\text{d}u \le b_0\log\zeta(1 + \eta_k).
\end{equation}
Second, note that
\begin{equation}
\frac{1}{b}\sum_{j = 1}^Db_j\log(jt) \le \frac{1}{b}\sum_{j = 1}^Db_j\left(L_1 - \log \frac{D}{j}\right) \le L_1 - 3.377,
\end{equation}
so that, by Lemma \ref{fordlem34}, 
\begin{equation}\label{classic_bound_main_term}
\frac{1}{2}\sum_{j = 1}^Db_j\int_{-\infty}^{\infty}\frac{\log|\zeta(1 - \eta_k + ijt + 2\eta_k u i/\pi)|}{\cosh^2 u}\text{d}u \le b\left(\frac{L_1 - 3.377}{2^{k} - 2} + L_2 + \log 1.546\right).
\end{equation}
Third, by Lemma \ref{fordlem43}, and using $5.86 \le 1.6825 L_2$ for $t \ge 3\cdot 10^{12}$,
\begin{equation}\label{classic_bound1}
\begin{split}
&\sum_{j = 1}^Db_j\sum_{|s_j - \rho| \ge \eta_k}\frac{1}{|s_j - \rho|^2}\\
&\qquad\qquad\le \sum_{j = 1}^D b_j\bigg\{\left(\frac{23.99}{\sqrt{\eta_k}} - 40.385\right)\log (jt) + 1.2031\log\log (jt) -40.236\\
&\qquad\qquad\qquad\qquad + \frac{0.3548\log\log (jt) + 5.86 - 0.5322\log \eta_k}{\eta_k^2} - \frac{N(t, \eta_k)}{\eta_k^2}\bigg\}\\
&\qquad\qquad< b\Bigg\{\left(\frac{23.99}{\sqrt{\eta_k}} - 40.385\right) (L_1 - 3.377) + 1.2031L_2 - 40.236 \\
&\qquad\qquad\qquad\qquad + \frac{2.0373 L_2 - 0.5322\log \eta_k}{\eta_k^2}\Bigg\} - \frac{1}{\eta_k^2}\sum_{j = 1}^D b_j N(jt, \eta_k).
\end{split}
\end{equation}
Fourth, from \cite[(7.7)]{ford_zero_2002}, if $\lambda \le 1 - \Re \rho$, $\lambda \le \eta_k/(R + 1)$ and $|1 + ijt - \rho| \le \eta_k$, then
\begin{equation}\label{classic_bound2}
\begin{split}
&-\Re\left\{F(1 + ijt - \rho) + f(0)\left(\frac{\pi}{2\eta_k}\cot\left(\frac{\pi}{2\eta_k}(1 + ijt - \rho)\right) - \frac{1}{1 + ijt - \rho}\right)\right\} \\
&\qquad\qquad \le c' \pi^2 \lambda f(0)
\end{split}
\end{equation}
where 
\begin{equation}
c' := \frac{4}{\pi^2}\left(1 + \frac{(R + 1)^2H(R)}{w(0)R^3}\right) = \frac{4}{\pi^2}\left(c - \frac{1}{R}\right).
\end{equation}
We use \eqref{classic_bound2} for $j \ne 1$. If $j = 1$, then $1 + ijt - \rho = 1 - \beta$. Furthermore, 
\begin{equation}
\cot x - x^{-1} \ge -0.348x,\qquad 0 < x \le \frac{\pi}{4}
\end{equation}
and $1 - \beta \le \eta_k /2$ by assumption, so $\frac{\pi}{2\eta}(1 - \beta) \le \frac{\pi}{4}$ and 
\begin{equation}
\begin{split}\label{classic_bound4}
&-\Re\left\{F(1 +it - \rho) + f(0)\left(\frac{\pi}{2\eta_k}\cot\left(\frac{\pi}{2\eta_k}(1 + it - \rho)\right) - \frac{1}{1 + it - \rho}\right)\right\} \\
&\qquad\qquad\qquad\qquad \le F(1 - \beta) - 0.348\pi^2 f(0)(1 - \beta).
\end{split}
\end{equation}
Combining \eqref{classic_bound2} and \eqref{classic_bound4}, and noting that $c'\pi^2\lambda f(0)b_1 N(t, \eta_k) \ge 0$,
\begin{align}
&-\Re\sum_{j = 0}^Db_j\sum_{|1 + ijt - \rho| \le \eta_k}\bigg\{F(1 + ijt - \rho) \notag\\
&\qquad\qquad\qquad\qquad + f(0)\left(\frac{\pi}{2\eta_k}\cot\left(\frac{\pi}{2\eta_k}(1 + ijt - \rho)\right) - \frac{1}{1 + ijt - \rho}\right)\bigg\} \label{classic_bound3}\\
&\qquad\qquad \le -b_1\left[F(1 - \beta) - 0.348 \pi^2 f(0)(1 - \beta)\right] + c'\pi^2\lambda f(0)\sum_{j = 0}^D b_j N(jt, \eta_k). \notag
\end{align}
Finally we also have $\log t < L_1 - \log D$, so that 
\begin{equation}\label{classic_bound_100}
b\left(1.8 + \frac{\log t}{3}\right) \le b\left(\frac{L_1 - 3.377}{3} + 1.65\right).
\end{equation}
Substituting \eqref{classic_bound_secondary_term}, \eqref{classic_bound_main_term}, \eqref{classic_bound1}, \eqref{classic_bound3} and \eqref{classic_bound_100} into \eqref{classic_bound0}, we obtain
\begin{align}
0 &\le -b_1\left[F(1 - \beta) - 0.348 \frac{\pi^2}{4\eta_k^2} f(0)(1 - \beta)\right] + \left(c'\pi^2\lambda f(0) - \frac{c\lambda f(0)}{\eta_k^2}\right)\sum_{j = 1}^Db_j N(jt,\eta) \notag\\
&\qquad + f(0)\left[b\left(\frac{L_1 - 3.377}{2^k - 2} + L_2 + \log 1.546\right) + b_0\zeta(1 + \eta_k) \right] + b_0 F(0) \notag\\
&\qquad + c\lambda f(0)\Bigg[b\;\bigg\{\left(\frac{23.99}{\sqrt{\eta_k}} - 40.051\right) (L_1 - 3.377) + 1.2031L_2 - 38.58\notag\\
&\qquad\qquad\qquad + \frac{2.0373 L_2 - 0.5322\log \eta_k}{\eta_k^2}\bigg\} + 1.8b_0\Bigg].\label{classic_bound5}
\end{align}
However, for all $\eta_k \le 1/2$, 
\begin{equation}\label{classic_bound8}
c'\pi^2\lambda f(0) - \frac{c\lambda f(0)}{\eta_k^2} \le \left[4\left(c - \frac{1}{R}\right) - 4c\right]\lambda f(0) < 0
\end{equation}
so we may drop the sum in \eqref{classic_bound5}. Furthermore, 
\begin{align}
b_0F(0) - b_1F(1 - \beta) &= b_1\left(G(0) - G\left(\frac{1 - \beta}{\lambda} - 1\right)\right)- \left(b_1 G(0) - b_0 G(-1)\right) \notag\\
&= b_1\left(G(0) - G\left(\frac{1 - \beta}{\lambda} - 1\right)\right) - \frac{b_0 f(0)\cos^2\theta}{\lambda}.\label{classic_bound6}
\end{align}
Note that $G'(z)$ is continuous and decreasing for $z > 0$, so by the mean-value theorem, and using \eqref{Wp0_defn} followed by \eqref{f0_equation} and \eqref{w0_equation},
\begin{equation}\label{classic_bound7}
\begin{split}
G(0) - G\left(\frac{1 - \beta}{\lambda} - 1\right) \le \left(1 - \frac{1 - \beta}{\lambda}\right)G'(0) \le 0.659108\left(\frac{1 - \beta}{\lambda} - 1\right)\\
\le \frac{0.11747b_0f(0)}{\lambda}\left(\frac{1 - \beta}{\lambda} - 1\right).
\end{split}
\end{equation}
Combining with \eqref{classic_bound6}, \eqref{classic_bound8} and \eqref{classic_bound7} with \eqref{classic_bound5}, we have 
\begin{align}
&\frac{b_0 f(0)\cos^2\theta}{\lambda} - \frac{0.11747b_1b_0f(0)}{\lambda}\left(\frac{1 - \beta}{\lambda} - 1\right) \le b_1F(1 - \beta) - b_0F(0) \notag\\
&\qquad \le \frac{f(0)}{2\eta_k}\bigg[b\left(\frac{L_1 - 3.377}{2^k - 2} + L_2 + \log 1.546\right) + b_0\log\zeta(1 + \eta_k) \bigg] \\
&\qquad\qquad+ c(R)\lambda f(0)\Bigg[b\;\bigg\{\left(\frac{23.99}{\sqrt{\eta_k}} - 40.051\right) (L_1 - 3.377) + 1.2031 L_2 - 38.58 \notag\\
&\qquad\qquad\qquad + \frac{2.0373 L_2 - 0.5322\log \eta_k}{\eta_k^2}\bigg\} + 1.8b_0\Bigg] + 0.348b_1\frac{\pi^2}{4\eta_k^2} f(0)(1 - \beta).\notag
\end{align}
Dividing through by $b_0f(0)$, the result follows. 
\end{proof}

We now proceed to the main proof of Corollary \ref{corollary1}, which is similar to the proof of \cite[Thm. 2]{ford_zero_2002}. Throughout, we take $t_0 := \exp(170.2)$ and $t_1 := \exp(967.6)$. If $3 \le t \le H := 5.45 \cdot 10^{8}$, then Corollary \ref{corollary1} follows from the partial verification of the Riemann Hypothesis up to height $H$, performed independently in e.g.\ \cite{van_de_lune_zeros_1986, wedeniwski_zetagrid_2003, gourdon_computation_2004, platt_isolating_2017, platt_riemann_2021}. Note that in \cite{platt_riemann_2021}, the Riemann Hypothesis is actually verified up to height $3\cdot 10^{12}$, however our results do not require this, and the lower value of $H$ has been independently verified by multiple authors. If $H \le t < t_0$, then Corollary \ref{corollary1} follows from \eqref{ford_classic_zfr}. Assume now that $t \ge t_0$. Throughout, let  
\begin{equation}
Z(\beta, t) := (1 - \beta)\frac{\log t}{\log\log t},\qquad M := \inf_{\substack{\zeta(\beta + it) = 0,\\t \ge t_0}}Z(\beta, t),
\end{equation}
\begin{equation}
\alpha := 0.13913,\qquad M_1 := 0.0470978.
\end{equation}
For integer $k \ge 4$ let 
\begin{equation}
T_k := \left(2^{k2^k}\right)^{1/\alpha}\qquad\text{and}\qquad \eta_k := \frac{k}{2^k - 2}
\end{equation}
so that 
\begin{equation}\label{s4_k_defn}
k = \frac{W_0(\alpha \log T_k)}{\log 2},
\end{equation}
where $W_0(x)$ is the principal branch of the Lambert $W$ function, which satisfies $W_0(x)e^{W_0(x)} = x$ for any $x > 0$. 

Note that $M \ge M_1$ implies Corollary \ref{corollary1}. Assume for a contradiction that $M < M_1$. Then, there exists a zero $\beta + it$ for which $Z(\beta, t)$ is arbitrarily close to $M$. In particular, let $\beta + it$ be a zero such that $t \ge H$ and satisfying 
\begin{equation}\label{Z_assumption_zfr}
M \le Z(\beta, t) \le M(1 + \delta),\qquad \delta := \min\left\{\frac{10^{-100}}{\log t_0}, \frac{M_1 - M}{2M}\right\}.   
\end{equation}
In particular, we have $Z(\beta, t) \le M_1$. If we let 
\begin{equation}
\lambda := \frac{ML_2}{L_1},
\end{equation}
then there are no zeroes of $\zeta(s)$ in the rectangle 
\begin{equation}\label{zerofreerect}
1 - \lambda < \Re s \le 1,\qquad t - 1 \le \Im s \le Dt + 1.
\end{equation}
This is because if a zero $\beta' + it'$ exists in that rectangle, then as $\log\log x / \log x$ is decreasing for $x > e^e$, and $e^e < t' < Dt' + 1$, 
\begin{equation}
1 - \beta' < \lambda \le \frac{M\log\log t'}{\log t'}
\end{equation}
so that $Z(\beta', t') < M$. If $t' \ge t_0$, then this is in contradiction to the definition of $M$. On the other hand if $t_0 - 1 \le t' < t_0$, then by \eqref{ford_classic_zfr} and a short numerical computation,
\[
1 - \beta' \le \frac{0.04962 - 0.0196/(J(t_0 - 1) + 1.15)}{J(t_0 - 1) + 0.685 + 0.155\log\log (t_0 - 1)} < \frac{M_1\log\log t_0}{\log t_0} < \frac{M_1\log\log t'}{\log t'}.
\]
Thus, $Z(\beta', t') \ge M_1 > M$, which is a contradiction. Therefore, \eqref{zerofreerect} is zero-free, and any zero $\beta + it$ must satisfy
\begin{equation}\label{lambda_condition_zfr}
\lambda \le 1 - \beta. 
\end{equation}
We now divide our argument into two parts. For $t_0 \le t < t_1$, our argument relies on a uniform bound on $\zeta(s)$ for $1/2 \le \Re s \le 5/7$. Here, we use Lemma \ref{fordlem34_2} and Lemma \ref{fordlem43_1}. These tools allow us to derive a sharp estimate of the zero-free region over a finite range of $t$ close to $t_0$, however are insufficient to obtain an asymptotic estimate of the required strength. Therefore, for $t \ge t_1$ we switch to bounds on $\zeta(\sigma_k + it)$, where we take $k\to\infty$ as $t\to\infty$, in the form of Lemma \ref{fordlem71}.

First, consider the case $t \ge t_1$. Since $T_5 = \exp(797.1\ldots) < t_1$, there is always an integer $k \ge 5$ such that $t \in [T_k, T_{k + 1})$. Then, since $(2^x - 2)/x$ is increasing for all $x > 0$, and $W_0(\alpha \log t)/\log 2 \ge k > 0$, we have, for $t \ge t_1$,
\begin{equation}\label{eta_inverse_main_bound_zfr}
\frac{1}{\eta_k} = \frac{2^k - 2}{k} \le \frac{\exp(W_0(\alpha\log t)) - 2}{W_0(\alpha\log t) / \log 2} = A_0(\alpha\log t) \frac{\log t}{(\log\log t)^2}
\end{equation}
where, since $(\log\log t)^2 = (\log(\alpha \log t) - \log \alpha)^2$,
\begin{equation}
A_0(x) := \frac{\alpha\log 2(\log x - \log \alpha)^2}{W_0(x)^2}\left(1 - \frac{2W_0(x)}{x}\right).
\end{equation}
Note that 
\begin{equation}\label{A_1_deriv_defn}
A_0'(x) = \frac{2\alpha \log 2 \log(x / \alpha)}{x^2 W_0(x)^2 (W_0(x) + 1)}A_1(x)
\end{equation}
where 
\begin{equation}
A_1(x) := (W_0(x)^2 + 2 W_0(x) - x)\log\frac{x}{\alpha} - 2 W_0(x)^2 - (2 - x) W_0(x) + x.
\end{equation}
Since $W_0(x) = \log x - \log\log x + o(1)$, we have $A_1(x) = -x\log\log x + O(x)$, i.e.\ $A_1(x) < 0$ for sufficiently large $x$. We in fact find computationally that $A_1(x) < 0$ for all $x > x^* = 15.832\ldots$. It follows from \eqref{A_1_deriv_defn} that if $x_0 := \max\{x^*, \alpha\log t_1\}$, then
\begin{equation}\label{A0_estimate}
A_0(\alpha \log t) \le A_0(x_0) \le 0.3297, \qquad t \ge t_1.
\end{equation}
This also implies that 
\begin{equation}\label{eta_inverse_bound_1_zfr}
\begin{split}
\frac{1}{\eta_k} < C_1\frac{\log t}{\log\log t},\qquad C_1 := \frac{A_0(x_0)}{\log\log t_0} \le 0.047958.
\end{split}
\end{equation}
Then
\begin{equation}\label{R_constraint_zfr}
\lambda \le \frac{M\log\log t}{\log t} \le 0.047958 M_1\eta_k \le \frac{\eta_k}{442.729},
\end{equation}
i.e.\ $\lambda \le \eta_k / (R + 1)$ with $R = 441.729$. This allows us to compute, via \eqref{c_defn} and \eqref{s4_H_defn},
\begin{equation}
c(R) \le 1.02268. 
\end{equation}
Collecting \eqref{lambda_condition_zfr} and \eqref{R_constraint_zfr}, all conditions of Lemma \ref{fordlem71} are met with $\eta = \eta_k$, $R = 441.729$. We thus have, for all zeroes $\beta + it$ such that $T_k \le t < T_{k + 1}$,
\begin{equation}\label{s4_main_inequality}
\begin{split}
&\frac{1}{\lambda}\left(0.17996 - 0.20523\left(\frac{1 - \beta}{\lambda} - 1\right)\right) \le 0.087\pi^2\frac{b_1}{b_0}\frac{1 - \beta}{\eta_k^2} \\
&\qquad + \frac{1}{2\eta_k}\left\{\frac{b}{b_0}\left[\frac{L_1 - 3.377}{2^k - 2} + L_2 + \log 1.546\right] + \log\zeta(1 + \eta_k) \right\}\\
&\qquad + c(R)\lambda \Bigg[\frac{b}{b_0}\,\bigg\{\left(\frac{23.99}{\sqrt{\eta_k}} - 40.051\right) (L_1 - 3.377) + 1.2031 L_2 - 38.58 \\
&\qquad\qquad\qquad\qquad + \frac{2.0373 L_2 - 0.5322\log \eta_k}{\eta_k^2}\bigg\} + 1.8\Bigg].
\end{split}
\end{equation}
We now proceed to bound each term appearing on the right hand side. First, 
\begin{equation}\label{W_ratio_bound_zfr}
\frac{W_0(\alpha\log t)}{W_0(\alpha\log T_k)} < \frac{W_0(\alpha\log T_{k + 1})}{W_0(\alpha\log T_k)} = \frac{k + 1}{k}.
\end{equation}
Furthermore, $\log x / W_0(\alpha x)$ is decreasing for $x > x_1 := \max\{e^{1 + \alpha e}, \log t_1\}$, since 
\begin{equation}
\frac{\text{d}}{\text{d}x}\frac{\log x}{W_0(\alpha x)} = \frac{W_0(\alpha x) - \log x + 1}{xW_0(\alpha x)(W_0(\alpha x)+ 1)}. 
\end{equation}
Taking $x = \log t$ and combining with \eqref{W_ratio_bound_zfr} and \eqref{s4_k_defn} gives
\begin{equation}\label{leading_term_bound_zfr}
\begin{split}
\frac{1}{\eta_k}\frac{L_1 - 3.377}{2^k - 2} &= \frac{(L_1 - 3.377)\log 2}{W_0(\alpha \log T_k)} < \frac{(L_1 - 3.377)(1 + k^{-1})\log 2}{W_0(\alpha\log t)} \\
&= (1 + k^{-1})\log 2\left(\frac{L_1 - 3.377}{\log t}\right)\frac{\log\log t}{W_0(\alpha\log t)}\frac{\log t}{\log\log t} \le C_2\frac{\log t}{\log\log t},
\end{split}
\end{equation}
where, since $L_1 - 3.377 \le \log t + \log(D + t_1^{-1}) - 3.377$ and $\log(D + t_1^{-1}) - 3.377 > 0$, 
\begin{equation}
C_2 := \frac{6}{5}\log 2\left(1 + \frac{\log(D + t_1^{-1}) - 3.377}{\log t_1}\right)\frac{\log x_1}{W_0(\alpha x_1)} \le 1.58176.
\end{equation}
Next, using \eqref{eta_inverse_main_bound_zfr}, we have 
\begin{equation}\label{L2_over_eta_bound_zfr}
\begin{split}
\frac{L_2}{\eta_k} &< \left(\frac{A_0(x_0)L_2}{\log\log t}\right)\frac{\log t}{\log\log t} \le C_3\frac{\log t}{\log\log t},
\end{split}
\end{equation}
where 
\begin{equation}
C_3 := A_0(x_0)\left(1 + \frac{\log\left(1 + \log(D + t_1^{-1})/\log t_1\right)}{\log\log t_1}\right) \le 0.32989.
\end{equation}
Next, since $-x^{-1}\log x$ is decreasing for $x < 1$, by \eqref{eta_inverse_main_bound_zfr} we have 
\begin{align}
-\frac{\log \eta_k}{\eta_k} &\le -\frac{\log[(\log\log t)^2/(A_0(x_0)\log t)]}{(\log\log t)^2 / (A_0(x_0)\log t)} \notag\\
&= A_0(x_0)\left[\frac{\log A_0(x_0) + \log\log t - 2\log\log\log t}{\log\log t}\right]\frac{\log t}{\log\log t} < C_4\frac{\log t}{\log \log t},\label{log_eta_over_eta_bound_zfr}
\end{align}
where 
\begin{equation}
C_4 := A_0(x_0)\left(\frac{\log A_0(x_0)}{\log\log t_1} + 1\right) \le 0.27649.
\end{equation}
Finally, using \eqref{bastien_bound} with $\sigma = 1 + \eta_k$, we have 
\begin{equation}\label{log_zeta_bound_zfr}
\log\zeta(1 + \eta_k) \le \gamma\eta_k - \log \eta_k.
\end{equation}
Combining \eqref{eta_inverse_bound_1_zfr}, \eqref{leading_term_bound_zfr}, \eqref{L2_over_eta_bound_zfr}, \eqref{log_eta_over_eta_bound_zfr} and \eqref{log_zeta_bound_zfr},
\begin{equation}
\frac{1}{2\eta_k}\left\{\frac{b}{b_0}\left(\frac{L_1 - 3.377}{2^k - 2} + L_2 + \log 1.546\right) + \log\zeta(1 + \eta_k)\right\} \le C_5 \frac{\log t}{\log\log t}
\end{equation}
where, from \eqref{poly_coefficients} we have $b = 3.57440943022073$, $b_0 = 1$, and
\begin{equation}\label{s4_zfr_C5_bound}
C_5 := \left[\frac{C_1\log 1.546}{2} + \frac{C_2}{2} + \frac{C_3}{2}\right]\frac{b}{b_0} + \frac{C_4}{2} + \frac{\gamma}{2}\frac{\log\log t_0}{\log t_0} \le 3.59415.
\end{equation}
This constitutes the ``main" term. Next, from \eqref{eta_inverse_main_bound_zfr} and using \eqref{Z_assumption_zfr},
\begin{equation}\label{s4_zfr_C6_bound}
\begin{split}
\frac{0.087\pi^2b_1}{b_0}\frac{1 - \beta}{\eta_k^2} &\le 1.5002\frac{Z(\beta, t) \log \log t}{\log t}\cdot \left(\frac{A_0(x_0)\log t}{(\log\log t)^2}\right)^2 \le C_6\frac{\log t}{\log\log t}
\end{split}
\end{equation}
where 
\begin{equation}
C_6 := \frac{1.5002A_0(x_0)^{2}M_1}{(\log \log t_1)^2} \le 0.00017.
\end{equation}
The remaining term of \eqref{s4_main_inequality} is majorised by
\begin{equation}\label{s4_last_term}
\begin{split}
&c(R)M\bigg\{\frac{b}{b_0}\bigg\{\left(\frac{23.99}{\sqrt{\eta_k}} - 40.051\right)\frac{L_2^2}{L_1} + (1.2031L_2 - 38.58)\frac{L_2^2}{L_1^2} \\
&\qquad\qquad\qquad + \frac{2.0373 L_2 - 0.5322\log \eta_k}{\eta_k^2}\frac{L_2^2}{L_1^2}\bigg\} + 1.8\frac{L_2^2}{L_1^2}\bigg\}\frac{L_1}{L_2}.
\end{split}
\end{equation}
We have, by \eqref{eta_inverse_main_bound_zfr}, for $t \ge t_1$, 
\begin{equation}
\frac{1}{\sqrt{\eta_k}}\frac{L_2^2}{L_1} \le A_0(x_0)^{1/2}\frac{\log^{1/2} t}{\log \log t}\frac{L_2^2}{L_1} < A_0(x_0)^{1/2}\frac{L_2}{L_1^{1/2}}
\end{equation}
so that 
\begin{equation}
\begin{split}
\left(\frac{23.99}{\sqrt{\eta_k}} - 40.051\right)\frac{L_2^2}{L_1} + (1.2031L_2 - 38.58)\frac{L_2^2}{L_1^2} \le U(L_1),
\end{split}
\end{equation}
where, using \eqref{A0_estimate},
\begin{equation}
U(x) := 13.775\frac{\log x}{x^{1/2}} - 40.051\frac{\log^2 x}{x}+ (1.2031\log x - 38.58)\frac{\log^2 x}{x^2}.
\end{equation}
By solving for stationary points explicitly, we find that for $x \ge L_1(t_1)$, 
\begin{equation}
U(x) \le C_7 := 1.18399.
\end{equation}
Meanwhile, we also have the estimates 
\begin{equation}\label{C10_bound}
\frac{1}{\eta_k^2}\frac{L_2^3}{L_1^2} \le A_0(x_0)^2 \frac{\log^2 t}{(\log\log t)^4}\frac{L_2^3}{L_1^2} \le C_{8}, 
\end{equation}
where, since $L_1 > \log t$ and $L_2 < \log\log t + \frac{\log(N + t_1^{-1})}{\log t_1}$,
\begin{equation}
C_{8} := A_0(x_0)^2\left(1 + \frac{\log (N + t_1^{-1}) / \log t_1}{\log\log t_1}\right)^3\frac{1}{\log\log t_1} \le 0.01584,
\end{equation}
and, since $-x^{-2}\log x$ is increasing and $x^{-2}\log^2 x$ is decreasing for $x \ge e$, 
\begin{equation}
-\frac{\log\eta_k}{\eta_k^2}\frac{L_2^2}{L_1^2} \le -\frac{\log[(\log\log t)^2/(A_0(x_0)\log t)]}{(\log\log t)^4 / (A_0(x_0)\log t)^2}\frac{(\log\log t)^2}{\log^2 t} \le C_{8}\frac{L_1}{L_2}
\end{equation}
where 
\begin{equation}
C_{9} := \frac{A_0(x_0)^2}{\log\log t_1}\left(1 + \frac{\log A_0(x_0)}{\log\log t_1}\right)\frac{L_2(t_1)}{L_1(t_1)} \le 0.0001.
\end{equation}
Next, we have $L_2^2 / L_1^2 \le C_{10} := 0.00006$ for $t \ge t_1$. Combining this with the bound for $U(x)$ and \eqref{C10_bound}, we have that \eqref{s4_last_term} is $\le C_{11}L_1 / L_2$, where 
\begin{equation}\label{s4_zfr_C12_bound}
C_{11} := c(R) M_1 \left(\frac{b}{b_0}\left(C_7 + 2.0373C_{8} + 0.5322C_{9}\right) + 1.8C_{10}\right) \le 0.20942.
\end{equation}
Lastly, applying \eqref{Z_assumption_zfr} we have
\begin{equation}\label{s4_zfr_lhs_bound}
\begin{split}
\frac{1 - \beta}{\lambda} - 1 &= \frac{Z(\beta, t)\log \log t}{\log t}\frac{L_1}{ML_2} - 1 < (1 + \delta)\frac{L_1}{\log t} - 1 \le \frac{C_{12}}{\log t_1},
\end{split}
\end{equation}
where, since $t \ge t_1$,
\begin{equation}\label{C13_bound}
C_{12} := 10^{-100}\left(1 + \frac{\log(D + t_1^{-1})}{\log t_1}\right) + \log(D + t_1^{-1}) \le 3.82865.
\end{equation}
Combining \eqref{s4_main_inequality}, \eqref{s4_zfr_C5_bound}, \eqref{s4_zfr_C6_bound}, \eqref{s4_zfr_C12_bound}, \eqref{s4_zfr_lhs_bound} and \eqref{C13_bound}, and from the definition of $M_1$, 
\begin{equation}
M = \lambda \frac{L_1}{L_2} \ge \frac{0.17996 - 0.7857/\log t_1}{C_5 + C_6 + C_{11}} \ge 0.04709785 > M_1,
\end{equation}
so the desired contradiction is reached if $t \ge t_1$. 

For $t_0 \le t < t_1$, we use a similar argument, except with Lemma \ref{fordlem34_2} and Lemma \ref{fordlem43_1}. Here, we choose
\begin{equation}
\eta := \eta(t) = \left(8 - \frac{E}{L_2(t)}\right)^{-1},\qquad E := 30.95461.
\end{equation}
These parameters guarantee that, for all $t_0 \le t \le t_1$, we have $2/7 \le \eta(t) \le 1/2$. Observe that for $t \ge t_0$, we have 
\[
\frac{L_2}{\log\log t} \le 1 + \frac{\log(D + t_0^{-1})}{\log t_0\log\log t_0} \le 1.0044
\]
so that 
\begin{equation}
\left(8 - \frac{E}{L_2}\right)\frac{\log\log t}{\log t} \le \frac{8\log\log t - E / 1.0044}{\log t} \le 0.06039.
\end{equation}
Here, the last inequality follows from solving explicitly the maximum of $(8\log x - E/1.0044)/x$ for $x \ge \log t_0$. Therefore, we have 
\[
\frac{1}{\eta} \le 0.06039\frac{\log t}{\log\log t}\qquad\implies\qquad \lambda \le 0.06039 M_1\eta \le \frac{\eta}{351.588},
\]
so that $\lambda \le \eta / (R' + 1)$ with $R' = 350.588$. This implies that $c(R') \le 1.0288$.  

Since $2/7 \le \eta(t) \le 1/2$ for all $t_0 \le t \le t_1$, we may apply Lemma \ref{fordlem34_2} and \ref{fordlem43_1}. We use these in place of Lemma \ref{fordlem34_1} and \ref{fordlem43} respectively, to obtain, via the same argument as Lemma \ref{fordlem71},
\begin{equation}\label{s4_main_inequality_1}
\begin{split}
&\frac{1}{\lambda}\left(0.17996 - 0.20523\left(\frac{1 - \beta}{\lambda} - 1\right)\right) \le 0.087\pi^2\frac{b_1}{b_0}\frac{1 - \beta}{\eta^2} \\
&\qquad + \frac{1}{2\eta}\left\{\frac{b}{b_0}\left(\frac{8\eta - 1}{18}(L_1 - 3.377) + L_2 + 1.659 - 4.279\eta\right) + \log\zeta(1 + \eta) \right\}\\
&\qquad + c(R')\lambda\Bigg[\frac{b}{b_0}\bigg\{(0.891 + 0.6079\nu)(L_1 - 3.377) + (0.7813 + 0.58\nu)L_2 \\
&\qquad\qquad\qquad + 7.329 + 3.898\nu\bigg\} + 1.8\Bigg],
\end{split}
\end{equation}
where, as before, $\nu := \frac{1}{2}(\eta^{-2} + (1 - \eta)^{-2})$. 
First, analogously to before, we have
\begin{equation}\label{s4_H_smol_term_bound}
\begin{split}
\frac{0.087\pi^2b_1}{b_0}\frac{1 - \beta}{\eta^2} &\le 1.5002\frac{Z(\beta, t) \log \log t}{\log t}\cdot \left(\frac{0.06039\log t}{\log\log t}\right)^2 \le \frac{0.00015L_1}{L_2}.
\end{split}
\end{equation}
Next, since $\eta \ge 2/7$ and by \eqref{bastien_bound},
\begin{equation}
\frac{\log\zeta(1 + \eta)}{2\eta} \le \frac{\gamma}{2} - \frac{\log \eta}{2\eta} \le 2.481.
\end{equation}
Hence 
\begin{align}
&\frac{1}{2\eta}\left\{\frac{b}{b_0}\left(1.659 - 4.279\eta + \frac{8\eta - 1}{18}(L_1 - 3.377) + L_2 \right) + \log\zeta(1 + \eta) \right\} \notag\\
&\qquad\qquad\le \frac{b}{2b_0}\left\{1.847\left(8 - \frac{E}{L_2}\right) - 5.779 + \frac{E}{18}\frac{L_1}{L_2} + L_2\left(8 - \frac{E}{L_2}\right)\right\} + 2.481 \notag\\
&\qquad\qquad\le \left(3.07346 + 14.298\frac{L_2^2}{L_1} - 36.761\frac{L_2}{L_1} - \frac{102.18}{L_1}\right)\frac{L_1}{L_2} \notag\\
&\qquad\qquad\le 3.59852\frac{L_1}{L_2},\label{s4_H_main_term_bound}
\end{align}
where the last inequality follows from 
\begin{equation}
\frac{14.298\log^2 x - 36.761\log x - 102.18}{x} \le 0.52506,\qquad x \ge L_1(t_0).
\end{equation}
Next, note that since $\eta \le 1/2$,
\[
\nu = \frac{1}{2}\left(\frac{1}{\eta^2} + \frac{1}{(1 - \eta)^2}\right) \le \frac{1}{2}\left(\left(8 - \frac{E}{L_2}\right)^2 + 4\right) = 34 - \frac{8E}{L_2} + \frac{E^2}{2L_2^2}
\]
so that 
\begin{equation}
\nu L_2 \le \left(\frac{34L_2^2 - 8EL_2 + \frac{E^2}{2}}{L_1}\right)\frac{L_1}{L_2} \le B_2\frac{L_1}{L_2}
\end{equation}
where 
\begin{equation}
B_2 := \max_{x \ge L_1(t_0)} \frac{34\log^2x - 8E\log x + E^2/2}{x} \le 0.61184.
\end{equation}
Furthermore, 
\begin{align}
&c(R')\lambda\bigg(\frac{b}{b_0}\bigg\{(0.891 + 0.6079\nu)(L_1 - 3.377) + (0.7813 + 0.58\nu)L_2  \notag \\&\qquad\qquad\qquad\qquad + 7.329 + 3.898\nu\bigg\} + 1.8\bigg) \le B_3(t)\frac{L_1}{L_2}, \label{s4_H_second_term_bound}
\end{align}
where, since $(7.329 - 3.377 \cdot 0.891) b / b_0 + 1.8 < 18$ and $3.898 - 3.377 \cdot 0.6079 < 2$, we may take
\begin{equation} 
\begin{split}
B_3(t) &:= c(R')M_1\bigg(\frac{b}{b_0}\bigg\{0.891\frac{L_2^2}{L_1} + 0.6079B_2 + 0.7813\frac{L_2^3}{L_1^2}\\
&\qquad\qquad + 0.58B_2\frac{L_2}{L_1} + \frac{2 B_2}{L_1}\bigg\} + 18\frac{L_2^2}{L_1^2}\bigg).
\end{split}
\end{equation}
Since $B_3(t)$ is decreasing in $t$ for $t \ge t_0$, we have $B_3(t) \le B_3(t_0) \le 0.09245$. Therefore, collecting \eqref{s4_main_inequality_1}, \eqref{s4_H_smol_term_bound}, \eqref{s4_H_main_term_bound} and \eqref{s4_H_second_term_bound}, and from the definition of $M_1$, 
\begin{equation}
M = \lambda \frac{L_1}{L_2} \ge \frac{0.17996 - 0.7857/\log t_0}{0.00015 + 3.59852 + 0.09245} \ge 0.0475 > M_1,
\end{equation}
hence we have also obtained the desired contradiction for $t_0 \le t < t_1$. 

Therefore, we have shown that if $t \ge t_0$,
\begin{equation}
(1 - \beta)\frac{\log t}{\log\log t} = Z(\beta, t) \ge M \ge M_1 = 0.0470978 \implies 1 - \beta \le \frac{\log\log t}{21.233\log t},
\end{equation}
so that Corollary \ref{corollary1} is proved for $t \ge t_0$. 

\section{Conclusion and future work}\label{sec:concluding_remarks}

In this article we proved an explicit version of Littlewood's zero-free region by deriving an explicit $k$th derivative test using van der Corput's $A^{k - 2}B$ process. Our method represents a significant departure from recent approaches to proving classical zero-free regions. In this section we identify some potential refinements, which we have forgone for sake of brevity, but which may serve as basis for future research.

One avenue to improve Theorem \ref{theorem1} is to reduce the constants $A_k$ and $B_k$ appearing in the explicit $k$th derivative test (Lemma \ref{kth_deriv_test}). By specialising the argument to a specific value of $k$, we can use sharper variants of the inequalities used in the general argument of Lemma \ref{kth_deriv_test}. Alternatively, we can also achieve savings by specialising the choice of the phase function $f(x)$ earlier on (this was demonstrated in \cite{hiary_improved_2022} for the $AB$ process). Both approaches inevitably involve long and tedious arithmetic, however the iterative nature of the $k$th derivative test means it may be well suited for an automated or computer-assisted proof program.    

Theorem \ref{theorem1} is presented as a bound holding uniformly for all $k \ge 4$ and $t \ge 3$. However, the proof of Corollary \ref{corollary1} only requires a bound holding for $t \ge t_k$, where $t_k \to \infty$ as $k \to \infty$. Therefore, one way to improve Corollary \ref{corollary1} is to replace Theorem \ref{theorem1} with a bound of the form $|\zeta(\sigma_k + it)| \le a(k)t^{1/(2^k - 2)}\log t$ ($t \ge t_k$) with $a(k)$ a decreasing function of $k$. In particular, if $a(k) \to 0$ sufficiently quickly, then we can derive an asymptotically better bound while leaving the rest of the argument largely unchanged. However, we anticipate limited benefits of applying this method to solely improve the constant of Corollary \ref{corollary1}, since the bottleneck appears to be in small values of $k$, and thus $t$. 

Lastly, we also note that the proof of Corollary \ref{corollary1} relies on multiple results that can be sharpened using Theorem \ref{theorem1}. For instance, Lemma \ref{fordlem43} and \ref{fordlem43_1} primarily depend on an estimate of $S(t)$, which in turn depend on upper bounds on $\zeta(s)$ in the critical strip. This is readily obtained via Theorem \ref{theorem1} and the Phragm\'en--Lindel\"of Principle. Another example is Lemma \ref{fordlem42}, where the Ford--Richert bound \eqref{s4_richert_bound} is used, which for small $t$ can be improved by using a similar bound derived from Theorem \ref{theorem1}.

\section*{Acknowledgements}
I would like to thank my supervisor Timothy S. Trudgian for suggesting the initial idea for this project, and for various helpful ideas throughout the writing of this paper. Thanks also to the anonymous referee for useful feedback and suggestions.  
\clearpage
\printbibliography
\end{document}